\documentclass[3p, sort&compress]{elsarticle}
\usepackage[adobe-utopia]{mathdesign}
\usepackage[T1]{fontenc}
\usepackage[UKenglish]{babel}

\usepackage{hyperref,mathcomp,relsize,amsmath,xfrac,nicefrac,color,mathtools,etoolbox,enumitem,accents,amsthm,microtype}

\setlist{leftmargin=*}
\newtoggle{arxiv}
\toggletrue{arxiv}

\newenvironment{proofof}[1]{\par\noindent{\bfseries\upshape Proof of #1.\ }}{\qed}

\DeclarePairedDelimiter{\group}{(}{)}
\DeclarePairedDelimiter{\sqgroup}{[}{]}
\DeclarePairedDelimiter{\set}{\{}{\}}

\newcommand{\reals}{\mathbb{R}}
\newcommand{\extreals}{\overline{\mathbb{R}}}
\newcommand{\posreals}{\reals_{> 0}}
\newcommand{\nnegreals}{\reals_{\geq 0}}
\newcommand{\posextreals}{\extreals_{> 0}}
\newcommand{\nnegextreals}{\extreals_{\geq 0}}
\newcommand{\nats}{\mathbb{N}}
\newcommand{\natz}{\mathbb{N}_{0}}

\newcommand{\indica}[1]{\mathbb{I}_{#1}}

\newcommand{\upprev}{\overline{\mathrm{E}}}
\newcommand{\upprevacc}[1]{\overline{\mathrm{E}}_{\raisebox{1pt}{\scriptsize $#1$}}^{\raisebox{-2pt}{\scriptsize $\prime$}}}
\newcommand{\upprevvovkk}{\overline{\mathrm{E}}_\mathrm{V}}
\newcommand{\oldupprevvovk}{\overline{\mathrm{E}}_{\raisebox{1pt}{\scriptsize $\mathrm{V}$}}^{\raisebox{-2pt}{\scriptsize $\prime$}}}
\newcommand{\oldupprevvovkb}{\overline{\mathrm{E}}_{\raisebox{1pt}{\scriptsize $\mathrm{V}$}}^{\raisebox{-2pt}{\scriptsize $\prime\prime$}}}
\newcommand{\uboldupprevvovk}{\overline{\mathrm{E}}_{\raisebox{1pt}{\scriptsize $\mathrm{ub}$}}^{\raisebox{-2pt}{\scriptsize $\prime$}}}
\newcommand{\ubupprevvovk}{\overline{\mathrm{E}}_\mathrm{ub}}

\newcommand{\lowprev}{\underline{\mathrm{E}}}
\newcommand{\lowprevvovkk}{\underline{\mathrm{E}}_\mathrm{V}}

\newcommand{\uboldlowprevvovk}{\underline{\mathrm{E}}_{\raisebox{1pt}{\scriptsize $\mathrm{ub}$}}^{\raisebox{-2pt}{\scriptsize $\prime$}}}
\newcommand{\ublowprevvovk}{\underline{\mathrm{E}}_\mathrm{ub}}

\newcommand{\lupprev}[1]{\overline{\mathrm{Q}}_{#1}}

\newcommand{\sit}{x_{1:n}}
\newcommand{\situa}[1]{x_{#1}}
\newcommand{\situation}[2]{x_{#1:#2}}

\newcommand{\martingale}{\mathscr{M}}
\newcommand{\setofextsupmartb}{\overline{\mathbb{M}}_\mathrm{b}}
\newcommand{\setofextsupmartbb}{\overline{\mathbb{M}}_{\raisebox{1pt}{\scriptsize $\mathrm{b}$}}^{\raisebox{-2pt}{\scriptsize $\prime$}}}
\newcommand{\setofextsupmartbbb}{\overline{\mathbb{M}}_{\raisebox{1pt}{\scriptsize $\mathrm{b}$}}^{\raisebox{-2pt}{\scriptsize $\prime\prime$}}}
\newcommand{\setofubextsupmartb}{\overline{\mathbb{M}}}
\newcommand{\setofubextsupmartbb}{\overline{\mathbb{M}}^{\raisebox{-2pt}{\scriptsize $\prime$}}}

\newcommand{\process}{\mathscr{P}}

\newcommand{\situations}{{\mathscr{X}^{\ast}}}
\newcommand{\statespace}{\mathscr{X}}
\newcommand{\statespaceseq}[2]{\mathscr{X}_{#1:#2}}

\newcommand{\samplespace}{\Omega}

\newcommand{\setofgengambles}{\mathscr{L}}
\newcommand{\setofgenextvariables}{\overline{\mathscr{L}}}
\newcommand{\setofgenextvariablesb}{\overline{\mathscr{L}}_b}

\newcommand{\setofgambles}{\mathbb{V}}
\newcommand{\setofextvariables}{\overline{\mathbb{V}}}
\newcommand{\setofextvariablesb}{\overline{\mathbb{V}}_b}

\newcommand{\setoflimitsoffinmeasb}{\overline{\mathbb{V}}_{b,\mathrm{lim}}}

\newcommand{\prodindica}[2]{#1\,\indica{#2}}

\newcommand{\andstate}{\,\cdot}

\newtheorem{example}{Example}{\itshape}{\rmfamily}

\newtheorem{definition}{Definition}{\bfseries}{\rmfamily}
\newtheorem{theorem}{Theorem}{\bfseries}
\newtheorem{proposition}[theorem]{Proposition}{\bfseries}
\newtheorem{corollary}[theorem]{Corollary}{\bfseries}
\newtheorem{lemma}[theorem]{Lemma}{\bfseries}

\begin{document}

\title{Game-Theoretic Upper Expectations for Discrete-Time Finite-State Uncertain Processes}
\author{Natan T'Joens}
\ead{natan.tjoens@ugent.be}
\author{Jasper De Bock}
\author{Gert de Cooman}

\address{FLip, Ghent University, Belgium \\
Technologiepark-Zwijnaarde 125,
9052 Zwijnaarde}

\begin{abstract}
Game-theoretic upper expectations are joint (global) probability models that mathematically describe the behaviour of uncertain processes in terms of supermartingales; capital processes corresponding to available betting strategies.
Compared to (the more common) measure-theoretic expectation functionals, they are not bounded to restrictive assumptions such as measurability or precision, yet succeed in preserving, or even generalising many of their fundamental properties.
We focus on a discrete-time setting where local state spaces are finite and, in this specific context, build on the existing work of Shafer and Vovk, the main developers of the framework of game-theoretic upper expectations.
In a first part, we study Shafer and Vovk's characterisation of a local upper expectation and show how it is related to Walley's behavioural notion of coherence.
The second part consists in a study of game-theoretic upper expectations on a more global level, where several alternative definitions, as well as a broad range of properties are derived, e.g. the law of iterated upper expectations, compatibility with local models, coherence properties, \dots\
Our main contribution, however, concerns the continuity behaviour of these operators.
We prove continuity with respect to non-increasing sequences of so-called lower cuts and continuity with respect to non-increasing sequences of finitary functions.
We moreover show that the game-theoretic upper expectation is uniquely determined by its values on the domain of bounded below limits of finitary functions, and show in addition that, for any such limit, the limiting sequence can be constructed in such a way that the game-theoretic upper expectation is continuous with respect to this particular sequence.

\end{abstract}


\maketitle
\openup4pt

\section{Introduction}

We consider sequences $X_1$, $X_2$, \dots, $X_n$, \dots\ of uncertain states where the state $X_k$ at each discrete time $k\in\nats$ takes values in a fixed finite state space $\statespace{}$.
The uncertain evolution of the state $X_k$ in such a sequence---which we will call a discrete-time finite-state uncertain process---can be described mathematically in various ways.
Most authors prefer a measure-theoretic approach, which involves the use of countably additive probability measures.
In this paper, we consider an alternative, game-theoretic approach that uses (super)martingales as a primitive notion.
Probabilities and expectations then only appear as derived, secondary objects.
This approach was largely developed by Shafer and Vovk \cite{Shafer:2005wx,Vovk2019finance}, but some of the essential underlying ideas are due to Ville \cite{ville1939etudecritique}.


Irrespectively of the chosen mathematical framework, the starting point for modelling an uncertain process is typically a set of assessments about the local uncertain behaviour of the process.
Such local assessments represent our beliefs on how the state $X_k$ of the process will change from one time instant to the next.
In a traditional approach, they are given in the form of probabilities $\mathrm{P}(X_{k+1} = x_{k+1} \vert X_{1} = x_1 , \cdots,  X_{k} = x_{k})$ for each possible history $X_{1} = x_1 , \cdots,  X_{k} = x_{k}$ (with $x_i \in \statespace{}$) and for each possible value $x_{k+1} \in \statespace{}$ for the next state $X_{k+1}$.
However, when information is scarce or when we want to act conservatively, specifying such a single probability measure for each history might not be appropriate or not even justified.
In such cases, our beliefs can be more accurately described by imprecise probabilities models \cite{Augustin:2014di,Walley:1991vk}, e.g. probability intervals, sets of probability charges (or measures), sets of desirable gambles or upper and lower expectations.
Each of these models allows us to express and incorporate partial knowledge about the parameters that would normally make up a traditional precise probability model.
We will here focus on one particular---yet general and unifying---such imprecise probabilities model: upper and lower expectations (or previsions).

Upper and lower expectations generalise traditional expectation operators in the sense that the additivity property is replaced by the weaker condition of sub- or super-additivity, respectively.
A typical interpretation for them is that they are respective upper and lower bounds on a set of plausible expectations.
In fact, it can be shown that they are in a one-to-one relation with closed, convex sets of linear expectations \cite{Walley:1991vk} or, equivalently, closed convex sets of probability charges or measures.
Interval probabilities therefore only correspond to a special case.
Walley \cite{Walley:1991vk}, on the other hand, uses a behavioural interpretation where bounded real-valued functions on a non-empty set $\mathscr{Y}$ represent gambles with an uncertain payoff and where the upper (or lower) expectation $\upprev{}(f)$ (or $\lowprev{}(f)$) of such a gamble $f$ represents a subject's infimum selling price (or supremum buying price) for $f$.
By imposing that a subject should gamble rationally---which he calls `coherence'---Walley then obtains the same axioms as those of a sub- or superadditive expectation operator.
Walley's interpretation moreover naturally connects upper and lower expectations with another type of imprecise probabilities model, called `sets of desirable gambles', which models uncertainty by considering the gambles a subject finds desirable.
Due to its different interpretations and its connections with various other (imprecise) uncertainty models, the theory of upper and lower expectations has a unifying role within the field of imprecise probabilities, hence our choice for using them here to model uncertainty in stochastic processes.
More detailed information about lower and upper expectations can be found in References~\cite{Walley:1991vk,troffaes2014}.

So suppose that, for each possible history $X_{1} = x_1 , \cdots,  X_{k} = x_{k}$, we are given an upper expectation $\lupprev{x_{1:k}}$ that expresses our beliefs about the value of the next state $X_{k+1}$.
How can we use these assessments to draw conclusions about more general behaviour of the process?
For instance, what is the probability that the process state will ever be equal to a particular value $x \in \statespace{}$?
When do we expect this to happen for the first time?
Such inferences relate to the value of the process state at multiple time instances.
To make statements about them, we need a global uncertainty model that incorporates and extends the information included in the local uncertainty models $\lupprev{x_{1:k}}$.
One particularly interesting way of constructing such a global model was proposed by Shafer and Vovk \cite{Shafer:2005wx,Vovk2019finance}.

They picture a game that consists of a (possibly infinite) series of rounds and three players, called Forecaster, Skeptic and Reality.
In the beginning of each round, Forecaster declares how he is willing to gamble on the outcome of the current round.
Subsequently, Skeptic takes him up on his commitments and chooses a gamble from the ones offered by Forecaster.
At the end of the round, Reality decides the outcome and Forecaster and Skeptic exchange money accordingly.
Now, given Forecaster's assessments, how should Skeptic determine a selling (or buying) price for a gamble $f$ whose uncertain payoff depends on an entire realisation $\omega$---which may take an infinitely long time---of the game?
Shafer and Vovk argue that Skeptic should certainly agree on selling the gamble $f$ for any price $\alpha$ such that there is a gambling strategy for Skeptic that starts with an initial capital~$\alpha$ and allows him to end up with at least as much money as the payoff $f(\omega)$ corresponding to $f$, irrespective of the realisation $\omega$ of the game.
The infimum of all such prices $\alpha$, which depends entirely on Forecaster's local assessments, is then what Shafer and Vovk call the game-theoretic upper expectation of $f$.
Game-theoretic lower expectations are subsequently defined in an analogous, yet dual way.

In the setting that we consider, Forecaster's local assessments are modelled by the local upper expectations $\lupprev{x_{1:k}}$ and the allowed gambles for Skeptic are those functions $f$ on $\statespace{}$ for which $\lupprev{x_{1:k}}(f) \leq 0$.\footnote{Note that there is a close relation with Walley's interpretation for upper expectations here; we refer to \cite{deCooman:2008km} for an extensive study.}
The evolution of Skeptic's capital corresponding with an allowable betting strategy is called a supermartingale.
As sketched above, these supermartingales can be used to define (global) game-theoretic upper and lower expectations.
Mathematically speaking, these operators have the advantage of being very general, in the sense that they do not rely on an assumption of measurability or on an assumption that uncertainty should be modelled through traditional (precise) probability models.
Despite their generality, they also possess rather strong technical properties, and therefore maintain great practical usefulness.
Though many slightly different definitions of the game-theoretic upper expectation can be found throughout the literature, the key elements always remain the same.
We consider one particular definition here, since we believe it to have, in our setting, superior mathematical properties compared to other versions; we will argue so in Section~\ref{sect: Discussion}.
Moreover, the version we consider can also be elegantly characterised as the most conservative uncertainty model that satisfies a set of intuitive axioms, and its use can therefore also be motivated without relying on an interpretation in terms of supermartingales or other game-theoretic concepts  \cite{Tjoens2019NaturalExtensionISIPTA,TJOENS202130}.


The contribution of this paper is fourfold, yet serves the general aim of strengthening the theory and relevance of game-theoretic upper expectations for the case where the state space $\statespace{}$ is finite.

Our first contribution consists in bringing clarity about a number of properties that were already proved elsewhere for game-theoretic upper expectations, but usually in a slightly other setting or for a slightly other version of its definition.
This concerns in particular the more basic properties (e.g. compatibility with local models, law of iterated upper expectations, \dots) in Section~\ref{Sect: Basic properties of game}, the equivalent definitions given in Section~\ref{Sect: equivalent characterisations} and the non-decreasing continuity property in Section~\ref{Sect: Continuity}.
Since it is our aim to bring clarity concerning these properties, we give self-contained proofs for these results.
We will state explicitly when we borrow ideas from other work.

Our second contribution consists in showing that---for a finite state space $\statespace{}$---our version of the game-theoretic upper expectation also satisfies several new continuity properties.
Such properties are powerful mathematical tools that provide the resulting theory with elegance and also greatly enhance its practical scope.
The continuity of the Lebesgue integral, for instance, is one of the reasons why it is the integral of choice for computing expected values associated with a probability measure \cite{billingsley1995probability,shiryaev1995probability}.
We show in Section~\ref{Sect: Continuity} that game-theoretic upper expectations satisfy continuity with respect to non-decreasing sequences (as pointed out in the previous paragraph, this property is not entirely new), continuity with respect to non-increasing sequences of lower cuts and a version of Fatou's lemma.
In Section~\ref{Sect: Continuity wrt n-measurables}, we focus on continuity with respect to sequences of so-called finitary functions.
Such finitary functions only depend on the process state at a finite number of time instances.
They can be given a clear interpretation \cite{Tjoens2019NaturalExtensionISIPTA,TJOENS202130} and their upper expectations can often be calculated in a straightforward manner \cite{8629186,8535240}.
And although many practically relevant functions are themselves not finitary, most of them can still be written as the (pointwise) limit of a sequence of finitary functions \cite{8627473,8629186}, e.g.~hitting times, hitting probabilities, \dots\
We prove continuity with respect to non-increasing sequences of finitary functions, and show that for any bounded below (pointwise) limit of finitary functions, the limiting sequence can be constructed in such a way that the game-theoretic upper expectation is continuous with respect to this particular sequence.
Finally, we will also show that the game-theoretic upper expectation is uniquely determined by its values on the domain of all bounded below (pointwise) limits of finitary functions.

Our third contribution---which will be the topic of Section~\ref{Sect: Upper Expectations} and, to a small extent, also of Section~\ref{Sect: Continuity}---consists in bridging part of the gap that currently exists between what Shafer and Vovk do, and the work on upper and lower expectations (or previsions) in the field of imprecise probabilities.
The latter is often based on Walley's notion of coherence
which only considers upper (and lower) expectations on bounded real-valued functions.
In such a context, the connection with Shafer and Vovk's earlier work was already studied elaborately in \cite{deCooman:2008km}.
Here, we deal with extended real-valued functions and study how both approaches are related on a local level.
We first propose a slightly weaker---and therefore, more general---version of Shafer and Vovk's axioms for a local upper expectation and show that, for a finite state space, this weaker version can equivalently be obtained by combining coherence (on bounded real-valued functions) with an upward continuity axiom.
This result shows that the local upper expectations defined in this way are, on the domain of bounded below extended real-valued functions, uniquely determined by their values on bounded real-valued functions.
This allows us to justify the use of our local models (on the domain of bounded below functions)---and therefore, also the use of the global game-theoretic upper expectations that are derived from them---from a more conventional imprecise probabilities point of view \cite{TJOENS202130}.
Moreover, by the end of Section~\ref{Sect: Continuity}, it will become clear that a particular downward continuity axiom needs to be satisfied by the local upper expectations in order for them to be (completely) compatible with the global game-theoretic upper expectation.
This axiom will generally not be satisfied, even if we use the stronger axiomatisation for local upper expectations proposed by Shafer and Vovk; see Section~\ref{Sect: Upper Expectations}.
Hence, if we desire compatibility of local and global models, we will need to additionally impose the downward continuity axiom.

Our fourth and final contribution can be found in Section~\ref{sect: Discussion}, where we compare different definitions of the game-theoretic upper expectation and argue why we have chosen to work with the particular version considered here.
We take into account not only technical aspects, but also interpretational aspects of the possible definitions.
We feel the need to include such a discussion here, because, as already mentioned, many different definitions have been used in the literature and it is often unclear---especially for an unexperienced reader---how these different versions relate to each other.
We moreover show there that our particular definition of the game-theoretic upper expectation can be replaced by a more intuitive one, if we restrict ourselves to bounded real-valued functions.

In short, the outline of the paper is as follows:
Section~\ref{Sect: Upper Expectations} shows how upper expectations on a finite state space can be equivalently axiomatised, either in a way similar to what Shafer and Vovk propose, or using Walley's concept of coherence together with a continuity axiom.
In Section~\ref{Sect: Game-theoretic Upper Expectations}, we use the upper expectations of Section~\ref{Sect: Upper Expectations} as local uncertainty models in a process and show how these can then be used to define supermartingales and a global game-theoretic upper expectation.
Sections \ref{Sect: Basic properties of game}--\ref{Sect: Continuity wrt n-measurables} then focus on the mathematical properties of this global operator.
We conclude the article in Section~\ref{sect: Discussion} with a discussion of other possible versions of game-theoretic upper expectations.

\section{Upper Expectations}\label{Sect: Upper Expectations}

We consider and relate two possible approaches to characterising upper expectations on extended real-valued functions: one similar to Shafer and Vovk's axiomatisation, and one that uses Walley's notion of coherence in combination with an upward continuity axiom.
Since we will later on use these characterisations to define local upper expectations in uncertain processes with a finite state space, we will mainly focus on extended real-valued functions with finite domain here.
We show that, in this case, both characterisations are equivalent.
Moreover, it will turn out that the characterising axioms of an upper expectation are also satisfied by the global game-theoretic upper expectation that we will consider further on, which indeed allows us to call this global operator an upper expectation.

We start by introducing some preliminary notions.
We denote the set of all natural numbers, without $0$, by $\nats$, and let $\natz \coloneqq\nats \cup \{0\}$.
The set of extended real numbers is denoted by $\extreals \coloneqq\reals \cup \{+\infty, -\infty\}$.
The set of positive real numbers is denoted by $\posreals$, the set of non-negative real numbers by $\nnegreals$, and we also let $\posextreals \coloneqq \posreals \cup \{+\infty\}$ and $\nnegextreals \coloneqq \nnegreals \cup \{+\infty\}$.
We extend the total order relation $<$ on $\reals{}$ to $\extreals$ by positing that $-\infty<c<+\infty$ for all $c\in\reals{}$ and endow $\extreals{}$ with the associated order topology.

For any non-empty set $\mathscr{Y}$, a \emph{variable} $f$ on $\mathscr{Y}$ is a map on $\mathscr{Y}$.
A \emph{real(-valued) variable} is a variable taking values in the reals $\reals{}$, and similarly for an \emph{extended real(-valued) variable}.
We say that an extended real variable $f$ is \emph{bounded below} if there is an $M\in\reals{}$ such that $f\geq M$---meaning that $f(y) \geq M$ for all $y\in\mathscr{Y}$---and bounded above if $-f$ is bounded below.
A \emph{gamble} on $\mathscr{Y}$ is a real variable on $\mathscr{Y}$ that is \emph{bounded}, meaning that it is both bounded below and bounded above.
The set of all extended real variables on $\mathscr{Y}$ is denoted by $\setofgenextvariables(\mathscr{Y})$, the set of all bounded below extended real variables on $\mathscr{Y}$ by $\setofgenextvariablesb(\mathscr{Y})$ and the linear space of all gambles on $\mathscr{Y}$ by $\setofgengambles(\mathscr{Y})$.
For any $f\in\setofgenextvariables(\mathscr{Y})$ we use $\sup f$ and $\inf f$ to denote the supremum $\sup\{f(y) \colon y\in\mathscr{Y}\}$ and infimum $\inf\{f(y) \colon y\in\mathscr{Y}\}$ of the variable $f$, respectively.
We say that a sequence $\{f_n\}_{n\in\natz}$ in $\setofgenextvariables(\mathscr{Y})$ is \emph{uniformly bounded below} if there is an $M\in\reals{}$ such that $f_n \geq M$ for all $n\in\natz$.
For any sequence $\{f_n\}_{n\in\natz}$ in $\setofgenextvariables(\mathscr{Y})$, we write $\lim_{n\to+\infty} f_n$ to mean the pointwise limit of the functions $f_n$.
For a subset $A$ of $\mathscr{Y}$, we define the \emph{indicator} $\indica{A}$ of $A$ as the gamble on $\mathscr{Y}$ that assumes the value $1$ on $A$ and $0$ elsewhere.

\subsection{Upper expectations according to Shafer and Vovk}

The following definition of an upper expectation is very similar to what Shafer, Vovk and Takemura call a `superexpectation' in \cite{shafer2011levy} and an `upper expectation' in \cite{Vovk2019finance}.\footnote{Their definition is, as far as we know, not based on a single interpretation. Rather, they draw inspiration from various subfields in probability theory to obtain these axioms.}

\begin{definition}\label{def: upper expectation}
Consider any non-empty set $\mathscr{Y}$ and any $\setofgenextvariablesb(\mathscr{Y})\subseteq\mathscr{D}\subseteq\setofgenextvariables(\mathscr{Y})$.
Then we define an \emph{upper expectation $\upprev$ on $\mathscr{D}$} as an extended real-valued map on $\mathscr{D}$  that satisfies the following axioms:
\begin{enumerate}[leftmargin=*,ref={\upshape E\arabic*},label={\upshape E\arabic*}.,series=vovkcoherence]
\item \label{coherence: const is const}
$\upprev(c)=c$ for all $c\in\reals{}$;
\item \label{coherence: sublinearity}
$\upprev(f+g)\leq\upprev(f) + \upprev(g)$ for all $f,g\in\setofgenextvariablesb(\mathscr{Y})$;
\item \label{coherence: homog for positive lambda}
$\upprev(\lambda f)=\lambda \upprev(f)$ for all $\lambda\in\posreals$ and all $f\in\setofgenextvariablesb(\mathscr{Y})$;
\item \label{coherence: monotonicity}
$f \leq g \, \Rightarrow \, \upprev(f)\leq\upprev(g)$ for all $f, g\in\setofgenextvariablesb(\mathscr{Y})$.
\item \label{coherence: continuity wrt non-negative functions}
$\lim_{n\to+\infty} \upprev(f_n)=\upprev\left( \lim_{n\to+\infty} f_n \right)$ for any non-decreasing sequence $\{f_n\}_{n\in\natz}$ of non-negative variables in \/ $\setofgenextvariablesb(\mathscr{Y})$.
\end{enumerate}
\end{definition}
Alternatively, we can also consider the so-called \emph{conjugate lower expectation} $\lowprev$, defined by $\lowprev(f)\coloneqq-\upprev(-f)$ for all $f \in-\mathscr{D}\coloneqq\{f\in\setofgenextvariables{}(\mathscr{Y})\colon -f\in\mathscr{D}\}$.
It clearly suffices to focus on only one of the two functionals and we will work mainly with upper expectations.
Furthermore, in the definition above, as well as further on, we adopt the following conventions: $c +\infty=+\infty$ for all real $c$, $+\infty+\infty=+\infty$, $\lambda \, (+\infty)= +\infty$ for all $\lambda\in\posextreals$ and $0 \, (+\infty)=0$.
In addition to \ref{coherence: const is const}--\ref{coherence: continuity wrt non-negative functions}, we will sometimes also impose `continuity with respect to lower cuts' (also called `bounded-below support') on an upper expectation.
To introduce this property, let $f^{\vee c}$ for any $f \in \setofgenextvariables{}(\mathscr{Y})$ and any $c \in \reals{}$ be the variable defined by $f^{\vee c}(y) \coloneqq \max \{f(y), c \}$ for all $y \in \mathscr{Y}$.
An upper expectation $\upprev{}$ on $\mathscr{D}$ is then continuous with respect to lower cuts if
\begin{enumerate}[leftmargin=*,ref={\upshape E\arabic*},label={\upshape E\arabic*}., resume=vovkcoherence]
\item \label{ext coherence lower cuts continuity}
$\upprev(f) = \lim_{c\to \, -\infty} \upprev(f^{\vee c})$ for all $f \in \mathscr{D}$.
\end{enumerate}
The limit on the right hand side exists because $f^{\vee c} \in \setofgenextvariablesb{}(\mathscr{Y})$ is non-decreasing in $c\in\reals$ and $\upprev$ satisfies \ref{coherence: monotonicity} by assumption.
Axiom~\ref{ext coherence lower cuts continuity} will play a particularly important role further on, where we will study the compatibility of local and global game-theoretic upper expectations.
Moreover, we will also already use \ref{ext coherence lower cuts continuity} in this section, to establish an important relation between our definition of an upper expectation and the one used by Shafer and Vovk.

The definition of Shafer and Vovk (and Takemura) \cite[Section~6.1]{Vovk2019finance} only differs from Definition~\ref{def: upper expectation} with regard to the choice of domain and the variables for which axioms \ref{coherence: sublinearity}, \ref{coherence: homog for positive lambda} and \ref{coherence: monotonicity} should hold.
Their definition is stated for $\mathscr{D}=\setofgenextvariables{}(\mathscr{Y})$ and requires that \ref{coherence: sublinearity}--\ref{coherence: monotonicity} should hold on all of $\setofgenextvariables{}(\mathscr{Y})$:
\begin{enumerate}[leftmargin=*,ref={\upshape E\arabic*'},label={\upshape E\arabic*'}.,start=2]
\item \label{coherence: sublinearity extended}
$\upprev(f+g)\leq\upprev(f) + \upprev(g)$ for all $f,g\in\setofgenextvariables(\mathscr{Y})$;
\item \label{coherence: homog for positive lambda extended}
$\upprev(\lambda f)=\lambda \upprev(f)$ for all $\lambda\in\posreals$ and all $f\in\setofgenextvariables(\mathscr{Y})$;
\item \label{coherence: monotonicity extended}
$f \leq g \, \Rightarrow \, \upprev(f)\leq\upprev(g)$ for all $f, g\in\setofgenextvariables(\mathscr{Y})$,
\end{enumerate}
where they use the convention that $+\infty -\infty=-\infty+\infty=+\infty$.
This is a typical choice when working with upper expectations---also see \cite{DECOOMAN201618} where they use the dual convention for lower expectations---and we will henceforth use this convention without mentioning it explicitly.
So, for example, $a \geq b$ implies that $a-b \geq 0$, but not necessarily $0 \geq b-a$ for any two $a$ and $b$ in $\extreals$.
Moreover, we also adopt the conventions that $c -\infty=-\infty$ for all real $c$, $-\infty-\infty=-\infty$, $\lambda \, (-\infty)= -\infty$ and $(- \lambda) \, (+\infty)= -\infty$ for all $\lambda\in\posextreals$ and $0 \, (-\infty)=0$.

Now, note that axioms \ref{coherence: sublinearity}--\ref{coherence: monotonicity} are weaker than axioms \ref{coherence: sublinearity extended}--\ref{coherence: monotonicity extended}.
Indeed, any extended real-valued map $\overline{\mathrm{E}}$ on $\mathscr{D}=\overline{\mathscr{L}}(\mathscr{Y})$ that satisfies \ref{coherence: sublinearity extended}--\ref{coherence: monotonicity extended} automatically satisfies \ref{coherence: sublinearity}--\ref{coherence: monotonicity}, simply because $\overline{\mathscr{L}}_{\mathrm{b}}(\mathscr{Y}) \subset \overline{\mathscr{L}}(\mathscr{Y})$.
The converse is not necessarily true because \ref{coherence: sublinearity}--\ref{coherence: monotonicity} do not impose any conditions on the values of $\overline{\mathrm{E}}$ on the domain $\overline{\mathscr{L}}(\mathscr{Y})\setminus\overline{\mathscr{L}}_{\mathrm{b}}(\mathscr{Y})$.
Therefore, and since our domain $\mathscr{D}$ is not bound to be equal to $\overline{\mathscr{L}}(\mathscr{Y})$, our definition of an upper expectation is more general than Shafer and Vovk's.
However, if we choose $\mathscr{D}=\setofgenextvariables(\mathscr{Y})$ and, besides \ref{coherence: const is const}--\ref{coherence: continuity wrt non-negative functions}, additionally impose \ref{ext coherence lower cuts continuity} on our upper expectations, then we always obtain an upper expectation according to Shafer and Vovk:

\begin{proposition}\label{prop: local models with lower cuts satisfy vovk extended axioms}
Consider any non-empty set $\mathscr{Y}$ and any upper expectation $\upprev$ on $\setofgenextvariables(\mathscr{Y})$.
Then $\upprev$ satisfies \ref{coherence: sublinearity extended}--\ref{coherence: monotonicity extended} if it satisfies \ref{ext coherence lower cuts continuity}.
\end{proposition}
\begin{proof}
\ref{coherence: sublinearity extended}:
Consider any two $f,g \in \setofgenextvariables{}(\mathscr{Y})$ and any $c \in \reals{}$.
Then, since $f^{\vee c} \geq c$ and $g^{\vee c} \geq c$, we have that $f^{\vee c} + g^{\vee c} \geq 2c$.
In a similar way, we deduce that $f^{\vee c} + g^{\vee c} \geq f + g$.
Hence, combining both inequalities, we obtain that $f^{\vee c} + g^{\vee c} \geq \max\{f + g, 2c\} = (f+g)^{\vee 2c}$.
Moreover note that $f^{\vee c}$, $g^{\vee c}$ and $(f+g)^{\vee 2c}$ are all variables in $\setofgenextvariablesb{}(\mathscr{Y})$, so we can apply \ref{coherence: monotonicity} and subsequently \ref{coherence: sublinearity} to infer that
\begin{align*}
\upprev{}\left( (f+g)^{\vee 2c} \right)
\leq \upprev{}( f^{\vee c} + g^{\vee c} )
\leq \upprev{}(f^{\vee c}) + \upprev{}(g^{\vee c}).
\end{align*}
The inequality above holds for any $c \in \reals{}$, so we have that
\begin{align*}
\upprev{}(f+g)
\overset{\text{\ref{ext coherence lower cuts continuity}}}{=}
\lim_{c \to \, -\infty} \upprev{}\left( (f+g)^{\vee c} \right)
= \lim_{c \to \, -\infty} \upprev{}\left( (f+g)^{\vee 2c} \right)
\leq \lim_{c \to \, -\infty} \left[ \upprev{}(f^{\vee c}) + \upprev{}(g^{\vee c}) \right]
&= \lim_{c \to \, -\infty} \upprev{}(f^{\vee c}) + \lim_{c \to \, -\infty} \upprev{}(g^{\vee c}) \\
&\overset{\text{\ref{ext coherence lower cuts continuity}}}{=}
\upprev{}(f) + \upprev{}(g),
\end{align*}
where the existence of all the limits follows from the monotonicity [\ref{coherence: monotonicity}] of $\upprev{}$, and where the second to last equality follows from the fact that $\upprev{}(f^{\vee c})$ and $\upprev{}(g^{\vee c})$ are non-decreasing in $c$ and our convention that $+\infty-\infty=+\infty$.

\ref{coherence: homog for positive lambda extended}:
Consider any $f \in \setofgenextvariables{}(\mathscr{Y})$.
First note that, since multiplication with a positive constant $\lambda \in \posreals{}$ is order preserving on $\extreals{}$, we have that
\begin{equation*}
\max\{\lambda f(y) , c\} = \max\{\lambda f(y) , \lambda \tfrac{c}{\lambda}\} = \lambda \max\{f(y) , \tfrac{c}{\lambda}\} \text{ for all } y \in \mathscr{Y} \text{ and all } c \in \reals{}.
\end{equation*}
Hence, $(\lambda f)^{\vee c} = \lambda f^{\vee \, \sfrac{c}{\lambda}}$ for all $c \in \reals{}$ and all $\lambda \in \posreals{}$.
Since moreover $f^{\vee \, \sfrac{c}{\lambda}} \in \setofgenextvariablesb{}(\mathscr{Y})$, we can apply \ref{coherence: homog for positive lambda} to infer that
$\upprev{}\left( (\lambda f)^{\vee c} \right)
= \upprev{}\left(\lambda f^{\vee \, \sfrac{c}{\lambda}}\right)
= \lambda \upprev{}\left(f^{\vee \, \sfrac{c}{\lambda}}\right)$ for all $c \in \reals{}$ and all $\lambda \in \posreals{}$.
This then implies that
\begin{align*}
\upprev{}\left(\lambda f\right)
\overset{\text{\ref{ext coherence lower cuts continuity}}}{=}
\lim_{c \to \, -\infty} \upprev{}\left((\lambda f)^{\vee c}\right)
= \lim_{c \to \, -\infty} \lambda \upprev{}\left(f^{\vee \, \sfrac{c}{\lambda}}\right)
= \lambda \lim_{c \to \, -\infty} \upprev{}\left(f^{\vee \, \sfrac{c}{\lambda}}\right)
= \lambda \lim_{c \to \, -\infty} \upprev{}\left(f^{\vee c}\right)
\overset{\text{\ref{ext coherence lower cuts continuity}}}{=}
\lambda \upprev{}\left(f\right) \text{ for all } \lambda \in \posreals{}.
\end{align*}

\ref{coherence: monotonicity extended}:
This follows trivially from the monotonicity [\ref{coherence: monotonicity}] of $\upprev{}$ on $\setofgenextvariablesb{}(\mathscr{Y})$ in combination with \ref{ext coherence lower cuts continuity}.
\end{proof}

The following counterexample shows that the converse is not necessarily true and hence, that axioms~\ref{coherence: const is const}--\ref{ext coherence lower cuts continuity} define a strictly smaller set of upper expectation operators on $\setofgenextvariables{}(\mathscr{Y})$ compared to Shafer and Vovk's axioms.

\begin{example}\label{example: Vovk's axioms do not necessarily satisfy lower cuts continuity}
Consider any set $\mathscr{Y}$ such that \/ $\vert \mathscr{Y} \vert > 1$ and the extended real-valued map $\upprev{}\colon \setofgenextvariables{}(\mathscr{Y}) \to \extreals{}$ defined by
\begin{align*}
\upprev{}(f) \coloneqq
\begin{cases}
-\infty &\text{ if } f < +\infty \text{ pointwise and } f(y) = -\infty \text{ for some } y \in \mathscr{Y}; \\
\sup f &\text{ otherwise,}
\end{cases}
\end{align*}
for all $f\in\setofgenextvariables{}(\mathscr{Y})$.
We show that $\upprev{}$ satisfies \ref{coherence: const is const}, \ref{coherence: sublinearity extended}--\ref{coherence: monotonicity extended} and \ref{coherence: continuity wrt non-negative functions}, but not \ref{ext coherence lower cuts continuity}.

\ref{coherence: const is const}:
This follows trivially from the definition of \/ $\upprev{}$.

\ref{coherence: sublinearity extended}:
Consider any two $f,g \in \setofgenextvariables{}(\mathscr{Y})$.
If there is some $y \in \mathscr{Y}$ such that $f(y) = +\infty$, then we have that $\upprev{}(f)=+\infty$ and therefore also that $\upprev{}(f) + \upprev{}(g) = +\infty$, which implies the desired inequality.
Due to symmetry, the inequality is also satisfied if $g(y)=+\infty$ for some $y\in\mathscr{Y}$.
Hence, consider the case where both $f < +\infty$ and $g < +\infty$ pointwise.
Then we clearly also have that $f+g < +\infty$ pointwise.
If moreover $f(y)=-\infty$ for some $y \in\mathscr{Y}$, then also $f(y)+g(y)=-\infty$ (because $g(y)<+\infty$) which, together with the fact that $f+g < +\infty$ pointwise, implies that $\upprev{}(f+g) = -\infty$ and thus the desired inequality.
Once more, the same can be concluded if $g(y)=-\infty$ for some $y \in\mathscr{Y}$ because of symmetry.
Hence, we are left with the situation where both $f$ and $g$---and therefore also $f+g$---are real-valued.
Then we can immediately infer that $\upprev{}(f+g) = \sup (f+g) \leq \sup f + \sup g = \upprev{}(f) + \upprev{}(g)$.

\ref{coherence: homog for positive lambda extended}:
Consider any $\lambda \in\posreals{}$ and any $f\in\setofgenextvariables{}(\mathscr{Y})$.
If $f < +\infty$ pointwise and $f(y)=-\infty$ for some $y\in\mathscr{Y}$, then also $\lambda f < +\infty$ pointwise and $\lambda f(y)=-\infty$, which implies that $\lambda \upprev{}(f) = \lambda (-\infty) = -\infty = \upprev{}(\lambda f)$.
Otherwise, if $f > -\infty$ pointwise or $f(y) = +\infty$ for some $y\in\mathscr{Y}$, then also $\lambda f > -\infty$ pointwise or $\lambda f(y) = +\infty$, which implies that $\lambda \upprev{}(f) = \lambda \sup f = \sup \lambda f = \upprev{}(\lambda f)$.

\ref{coherence: monotonicity extended}:
Consider any $f,g\in \setofgenextvariables{}(\mathscr{Y})$ such that $f \leq g$.
If $f<+\infty$ pointwise and $f(y)=-\infty$ for some $y \in \mathscr{Y}$, then $\upprev{}(f)=-\infty$ and therefore automatically $\upprev{}(f)\leq\upprev{}(g)$.
Otherwise, if $f>-\infty$ pointwise or $f(y)=+\infty$ for some $y \in \mathscr{Y}$, then also $g>-\infty$ pointwise or $g(y)=+\infty$ for some $y \in \mathscr{Y}$.
Then it follows from the definition of \/ $\upprev{}$ that $\upprev{}(f) = \sup f \leq \sup g = \upprev{}(g)$.

\ref{coherence: continuity wrt non-negative functions}:
Consider any non-decreasing sequence $\{f_n\}_{n\in\natz}$ of non-negative variables in $\setofgenextvariables{}(\mathscr{Y})$.
Since $f_n\geq 0 >-\infty$ pointwise, we have that $\upprev{}(f_n)=\sup f_n$ for all $n\in\natz$.
Clearly, $f\coloneqq \lim_{n \to +\infty} f_n$ is non-negative too, so we also have that $\upprev(f) = \sup f$.
Hence, we infer that
\begin{align*}
\lim_{n\to+\infty} \upprev(f_n)
= \lim_{n\to+\infty} \sup_{y\in\mathscr{Y}} f_n(y)
= \sup_{n\in\natz} \sup_{y\in\mathscr{Y}} f_n(y)
= \sup_{y\in\mathscr{Y}} \sup_{n\in\natz} f_n(y)
= \sup_{y\in\mathscr{Y}} f(y)
= \upprev(f),
\end{align*}
where the second and the fourth equality follows from the non-decreasing character of $\{f_n\}_{n\in\natz}$.

So $\upprev$ is an upper expectation on $\setofgenextvariables{}(\mathscr{Y})$ that moreover satisfies the additional axioms \ref{coherence: sublinearity extended}--\ref{coherence: monotonicity extended}.
However, it is easy to see that it does not satisfy \ref{ext coherence lower cuts continuity}.
Indeed, consider the extended real variable $-\infty \indica{y}$ where $y \in \mathscr{Y}$.
Then we have that $\upprev(-\infty \indica{y}) = -\infty$.
On the other hand, $\upprev( (-\infty\indica{y})^{\vee c}) = \upprev(c\indica{y}) = 0$ for all non-positive $c\in\reals{}$ (indeed, note that $\sup c\indica{y} = 0$ because $\vert\mathscr{Y}\vert > 1$).
So $\upprev(-\infty \indica{y})=+\infty\not=0=\lim_{c\to-\infty} \upprev( (-\infty\indica{y})^{\vee c})$, which implies that \/ $\upprev$ does not satisfy \ref{ext coherence lower cuts continuity}.
\hfill$\Diamond$
\end{example}

Now, our reason for establishing Proposition~\ref{prop: local models with lower cuts satisfy vovk extended axioms} is not merely because of the result itself, but rather because it will allow us to draw the following essential conclusion:
\emph{as far as the global game-theoretic upper expectation---our main object of interest in this paper---is concerned, it does not matter what $\mathscr{D}$ is and whether we do, or do not additionally impose \/ \ref{coherence: sublinearity extended}--\ref{coherence: monotonicity extended} on the local upper expectations.}
In order to see this, it suffices for now to know that the global game-theoretic upper expectation---which will only be introduced later on in Section~\ref{Sect: Game-theoretic Upper Expectations}---will only depend on the values of our local upper expectations on the (sub)domain $\setofgenextvariablesb{}(\mathscr{Y})$ (these local upper expectations are simply upper expectations on $\mathscr{D}$, with $\setofgenextvariablesb{}(\mathscr{Y})\subseteq\mathscr{D}\subseteq\setofgenextvariables{}(\mathscr{Y})$, where $\mathscr{Y}$ is the local state space $\statespace{}$ of the considered process).
So our statement follows if we can show that letting $\mathscr{D}=\setofgenextvariables{}(\mathscr{Y})$ and imposing axioms~\ref{coherence: sublinearity extended}--\ref{coherence: monotonicity extended} on (local) upper expectations, does not restrict the possible values that these upper expectations can take on the subdomain $\setofgenextvariablesb{}(\mathscr{Y})$.
To show this, we consider for any upper expectation $\upprev{}$ on $\mathscr{D}$, with $\setofgenextvariablesb{}(\mathscr{Y})\subseteq\mathscr{D}\subseteq\setofgenextvariables{}(\mathscr{Y})$, the map $\upprev{}_{\textnormal{\tiny E6}} \colon \setofgenextvariables{}(\mathscr{Y})\to\extreals{}$ defined by
\begin{equation*}
\upprev{}_{\textnormal{\tiny E6}}(f)\coloneqq \lim_{c\to -\infty} \upprev{} (f^{\vee c}) \text{ for all } f\in\setofgenextvariables{}(\mathscr{Y}),
\end{equation*}
where the right hand side is well-defined because $f^{\vee c}\in\setofgenextvariablesb{}(\mathscr{Y})$ for any $c\in\reals$ and $\upprev{}$ satisfies \ref{coherence: monotonicity} (therefore implying the existence of the limit).
Then it is quite easy to see that $\upprev{}_{\textnormal{\tiny E6}}$ coincides with $\upprev{}$ on $\setofgenextvariablesb{}(\mathscr{Y})$ and that it is an upper expectation that furthermore satisfies \ref{ext coherence lower cuts continuity} and therefore, by Proposition~\ref{prop: local models with lower cuts satisfy vovk extended axioms}, also satisfies \ref{coherence: sublinearity extended}--\ref{coherence: monotonicity extended}.

\begin{proposition}\label{prop: extension with E6 is an extension}
Consider any non-empty set $\mathscr{Y}$ and any $\setofgenextvariablesb(\mathscr{Y})\subseteq\mathscr{D}\subseteq\setofgenextvariables(\mathscr{Y})$.
For any upper expectation $\upprev{}$ on $\mathscr{D}$, the associated map \/ $\upprev{}_{\textnormal{\tiny E6}}$ is an upper expectation on $\setofgenextvariables{}(\mathscr{Y})$ that coincides with $\upprev{}$ on $\setofgenextvariablesb{}(\mathscr{Y})$---and hence extends $\upprev{}$ if $\mathscr{D}=\setofgenextvariablesb{}(\mathscr{Y})$---and that moreover satisfies~\ref{ext coherence lower cuts continuity} and \ref{coherence: sublinearity extended}--\ref{coherence: monotonicity extended}.
\end{proposition}
\begin{proof}
The map $\upprev{}_{\textnormal{\tiny E6}}$ coincides with $\upprev{}$ on $\setofgenextvariablesb{}(\mathscr{Y})$---and hence extends $\upprev{}$ if $\mathscr{D}=\setofgenextvariablesb{}(\mathscr{Y})$---because, for any $f\in\setofgenextvariablesb{}(\mathscr{Y})$, we have that $f^{\vee c} = f$ for all $c\in\reals{}$ such that $c\leq\inf f$, and hence, $\upprev{}_{\textnormal{\tiny E6}}(f) = \lim_{c\to -\infty} \upprev{} (f^{\vee c}) = \upprev{} (f)$.
Since $\upprev{}_{\textnormal{\tiny E6}}$ coincides with $\upprev{}$ on $\setofgenextvariablesb{}(\mathscr{Y})$ and since $\upprev{}$ is an upper expectation, it is now immediate that $\upprev{}_{\textnormal{\tiny E6}}$ is an upper expectation as well.
Furthermore, again since $\upprev{}_{\textnormal{\tiny E6}}$ coincides with $\upprev{}$ on $\setofgenextvariablesb{}(\mathscr{Y})$, it follows from its definition that $\upprev{}_{\textnormal{\tiny E6}}$ satisfies \ref{ext coherence lower cuts continuity} and hence, due to Proposition~\ref{prop: local models with lower cuts satisfy vovk extended axioms}, that it also satisfies \ref{coherence: sublinearity extended}--\ref{coherence: monotonicity extended}.
\end{proof}

This result,
together with the fact that our global game-theoretic upper expectations will only depend on the restriction of our local models to $\setofgenextvariablesb{}(\mathscr{Y})$, indeed confirms our statement above.
For this reason, and since letting $\mathscr{D}=\setofgenextvariables{}(\mathscr{Y})$ and imposing axioms~\ref{coherence: sublinearity extended}--\ref{coherence: monotonicity extended} would only reduce the generality of what we do, we have chosen to only require that $\setofgenextvariablesb(\mathscr{Y})\subseteq\mathscr{D}\subseteq\setofgenextvariables(\mathscr{Y})$ and to adopt \ref{coherence: const is const}--\ref{coherence: continuity wrt non-negative functions} as our axioms for an upper expectation.
In some cases, however, when we desire compatibility of local and global models, we will let $\mathscr{D}=\setofgenextvariables{}(\mathscr{Y})$ and impose \ref{ext coherence lower cuts continuity} on our local upper expectations, which by Proposition~\ref{prop: local models with lower cuts satisfy vovk extended axioms} and Example~\ref{example: Vovk's axioms do not necessarily satisfy lower cuts continuity} is actually stronger than imposing \ref{coherence: sublinearity extended}--\ref{coherence: monotonicity extended}.

\subsection{Some Basic Properties of Upper Expectations}
We now continue this section with a second part in which we establish various properties of upper expectations that will prove convenient for the rest of the paper, or, as is for example the case for Proposition~\ref{Lemma: local continuity wrt non-decreasing seq} below, are just interesting in their own right.
\begin{proposition}\label{prop: coherence properties}
Consider any non-empty set $\mathscr{Y}$, any $\setofgenextvariablesb(\mathscr{Y})\subseteq\mathscr{D}\subseteq\setofgenextvariables(\mathscr{Y})$ and any extended real-valued map $\upprev$ on $\mathscr{D}$ that satisfies \ref{coherence: const is const}--\ref{coherence: monotonicity}, and let \/ $\lowprev{}$ be the conjugate map of \/ $\upprev{}$ defined on $-\mathscr{D}$.
Then we have that
\begin{enumerate}[leftmargin=*,labelwidth=*,ref={\upshape E\arabic*},label={\upshape E\arabic*}., resume=vovkcoherence]
\item \label{coherence: bounds}
$-\infty <\inf f\leq\upprev(f)\leq\sup f$ for all $f\in\setofgenextvariablesb(\mathscr{Y})$;
\item \label{coherence: const add}
$\upprev(f + \mu)=\upprev(f) + \mu$ for all $\mu\in\reals{} \cup \{+\infty\}$ and all $f\in\setofgenextvariablesb(\mathscr{Y})$;
\item \label{coherence: homog for non-negative lambda}
$\upprev(\lambda f)=\lambda \upprev(f)$ for all $\lambda\in\nnegreals{}$ and all $f\in\setofgenextvariablesb(\mathscr{Y})$.
\item \label{coherence: low smaller than up}
$\lowprev(f + g)\leq\upprev(f) + \lowprev(g)\leq\upprev(f+g)$ for all $f,g\in\setofgengambles(\mathscr{Y})$;
\item \label{coherence: uniform convergence}
if\/ $\lim_{n\to+\infty} \sup{\vert f-f_n \vert}=0$ then $\lim_{n\to+\infty} \vert \upprev(f)-\upprev(f_n) \vert=0$
for any sequence $\{f_n\}_{n\in\natz}$ in $\setofgengambles(\mathscr{Y})$.
\end{enumerate}
If \/ $\upprev{}$ moreover satisfies \ref{coherence: continuity wrt non-negative functions}, then we also have that
\begin{enumerate}[leftmargin=*,labelwidth=*,ref={\upshape E\arabic*},label={\upshape E\arabic*}., resume=vovkcoherence]
\item \label{coherence: continuity wrt bounded below functions}
$\lim_{n\to+\infty} \upprev(f_n)=\upprev\left( \lim_{n\to+\infty} f_n \right)$ for any non-decreasing sequence $\{f_n\}_{n\in\natz}$ in $\setofgenextvariablesb(\mathscr{Y})$.
\end{enumerate}
\end{proposition}
\begin{proof}
\ref{coherence: bounds}:
Consider any $f\in\setofgenextvariablesb(\mathscr{Y})$.
If $\sup f=+\infty$, we trivially have that $\upprev(f)\leq\sup f$.
If $\sup f$ is real, it follows immediately from \ref{coherence: monotonicity} that $\upprev(f)\leq\upprev(\sup f)$ and therefore that $\upprev(f)\leq\sup f$ because of \ref{coherence: const is const}.
That $\sup f=-\infty$, is impossible because $f$ is bounded below.
To see that $-\infty <\inf f\leq\upprev(f)$, note that $\inf f$ is real or equal to $+\infty$ [because $f$ is bounded below] and therefore that $-\infty <\inf f$ is automatically satisfied.
Moreover, for any real $\alpha <\inf f$ we clearly have that $\alpha < f$, implying by \ref{coherence: monotonicity} and \ref{coherence: const is const} that $\alpha < \upprev(f)$.
Since this holds for any $\alpha <\inf f$ we indeed have that $\inf f\leq\upprev(f)$.

\ref{coherence: const add}:
That $\upprev(f + \mu)\leq\upprev(f) + \mu$ for all real $\mu$ and all $f\in\setofgenextvariablesb(\mathscr{Y})$, follows directly from \ref{coherence: sublinearity} and \ref{coherence: const is const}.
The other inequality follows from the fact that
\begin{align*}
\upprev(f)=\upprev(f +\mu-\mu) \overset{\text{\ref{coherence: sublinearity}}}{\leq} \upprev(f + \mu) + \upprev(- \mu) \overset{\text{\ref{coherence: const is const}}}{=} \upprev(f + \mu)-\mu,
\end{align*}
 for all $f\in\setofgenextvariablesb(\mathscr{Y})$ and all real $\mu$.
 If $\mu=+\infty$, then since $\upprev(f) > -\infty$ [because of \ref{coherence: bounds}], it remains to check that $\upprev(+\infty)=+\infty$, which follows trivially from \ref{coherence: bounds}.

 \ref{coherence: homog for non-negative lambda}:
If $\lambda=0$, we have to show that $\upprev(0)=0$, which follows immediately from \ref{coherence: const is const}.
Otherwise, if $\lambda \in \posreals{}$, the equality follows from \ref{coherence: homog for positive lambda}.

 \ref{coherence: low smaller than up}:
 For all $f,g\in\setofgengambles(\mathscr{Y})$ we have that
 \begin{align*}
 \upprev(f)=\upprev(f + g-g) \overset{\text{\ref{coherence: sublinearity}}}{\leq} \upprev(f + g) + \upprev(- g)=\upprev(f + g)-\lowprev(g),
 \end{align*}
 where the last step follows from the definition of $\lowprev$.
 Hence, because $- \lowprev(g) = \upprev(-g)$ is real by \ref{coherence: bounds} [$g$ is a gamble], we have that $\upprev(f) + \lowprev(g)\leq\upprev(f + g)$ for all $f,g\in\setofgengambles(\mathscr{Y})$.
 The remaining inequality then follows immediately from conjugacy.
 Indeed, for any $f,g\in\setofgengambles(\mathscr{Y})$, we already have that $\upprev(-g) + \lowprev(-f)\leq\upprev(-g + (-f))$.
 By conjugacy, this implies that $-\lowprev(g)-\upprev(f)\leq-\lowprev(g + f)$ and therefore that $\upprev(f)+\lowprev(g)\geq\lowprev(f + g)$ for any $f,g\in\setofgengambles(\mathscr{Y})$.

\ref{coherence: uniform convergence}:
It is easy to see that, if $\lim_{n\to+\infty} \sup{\vert f-f_n \vert}=0$ for some sequence $\{f_n\}_{n\in\natz}$ of gambles, then $f$ is also a gamble and so is each $f-f_n$.
Hence, it follows from \ref{coherence: low smaller than up} that
\begin{align}\label{Eq: proof coherence uniform convergence}
\lowprev(f-f_n)\leq\upprev(f) + \lowprev(-f_n)\leq\upprev(f-f_n) \text{ for all } n\in\natz.
\end{align}
If we now apply \ref{coherence: bounds} to $\upprev(f-f_n)$, and \ref{coherence: bounds} and conjugacy to $\lowprev(f-f_n)$, it follows from \eqref{Eq: proof coherence uniform convergence} that $\inf (f-f_n)\leq\upprev(f)+\lowprev(-f_n)\leq\sup (f-f_n)$ for all  $n\in\natz$.
Since moreover $\upprev(f)+\lowprev(-f_n) = \upprev(f)-\upprev(f_n)$ for all $n\in\natz$, we then have that $\lim_{n\to+\infty} \smash{\vert \upprev(f)-\upprev(f_n) \vert}=0$ due to the fact that $\lim_{n\to+\infty} \sup{\vert f-f_n \vert}=0$.

\ref{coherence: continuity wrt bounded below functions}:
Let $\{f_n\}_{n \in \natz{}}$ be any non-decreasing sequence in $\setofgenextvariablesb{}(\mathscr{Y})$.
Since $f_0$ is bounded below and $\{f_n\}_{n \in \natz{}}$ is non-decreasing, $\{f_n\}_{n \in \natz{}}$ is uniformly bounded below by some $M \in \reals{}$.
Therefore, $\{f_n - M\}_{n \in \natz{}}$ is a non-decreasing sequence of non-negative variables in $\setofgenextvariables{}(\mathscr{Y})$.
Hence, due to \ref{coherence: continuity wrt non-negative functions}, we have that $\lim_{n \to +\infty} \upprev{}(f_n - M) = \upprev{}(\lim_{n \to +\infty} f_n - M)$, which is equivalent to $\lim_{n \to +\infty} \upprev{}(f_n) = \upprev{}(\lim_{n \to +\infty} f_n)$ due to \ref{coherence: const add} [which is applicable because $\lim_{n \to +\infty} f_n$ and all $f_n$ are bounded below].
\end{proof}

As a consequence of their continuity with respect to non-decreasing sequences [\ref{coherence: continuity wrt non-negative functions}], upper expectations also satisfy the following countable sub-additivity property.

\begin{proposition}\label{proposition: countable subadditivity}
Consider any non-empty set $\mathscr{Y}$, any $\setofgenextvariablesb(\mathscr{Y})\subseteq\mathscr{D}\subseteq\setofgenextvariables(\mathscr{Y})$ and any upper expectation $\upprev$ on $\mathscr{D}$.
Then
$\upprev\left(\sum_{n\in\natz} f_n\right)\leq\sum_{n\in\natz} \upprev(f_n)$
for any sequence $\{f_n\}_{n\in\natz}$ of non-negative variables in $\setofgenextvariablesb(\mathscr{Y})$.
\end{proposition}
\begin{proof}
Consider the sequence $\{g_n\}_{n\in\natz}$ of non-negative variables defined by $g_n \coloneqq\sum_{i=0}^{n} f_i$ for all $n\in\natz$.
Then, $\{g_n\}_{n\in\natz}$ is non-decreasing because $\{f_n\}_{n\in\natz}$ is non-negative.
Moreover, it is clear that $\{g_n\}_{n\in\natz}$ converges pointwise to $\sum_{n\in\natz} f_n$.
Hence, we can apply \ref{coherence: continuity wrt non-negative functions} to find that
\begin{align*}
\upprev\left(\sum_{n\in\natz} f_n\right)
= \lim_{n\to+\infty} \upprev(g_n)
= \lim_{n\to+\infty} \upprev\left(\sum_{i=0}^{n} f_i\right)
\overset{\text{\ref{coherence: sublinearity}}}{\leq}  \lim_{n\to+\infty} \sum_{i=0}^{n} \upprev(f_i)
= \sum_{n\in\natz} \upprev(f_n),
\end{align*}
where the limit on the right hand side of the inequality exists because all $\upprev(f_n)$ are non-negative as a consequence of \ref{coherence: bounds}.
\end{proof}

Finally, the following proposition shows that for a \emph{finite} set $\mathscr{Y}$, the continuity axiom \ref{coherence: continuity wrt non-negative functions} can in fact be replaced by a much simpler property.

\begin{proposition}\label{Lemma: local continuity wrt non-decreasing seq}
For any finite non-empty set $\mathscr{Y}$ and any $\setofgenextvariablesb(\mathscr{Y})\subseteq\mathscr{D}\subseteq\setofgenextvariables(\mathscr{Y})$, an extended real-valued map $\upprev$ on $\mathscr{D}$ is an upper expectation if and only if it satisfies \ref{coherence: const is const}--\ref{coherence: monotonicity} and
\begin{enumerate}[leftmargin=*,labelwidth=*,ref={\upshape E\arabic*},label={\upshape E\arabic*}., resume=vovkcoherence]
\item \label{coherence: homogeneity wrt infty}
$(+\infty) \, \upprev(f)=\upprev \left( (+\infty) f \right)$ for all non-negative $f \in \setofgenextvariablesb(\mathscr{Y})$.
\end{enumerate}
\end{proposition}
\begin{proof}
We first prove the direct implication; that is, we consider any upper expectation $\upprev{}$ on $\mathscr{D}$ and show that it satisfies \ref{coherence: homogeneity wrt infty}.
Fix any non-negative $f\in\setofgenextvariablesb(\mathscr{Y})$
and observe that $\{n f\}_{n \in \nats}$ is a non-decreasing sequence in $\setofgenextvariables{}(\mathscr{Y})$ that converges pointwise to $(+\infty)f$ [because of the convention that $(+\infty) \, 0 = 0$].
Hence,
\begin{align*}
\upprev\left( (+\infty) f \right)
= \upprev \left( \lim_{n\to+\infty} n f \right)
\overset{\text{\ref{coherence: continuity wrt bounded below functions}}}{=} \lim_{n\to+\infty} \upprev(n f)
\overset{\text{\ref{coherence: homog for positive lambda}}}{=} \lim_{n\to+\infty} n \upprev(f)
= (+\infty) \upprev(f),
\end{align*}
where we once more used the convention that $(+\infty) \, 0 = 0$ for the last step, together with the fact that $\upprev(f)\geq0$ because of~\ref{coherence: bounds}.

To prove the converse implication, consider any extended real-valued map $\upprev{}$ on $\mathscr{D}$ that satisfies \ref{coherence: const is const}--\ref{coherence: monotonicity} and \ref{coherence: homogeneity wrt infty}.
Let $\{f_n\}_{n\in\natz}$ be a non-decreasing sequence of non-negative variables in $\setofgenextvariablesb(\mathscr{Y})$ and let $f \coloneqq\lim_{n\to+\infty} f_n$.
We show that $\lim_{n\to+\infty} \upprev{}(f_n) = \upprev{}(f)$.
Because $\{f_n\}_{n\in\natz}$ is non-decreasing, we have that $f_n \leq f_{n+1} \leq f$ for all $n\in\natz$.
Then it follows from \ref{coherence: monotonicity} [which we are allowed to use because $f$ and all $f_n$ are non-negative and therefore bounded below] that $\upprev(f_n)\leq\upprev(f_{n+1})\leq\upprev(f)$ for all $n\in\natz$.
Hence, $\lim_{n\to+\infty} \upprev(f_n)$ exists and $\lim_{n\to+\infty} \upprev(f_n)\leq\upprev(f)$.
To show that the converse inequality holds, let $A \coloneqq\{ y\in\mathscr{Y} \colon f(y)=+\infty \}$ and consider the following two cases.

If $\upprev(\indica{A})=0$, we have that
\begin{equation}\label{Eq: proof cont wrt nondecreasing seq 1}
\upprev(f)
= \upprev\big((+\infty)\indica{A} + \prodindica{f}{A^c} \big)
\overset{\text{\ref{coherence: sublinearity}}}{\leq} \upprev\big((+\infty)\indica{A}\big) + \upprev(\prodindica{f}{A^c})
\overset{\text{\ref{coherence: homogeneity wrt infty}}}{=} (+\infty) \upprev(\indica{A}) + \upprev(\prodindica{f}{A^c})
= \upprev(\prodindica{f}{A^c}).
\end{equation}
Because $\prodindica{f}{A^c}$ is real-valued [it cannot be $-\infty$ because it is non-negative] and $\mathscr{Y}$ is finite, $\prodindica{f}{A^c}$ is a gamble and $\{\prodindica{f_n}{A^c} \}_{n\in\natz}$ converges uniformly to $\prodindica{f}{A^c}$.
$\{\prodindica{f_n}{A^c} \}_{n\in\natz}$ is moreover also a sequence of gambles because its elements are non-negative and are bounded from above by the gamble $\prodindica{f}{A^c}$.
Hence, it follows from \ref{coherence: uniform convergence} that
\begin{align*}
\upprev(\prodindica{f}{A^c})
= \lim_{n\to+\infty} \upprev(\prodindica{f_n}{A^c})
\leq \lim_{n\to+\infty} \upprev(f_n),
\end{align*}
where we used the non-negativity of $\{f_n\}_{n\in\natz}$ and \ref{coherence: monotonicity} in the last step.
Together with Equation \eqref{Eq: proof cont wrt nondecreasing seq 1}, this then leads to the desired inequality.

If $\upprev(\indica{A}) \not= 0$, we have that $\upprev(\indica{A}) > 0$ because of \ref{coherence: bounds}.
Furthermore, all $f_n$ are non-negative, and therefore
\begin{align}\label{Eq: proof cont wrt nondecreasing seq 2}
\upprev(f_n)
\overset{\text{\ref{coherence: monotonicity}}}{\geq} \upprev(\prodindica{f_n}{A})
\overset{\text{\ref{coherence: monotonicity}}}{\geq} \upprev\group[\bigg]{\sqgroup[\Big]{\inf_{y\in A} f_n(y)}\,\indica{A}}
\overset{\text{\ref{coherence: homog for non-negative lambda},\ref{coherence: homogeneity wrt infty}}}{=} \sqgroup[\Big]{\inf_{y\in A} f_n(y)}  \upprev(\indica{A}) \text{ for all } n\in\natz.
\end{align}
Since $\{f_n\}_{n\in\natz}$ converges to $+\infty$ on $A$ and $A$ is moreover finite [since $\mathscr{Y}$ is], we have that $\lim_{n\to+\infty}\inf_{y\in A} f_n(y)=+\infty$.
This implies, together with $\upprev(\indica{A}) > 0$ and \eqref{Eq: proof cont wrt nondecreasing seq 2}, that $\lim_{n\to+\infty} \upprev(f_n)=+\infty$.
Hence, the desired inequality follows.
\end{proof}

\subsection{An Alternative Characterisation using Coherence}
The axioms in Definition~\ref{def: upper expectation} are rather abstract, particularly because the concept of infinity plays such a prominent role in them.
Walley \cite{Walley:1991vk} avoids this issue by only considering upper expectations---which he calls upper previsions---on gambles, that is, bounded real-valued variables.
This allows him to give upper expectations a clear behavioural interpretation in terms of attitudes towards gambling \cite{Walley:1991vk,Quaeghebeur:2012vw}.
Concretely, the upper expectation $\upprev(f)$ for any $f\in\setofgengambles(\mathscr{Y})$ then represents a subject's infimum selling price for the gamble $f$.
This interpretation in turn leads to a notion of rationality that he calls coherence.

\begin{definition}
Consider any non-empty set $\mathscr{Y}$ and any map $\upprev$ on the linear space $\setofgengambles(\mathscr{Y})$ of all gambles on $\mathscr{Y}$.
Then $\upprev$ is called a \emph{coherent upper prevision} if it is real-valued and satisfies the following three \emph{coherence axioms} \cite[Definition 2.3.3]{Walley:1991vk}:
\begin{enumerate}[leftmargin=*,ref={\upshape C\arabic*},label={\upshape C\arabic*}., series=sepcoherence ]
\item \label{sep coherence 1}
$\upprev(f)\leq\sup f$ for all $f\in\setofgengambles(\mathscr{Y})$;
\item \label{sep coherence 2}
$\upprev(f+g)\leq\upprev(f) + \upprev(g)$ for all $f,g\in\setofgengambles(\mathscr{Y})$;
\item \label{sep coherence 3}
$\upprev(\lambda f)=\lambda \upprev(f)$ for all $\lambda\in\posreals{}$ and $f\in\setofgengambles(\mathscr{Y})$.
\end{enumerate}
\end{definition}

One can easily show \cite[Section~2.6.1]{Walley:1991vk} that the coherence axioms \ref{sep coherence 1}--\ref{sep coherence 3} imply the following additional properties, with $\lowprev(f) \coloneqq -\upprev(-f)$ for all $f\in\setofgengambles(\mathscr{Y})$:
\begin{enumerate}[leftmargin=*,ref={\upshape C\arabic*},label=\normalfont{\upshape C\arabic*}.,resume=sepcoherence]
\item \label{sep coherence 4}
$f \leq g \, \Rightarrow \, \upprev(f)\leq\upprev(g)$ for all $f, g\in\setofgengambles(\mathscr{Y})$;
\item \label{sep coherence 5}
$\inf f\leq\lowprev(f)\leq\upprev(f)\leq\sup f$ for all $f\in\setofgengambles(\mathscr{Y})$;
\item \label{sep coherence 6}
$\upprev(f + \mu)=\upprev(f) + \mu$ for all real $\mu$ and all $f\in\setofgengambles(\mathscr{Y})$;
\item \label{sep coherence 7}
$\lim_{n\to+\infty} \sup{\vert f-f_n \vert}=0 \Rightarrow \lim_{n\to+\infty} \vert \upprev(f)-\upprev(f_n) \vert=0$ for any sequence $\{f_n\}_{n\in\natz}$ in $\setofgengambles(\mathscr{Y})$.
\end{enumerate}
Now, in order to use coherent upper previsions as local uncertainty models in a game-theoretic framework, we need to extend their domain to at least the set $\setofgenextvariablesb{}(\mathscr{Y})$ of all bounded below extended real-valued variables, and in some cases to all of $\setofgenextvariables{}(\mathscr{Y})$.
We propose the following step-wise approach.

Consider any non-empty set $\mathscr{Y}$ and any $\setofgengambles{}(\mathscr{Y})\subseteq\mathscr{D}\subseteq\setofgenextvariables{}(\mathscr{Y})$.
Furthermore, for any $f \in \setofgenextvariables{}(\mathscr{Y})$ and any $c \in \reals{}$, let $f^{\wedge c}$ be the variable defined by $f^{\wedge c}(y) \coloneqq \min \{f(y), c \}$ for all $y \in \mathscr{Y}$.
Consider now the following continuity property for an extended real-valued map $\upprev{}$ on $\mathscr{D}$ whose restriction to $\setofgengambles{}(\mathscr{Y})$ is a coherent upper prevision:
\begin{enumerate}[leftmargin=28pt,ref={\upshape E\arabic*},label={\upshape E\arabic*}., resume=vovkcoherence]
\item \label{ext coherence cuts continuity}
$\upprev(f) = \lim_{c\to+\infty} \upprev(f^{\wedge c})$ for all $f \in \mathscr{D}\cap\setofgenextvariablesb{}(\mathscr{Y})$.
\end{enumerate}
Axiom~\ref{ext coherence cuts continuity} is called `continuity with respect to upper cuts' (also called `bounded-above support').
The limit on the right hand side exists, because $f^{\wedge c} \in \setofgengambles{}(\mathscr{Y})$ is non-decreasing in $c \in \reals{}$ and $\upprev{}$---or rather, its restriction to $\setofgengambles{}(\mathscr{Y})$---satisfies \ref{sep coherence 4}.
Similar to what we did with \ref{ext coherence lower cuts continuity}, property~\ref{ext coherence cuts continuity} can be used to (uniquely) extend a coherent upper prevision on $\setofgengambles{}(\mathscr{Y})$ to $\setofgenextvariablesb{}(\mathscr{Y})$.
Indeed, let $\upprev{}$ be any coherent upper prevision on $\setofgengambles{}(\mathscr{Y})$ and let $\upprev{}_{\textnormal{\tiny E14}} \colon \setofgenextvariablesb{}(\mathscr{Y})\to\extreals{}$ be defined by
\begin{align*}
\upprev{}_{\textnormal{\tiny E14}}(f)\coloneqq \lim_{c\to+\infty} \upprev(f^{\wedge c}) \text{ for all } f\in\setofgenextvariablesb{}(\mathscr{Y}),
\end{align*}
where the right hand side is well-defined because $f^{\wedge c}\in\setofgengambles{}(\mathscr{Y})$ for all $c \in \reals{}$ and $\upprev{}$ satisfies \ref{sep coherence 4} [which implies the existence of the limit]. Then $\upprev{}_{\textnormal{\tiny E14}}$ is an extension of $\upprev$:

\begin{proposition}\label{prop: upprev_E14 is an extension}
Consider any non-empty set $\mathscr{Y}$ and let \/ $\upprev{}$ be any coherent upper prevision on $\setofgengambles{}(\mathscr{Y})$.
Then $\upprev{}_{\textnormal{\tiny E14}}$ is an extension of \/ $\upprev{}$ that satisfies \ref{ext coherence cuts continuity}.
\end{proposition}
\begin{proof}
$\upprev{}_{\textnormal{\tiny E14}}$ is an extension of $\upprev{}$ because, for any gamble $f\in\setofgengambles{}(\mathscr{Y})$, we have that $f^{\wedge c} = f$ for all $c\in\reals{}$ such that $c\geq\sup f$, and hence, $\upprev{}_{\textnormal{\tiny E14}}(f) = \lim_{c\to+\infty} \upprev{} (f^{\wedge c}) = \upprev{} (f)$. That $\upprev{}_{\textnormal{\tiny E14}}$ satisfies \ref{ext coherence cuts continuity} now follows immediately from its definition.
\end{proof}

Our next result shows, for finite $\mathscr{Y}$, that this extension $\upprev{}_{\textnormal{\tiny E14}}$ is an upper expectation on $\setofgenextvariablesb{}(\mathscr{Y})$ and, moreover, that the restriction to $\setofgenextvariablesb{}(\mathscr{Y})$ of any upper expectation is the extension $\upprev{}_{\textnormal{\tiny E14}}$ of some coherent upper prevision $\upprev{}$ on $\setofgengambles{}(\mathscr{Y})$.


\begin{proposition}\label{Prop: alt. characterisation upper exp.}
Consider any finite non-empty set $\mathscr{Y}$ and any extended real-valued map $\upprevacc{}$ on $\mathscr{D}$, with $\setofgenextvariablesb(\mathscr{Y})\subseteq\mathscr{D}\subseteq\setofgenextvariables(\mathscr{Y})$.
Then $\upprevacc{}$ is an upper expectation if and only if there is some coherent upper prevision $\upprev$ on $\setofgengambles(\mathscr{Y})$ such that \/ $\upprev{}_{\textnormal{\tiny E14}}$ coincides with $\upprevacc{}$ on $\setofgenextvariablesb(\mathscr{Y})$.
\end{proposition}

The proof of this result is based on the following two lemmas.

\begin{lemma}\label{lemma: cont. wrt cuts iff cont. wrt nondecreasing for coherent upper prevision}
Consider any finite non-empty set $\mathscr{Y}$, any $\setofgenextvariablesb(\mathscr{Y})\subseteq\mathscr{D}\subseteq\setofgenextvariables(\mathscr{Y})$ and any extended real-valued map $\upprev$ on $\mathscr{D}$ whose restriction to $\setofgengambles{}(\mathscr{Y})$ is a coherent upper prevision.
Then $\upprev{}$ satisfies \ref{coherence: continuity wrt bounded below functions} if and only if it satisfies \ref{ext coherence cuts continuity}.
\end{lemma}
\begin{proof}
In order to prove the direct implication, suppose that $\upprev{}$ satisfies \ref{coherence: continuity wrt bounded below functions} and consider any $f\in\setofgenextvariablesb{}(\mathscr{Y})=\mathscr{D}\cap\setofgenextvariablesb{}(\mathscr{Y})$ and any non-decreasing sequence $\{c_n\}_{n\in\natz{}}$ of reals such that $\lim_{n\to+\infty}c_n=+\infty$.
Then clearly $\{f^{\wedge c_n}\}_{n\in\natz}$ is non-decreasing in $\setofgenextvariablesb{}(\mathscr{Y})$ and converges to $f$.
Hence, \ref{coherence: continuity wrt bounded below functions} implies that $\lim_{n\to+\infty}\upprev(f^{\wedge c_n}) = \upprev(f)$.
Furthermore, for any $n,m\in\natz{}$ such that $m>n$, we clearly have that $f^{\wedge c_{n}} \leq f^{\wedge c} \leq f^{\wedge c_{m}}$ for all $c\in\reals$ such that $c_{n} \leq c \leq c_{m}$.
Due to \ref{sep coherence 4} [which we can apply because the restriction of $\upprev{}$ to $\setofgengambles{}(\mathscr{Y})$ is a coherent upper prevision and because $f^{\wedge c}$ is a gamble for all $c\in\reals{}$], this also implies that $\upprev(f^{\wedge c_{n}}) \leq \upprev(f^{\wedge c}) \leq \upprev(f^{\wedge c_{m}})$ for all $c\in\reals$ such that $c_{n} \leq c \leq c_{m}$.
Since this holds for any $n,m\in\natz{}$ such that $m>n$, it follows that $\lim_{c\to+\infty} \upprev(f^{\wedge c}) = \lim_{n\to+\infty}\upprev(f^{\wedge c_n}) = \upprev(f)$, where the last equality follows from our earlier considerations.

To see that the converse implication holds, suppose that $\upprev{}$ satisfies \ref{ext coherence cuts continuity} and fix any non-decreasing sequence $\{f_n\}_{n \in \natz}$ in $\setofgenextvariablesb{}(\mathscr{Y})$.
Let $f \coloneqq \lim_{n \to +\infty} f_n \in \setofgenextvariablesb{}(\mathscr{Y})$.
Then, for any $c\in\reals$, $\{f^{\wedge c}_n\}_{n \in \natz}$ is a non-decreasing sequence in $\setofgengambles{}(\mathscr{Y})$ that clearly converges pointwise to $f^{\wedge c} \in \setofgengambles{}(\mathscr{Y})$.
Moreover, since $f^{\wedge c}$ is a real-valued function on a finite set $\mathscr{Y}$, the sequence $\{f^{\wedge c}_n\}_{n \in \natz}$ converges uniformly to $f^{\wedge c}$.
Hence, we have that
\begin{align*}
\upprev{}(f)
\overset{\text{\ref{ext coherence cuts continuity}}}{=} \lim_{c \to +\infty} \upprev{}(f^{\wedge c})
\overset{\text{\ref{sep coherence 7}}}{=} \lim_{c \to +\infty} \lim_{n \to +\infty} \upprev{}(f^{\wedge c}_n)
\overset{\text{\ref{sep coherence 4}}}{=} \sup_{c \in \reals} \sup_{n \in \natz} \upprev{}(f^{\wedge c}_n)
= \sup_{n \in \natz} \sup_{c \in \reals} \upprev{}(f^{\wedge c}_n)
&\overset{\text{\ref{sep coherence 4}}}{=} \lim_{n \to +\infty} \sup_{c \in \reals} \upprev{}(f^{\wedge c}_n) \\
&\overset{\text{\ref{sep coherence 4}}}{=} \lim_{n \to +\infty} \lim_{c \to +\infty} \upprev{}(f^{\wedge c}_n)  \\
&\overset{\text{\ref{ext coherence cuts continuity}}}{=} \lim_{n \to +\infty} \upprev{}(f_n).
\tag*{\qedhere}
\end{align*}
\end{proof}

\begin{lemma}\label{lemma: C1-C3 + cont.=upper exp.}
Consider any finite non-empty set $\mathscr{Y}$, any $\setofgenextvariablesb(\mathscr{Y})\subseteq\mathscr{D}\subseteq\setofgenextvariables(\mathscr{Y})$ and any extended real-valued map $\upprev$ on $\mathscr{D}$. If the restriction of  \/ $\upprev$ to $\setofgengambles{}(\mathscr{Y})$ is a coherent upper prevision and \/ $\upprev{}$ satisfies \ref{coherence: continuity wrt bounded below functions}, then $\upprev$ is an upper expectation.
\end{lemma}
\begin{proof}
Assume that the restriction of  \/ $\upprev$ to $\setofgengambles{}(\mathscr{Y})$ is a coherent upper prevision and \/ $\upprev{}$ satisfies \ref{coherence: continuity wrt bounded below functions}.
That \ref{coherence: const is const} holds, follows immediately from \ref{sep coherence 5}.
To prove \ref{coherence: sublinearity}, consider any two $f,g\in\setofgenextvariablesb(\mathscr{Y})$ and the corresponding non-decreasing sequences $\smash{\{f^{\wedge n}\}_{n\in\natz}}$ and $\smash{\{g^{\wedge n}\}_{n\in\natz}}$ in $\setofgengambles(\mathscr{Y})$.
Then due to \ref{coherence: continuity wrt bounded below functions}, $\lim_{n\to+\infty} \upprev(f^{\wedge n})=\upprev(f)$ and $\lim_{n\to+\infty} \upprev(g^{\wedge n})=\smash{\upprev(g)}$.
Moreover, $\{f^{\wedge n} + g^{\wedge n}\}_{n\in\natz}$ is also a non-decreasing sequence in $\setofgengambles(\mathscr{Y})$ and clearly $\lim_{n\to+\infty} (f^{\wedge n} + g^{\wedge n})=f + g\in\setofgenextvariablesb{}(\mathscr{Y})$, which again implies by \ref{coherence: continuity wrt bounded below functions} that $\lim_{n\to+\infty} \upprev(f^{\wedge n} + g^{\wedge n})=\upprev(f + g)$.
All together, we have that
\begin{align*}
\upprev(f+g)=\lim_{n\to+\infty} \upprev(f^{\wedge n} + g^{\wedge n}) \overset{\text{\ref{sep coherence 2}}}{\leq} \lim_{n\to+\infty} \left[ \upprev(f^{\wedge n}) + \upprev(g^{\wedge n})\right]
= \upprev(f) + \upprev(g),
\end{align*}
which concludes the proof of \ref{coherence: sublinearity}.

Property~\ref{coherence: homog for positive lambda} can be proved in a similar way.
Consider any $\lambda \in \posreals{}$ and any $f\in\setofgenextvariablesb{}(\mathscr{Y})$.
Note that $\{\lambda f^{\wedge n}\}_{n \in \natz}$ is a non-decreasing sequence [because $\lambda > 0$] in $\setofgengambles{}(\mathscr{Y})$ that converges pointwise to $\lambda f\in\setofgenextvariablesb{}(\mathscr{Y})$.
Hence,
\begin{align*}
\upprev{}(\lambda f)
\overset{\text{\ref{coherence: continuity wrt bounded below functions}}}{=} \lim_{n \to +\infty} \upprev{}(\lambda f^{\wedge n})
\overset{\text{\ref{sep coherence 3}}}{=} \lim_{n \to +\infty} \lambda \upprev{}(f^{\wedge n})
\overset{\text{\ref{coherence: continuity wrt bounded below functions}}}{=} \lambda \upprev{}(f).
\end{align*}

To prove \ref{coherence: monotonicity}, consider any two $f,g\in\setofgenextvariablesb(\mathscr{Y})$ such that $f \leq g$, and the corresponding non-decreasing sequences $\smash{\{f^{\wedge n}\}_{n\in\natz}}$ and $\smash{\{g^{\wedge n}\}_{n\in\natz}}$ in $\setofgengambles(\mathscr{Y})$.
Clearly, $f^{\wedge n} \leq g^{\wedge n}$ for all $n\in\natz$ and therefore $\upprev(f^{\wedge n})\leq\upprev(g^{\wedge n})$ by \ref{sep coherence 4}.
Hence, $\lim_{n\to+\infty} \upprev(f^{\wedge n})\leq\lim_{n\to+\infty} \upprev(g^{\wedge n})$ and therefore, because of \ref{coherence: continuity wrt bounded below functions}, also $\upprev(f)\leq\upprev(g)$.
Finally, property~\ref{coherence: continuity wrt non-negative functions} follows immediately from \ref{coherence: continuity wrt bounded below functions}.
\end{proof}

\begin{proofof}{Proposition~\ref{Prop: alt. characterisation upper exp.}}
First assume that $\upprevacc{}$ is an upper expectation.
We then let $\upprev{}$ be the restriction of $\upprevacc{}$ to $\setofgengambles{}(\mathscr{Y})$. $\upprev{}$ trivially satisfies~\ref{sep coherence 1}--\ref{sep coherence 3} because $\upprevacc{}$ is an upper expectation, and $\upprev$ is real-valued because $\upprevacc{}$ satisfies \ref{coherence: bounds}.
Hence, $\upprev$ is a coherent upper prevision.
Moreover, since $\upprevacc{}$ satisfies \ref{coherence: continuity wrt bounded below functions} because of Proposition~\ref{prop: coherence properties}, Lemma~\ref{lemma: cont. wrt cuts iff cont. wrt nondecreasing for coherent upper prevision} implies that it also satisfies \ref{ext coherence cuts continuity}.
Hence,
\begin{equation*}
\upprevacc{}(f)
= \lim_{c\to+\infty} \upprevacc{}(f^{\wedge c})
= \lim_{c\to+\infty} \upprev{}(f^{\wedge c})
= \upprev{}_{\textnormal{\tiny E14}}(f) \text{ for all } f\in\setofgenextvariablesb{}(\mathscr{Y}),
\end{equation*}
where the second step follows from the fact that $\upprev{}$ is the restriction of $\upprevacc{}$ to $\setofgengambles{}(\mathscr{Y})$ and because $f^{\wedge c}\in\setofgengambles{}(\mathscr{Y})$ for all $c\in\reals{}$, and where the last step follows from the definition of $\upprev{}_{\textnormal{\tiny E14}}$.

For the converse implication, we assume that there is some coherent upper prevision $\upprev$ on $\setofgengambles(\mathscr{Y})$ such that $\upprev{}_{\textnormal{\tiny E14}}$ coincides with $\upprevacc{}$ on $\setofgenextvariablesb(\mathscr{Y})$. Now recall from Proposition~\ref{prop: upprev_E14 is an extension} that $\upprev{}_{\textnormal{\tiny E14}}$ is an extension of $\upprev{}$ that satisfies \ref{ext coherence cuts continuity}.
Since $\upprev{}_{\textnormal{\tiny E14}}$ coincides with $\upprevacc{}$ on $\setofgenextvariablesb(\mathscr{Y})$, this implies that $\upprevacc{}$ satisfies~\ref{ext coherence cuts continuity} and that the restriction of $\upprevacc{}$ to $\setofgengambles(\mathscr{Y})$ is equal to $\upprev$ and therefore a coherent upper prevision.
Then, by Lemma~\ref{lemma: cont. wrt cuts iff cont. wrt nondecreasing for coherent upper prevision}, $\upprevacc{}$ satisfies \ref{coherence: continuity wrt bounded below functions}, which by Lemma~\ref{lemma: C1-C3 + cont.=upper exp.} implies that $\upprevacc{}$ is an upper expectation.
\end{proofof}

\vspace*{0.8em}

An important consequence of this alternative characterisation for an upper expectation is that the restriction of an upper expectation to $\setofgenextvariablesb{}(\mathscr{Y})$ is uniquely determined by its values on the domain $\setofgengambles{}(\mathscr{Y})$ of all gambles.
This allows us to justify the use of upper expectations on $\setofgenextvariablesb{}(\mathscr{Y})$---and therefore the use of these as local uncertainty models further on---from a more conventional imprecise probabilities point of view, by combining coherence on gambles with~\ref{ext coherence cuts continuity}~\cite{TJOENS202130}.
The same cannot be said about the values of our upper expectations on $\setofgenextvariables{}(\mathscr{Y})\setminus\setofgenextvariablesb{}(\mathscr{Y})$.
If we also want these values to be uniquely determined by the values on the gambles, we can additionally impose \ref{ext coherence lower cuts continuity}.
This is made explicit in our next result, where we will use, for any coherent upper prevision $\upprev{}$ on $\setofgengambles{}(\mathscr{Y})$, the notation $\upprev{}_{\mathrm{ext}}$ to denote the map $[\upprev{}_{\textnormal{\tiny E14}}]_{\raisebox{-1pt}{\textnormal{\tiny E6}}}$ that extends $\upprev{}$ to $\setofgenextvariablesb{}(\mathscr{Y})$ using~\ref{ext coherence cuts continuity}, and subsequently to $\setofgenextvariables{}(\mathscr{Y})$ using~\ref{ext coherence lower cuts continuity}.
Note that $\upprev{}_{\mathrm{ext}}$ is well-defined because $\upprev{}_{\textnormal{\tiny E14}}$ is an upper expectation according to Proposition~\ref{Prop: alt. characterisation upper exp.}.
It should moreover be clear that $\upprev{}_{\mathrm{ext}}$ is also an extension of $\upprev$; this can be checked using Propositions~\ref{prop: upprev_E14 is an extension}, \ref{Prop: alt. characterisation upper exp.} and~\ref{prop: extension with E6 is an extension}.


\begin{proposition}\label{prop: upper expectation iff extended coherent upper prevision}
Consider any finite non-empty set $\mathscr{Y}$, any extended real-valued map \/ $\upprevacc{}$ on $\setofgenextvariables(\mathscr{Y})$ and let \/ $\upprev{}$ be the restriction of \/ $\upprevacc{}$ to $\setofgengambles{}(\mathscr{Y})$.
Then the following are equivalent:
\begin{enumerate}[leftmargin=*,ref=\upshape{\roman*},label={\upshape{\roman*.}},itemsep=3pt]
\item $\upprevacc{}$ is an upper expectation on $\setofgenextvariables{}(\mathscr{Y})$ that satisfies \ref{ext coherence lower cuts continuity}. \label{Eq: upper expectation iff extended coherent upper prevision i}
\item $\upprev{}$ is a coherent upper prevision such that \/ $\upprevacc{}=\upprev{}_{\textnormal{ext}}$. \label{Eq: upper expectation iff extended coherent upper prevision ii}
\item $\upprev{}$ is a coherent upper prevision and \/ $\upprevacc{}$ satisfies \ref{ext coherence lower cuts continuity} and \ref{ext coherence cuts continuity}. \label{Eq: upper expectation iff extended coherent upper prevision iii}
\end{enumerate}
\end{proposition}
\begin{proof}
\ref{Eq: upper expectation iff extended coherent upper prevision i} $\Rightarrow$ \ref{Eq: upper expectation iff extended coherent upper prevision ii}.
Assume that $\upprevacc{}$ is an upper expectation on $\setofgenextvariables{}(\mathscr{Y})$ that satisfies \ref{ext coherence lower cuts continuity}.
Then, according to Proposition~\ref{Prop: alt. characterisation upper exp.}, there is a coherent upper prevision $\upprev{}^{\raisebox{1pt}{\scriptsize$\ast$}}$ on $\setofgengambles{}(\mathscr{Y})$ such that $\upprevacc{}(f)=\upprev{}^{\raisebox{1pt}{\scriptsize$\ast$}}_{\raisebox{1pt}{\textnormal{\tiny E14}}}(f)$ for all $f\in\setofgenextvariablesb{}(\mathscr{Y})$.
Then, since $\upprev{}^{\raisebox{1pt}{\scriptsize$\ast$}}_{\raisebox{1pt}{\textnormal{\tiny E14}}}$ is an extension of $\upprev{}^{\raisebox{1pt}{\scriptsize$\ast$}}$ [due to Proposition~\ref{prop: upprev_E14 is an extension}] and since $\upprevacc{}$ is an extension of $\upprev{}$, we have that $\upprev{}^{\raisebox{1pt}{\scriptsize$\ast$}}=\upprev{}$. On the on hand, this already implies that $\upprev$ is indeed a coherent upper prevision. On the other hand, it implies that $\upprevacc{}(f)=\upprev{}^{\raisebox{1pt}{\scriptsize$\ast$}}_{\raisebox{1pt}{\textnormal{\tiny E14}}}(f)=\upprev{}_{\textnormal{\tiny E14}}(f)$ for all $f\in\setofgenextvariablesb{}(\mathscr{Y})$.
Note that $\upprev{}_{\textnormal{\tiny E14}}$ is therefore an upper expectation on $\setofgenextvariablesb{}(\mathscr{Y})$ [because $\upprevacc{}$ was assumed to be an upper expectation on $\setofgenextvariables{}(\mathscr{Y})$], which, due to Proposition~\ref{prop: extension with E6 is an extension}, implies that $\upprev{}_{\textnormal{ext}}$ is an extension of $\upprev{}_{\textnormal{\tiny E14}}$ satisfying \ref{ext coherence lower cuts continuity}.
Hence, we have that $\smash{\upprevacc{}(f)=\upprev{}_{\textnormal{\tiny E14}}(f)=\upprev{}_{\textnormal{ext}}(f)}$ for all $f\in\setofgenextvariablesb{}(\mathscr{Y})$.
Since both $\smash{\upprevacc{}}$ and $\upprev{}_{\textnormal{ext}}$ satisfy \ref{ext coherence lower cuts continuity}, it then also follows that $\smash{\upprevacc{}(f)=\lim_{c\to-\infty}\upprevacc{}(f^{\vee c})=\lim_{c\to-\infty}\upprev{}_{\textnormal{ext}}(f^{\vee c})=\upprev{}_{\textnormal{ext}}(f)}$ for all $f\in\setofgenextvariables{}(\mathscr{Y})$.

\ref{Eq: upper expectation iff extended coherent upper prevision ii} $\Rightarrow$ \ref{Eq: upper expectation iff extended coherent upper prevision iii}.
Assume that $\upprev{}$ is a coherent upper prevision such that $\upprevacc{}=\upprev{}_{\textnormal{ext}}$.
Then, due to Proposition~\ref{Prop: alt. characterisation upper exp.}, $\upprev{}_{\textnormal{\tiny E14}}$ is an upper expectation on $\setofgenextvariablesb{}(\mathscr{Y})$, which, by Proposition~\ref{prop: extension with E6 is an extension} implies that $\upprev{}_{\textnormal{ext}}$---and therefore, also $\upprevacc{}$---satisfies \ref{ext coherence lower cuts continuity}.
To see that $\upprevacc{}$ moreover satisfies \ref{ext coherence cuts continuity}, it suffices to recall that $\upprevacc{}=\upprev{}_{\textnormal{ext}}$ is an extension of $\upprev{}_{\textnormal{\tiny E14}}$ [because of Proposition~\ref{prop: extension with E6 is an extension}] and that $\upprev{}_{\textnormal{\tiny E14}}$ satisfies \ref{ext coherence cuts continuity} [because of Proposition~\ref{prop: upprev_E14 is an extension}].

\ref{Eq: upper expectation iff extended coherent upper prevision iii} $\Rightarrow$ \ref{Eq: upper expectation iff extended coherent upper prevision i}.
Assume that $\upprev{}$ is a coherent upper prevision and $\upprevacc{}$ satisfies \ref{ext coherence lower cuts continuity} and \ref{ext coherence cuts continuity}.
Then it suffices to prove that $\upprevacc{}$ is an upper expectation on $\setofgenextvariables{}(\mathscr{Y})$, which can readily be inferred by combining Lemma~\ref{lemma: cont. wrt cuts iff cont. wrt nondecreasing for coherent upper prevision} and Lemma~\ref{lemma: C1-C3 + cont.=upper exp.}.
\end{proof}

So if we assume our upper expectations to satisfy \ref{ext coherence lower cuts continuity}, we can justify their use on the entire domain $\setofgenextvariables{}(\mathscr{Y})$ by combining Walley's behavioural interpretation \cite{TJOENS202130} with \ref{ext coherence lower cuts continuity} and~\ref{ext coherence cuts continuity}.
For the local uncertainty models further on---which will simply be upper expectations on $\setofgenextvariables{}(\statespace{})$ with $\statespace{}$ a finite state space---, we would therefore typically be inclined to adopt this assumption, as we do in \cite{TJOENS202130}.
Moreover, it will turn out that imposing~\ref{ext coherence lower cuts continuity} on the local models is necessary if we want to guarantee compatibility of local and global models; see Section~\ref{Sect: Continuity}.
Nonetheless, we will generally not impose ~\ref{ext coherence lower cuts continuity} on our local models here, because the current paper focusses on the mathematical properties of global game-theoretic upper expectations, which---apart from the compatibility with local models--- are not affected by this additional continuity axiom.
It can easily be seen that this is a consequence of the fact that the global game-theoretic upper expectation will only depend on the restrictions of the local models to $\setofgenextvariablesb{}(\statespace{})$.

\section{Game-theoretic Upper Expectations}\label{Sect: Game-theoretic Upper Expectations}

When mathematically modelling an uncertain process, one typically starts out with information about its local behaviour, that is, how its state $X_k$ will evolve from one time instant to the next.
We will represent this information using so-called `local' upper expectations; upper expectations on $\setofgenextvariables{}(\statespace{})$ where $\statespace{}$ is the finite state space of the process.
Most often, one is interested in more general behaviour of the process, though, which confronts us with the question of how to combine the individual local assessments---represented by local upper expectations in our case---to obtain a single global uncertainty model.
We consider one possible way of doing so; using the game-theoretic approach proposed by Shafer and Vovk.
We will leave out most of the contextual discussion surrounding the following definitions.
For more details, we refer the interested reader to \cite{Shafer:2005wx,Vovk2019finance,TJOENS202130,deCooman:2008km}.

As explained in the introduction, we consider sequences $X_1, X_2, ..., X_n , ...$ of uncertain states that take values in a finite state space $\statespace{}$.
We call any finite string $\sit \coloneqq (x_1,...,x_n)\in\statespace_{1:n} \coloneqq\statespace^n$ of possible state values a \emph{situation} and we denote the set of all situations by $\statespace^\ast \coloneqq\cup_{n\in\natz} \mathscr{X}_{1:n}$.
In particular, the unique empty string $x_{1:0}$, denoted by $\Box$, is called the \emph{initial situation}, and $\statespace_{1:0} \coloneqq\{\Box\}$.
In order to model the local behaviour of an uncertain process, we attach to each situation $x_{1:n}\in\situations$ an upper expectation $\lupprev{x_{1:n}}$ on $\setofgenextvariables(\statespace)$.
Such an upper expectation $\lupprev{x_{1:n}}$---which we will call a \emph{local upper expectation}---represents a subject's beliefs about what the next state of the process will be, given that it was in the states $x_1 \cdots x_n$ at times $k=1$ through $k=n$.
For instance, if we adopt a behavioural interpretation, the upper expectation $\lupprev{x_{1:n}}(f)$ for some $f \in \setofgengambles{}(\statespace{})$ is the subject's infimum selling price for the gamble $f(X_{n+1})$ that takes the value $f(x)$ if $X_{n+1}=x$ for any $x \in \statespace{}$, given that he observed the history $X_1=x_1 \cdots X_n=x_n$.
However, the local upper expectations could equally well be interpreted in terms of upper envelopes of linear expectations---which may in turn come from a set of probability mass functions.
We do not enforce any interpretation for the local models $\lupprev{x_{1:n}}$; we simply assume that they are maps on $\setofgenextvariables{}(\statespace{})$ that satisfy \ref{coherence: const is const}--\ref{coherence: continuity wrt non-negative functions}.
A collection of local upper expectations $\lupprev{s}$, one for every $s \in \situations{}$, is called an \emph{imprecise probabilities tree}.
Note that an imprecise Markov chain under epistemic irrelevance \cite{deCooman:2009jz,HermansITIP,DECOOMAN201618}, for instance, corresponds to a special type of imprecise probabilities tree where, for any $n\in\nats{}$ and any $x_{1:n}\in\statespace{}^n$, the local model $\lupprev{x_{1:n}}$ does not depend on the previous $n-1$ states $x_{1:n-1}$, nor on the time point $n$.
So, in that case, we would have that $\lupprev{x_{1:n}} = \lupprev{x_{n}}$ for all $n\in\nats{}$ and all $x_{1:n}\in\statespace{}^n$.
We refer to \cite{8535240,DECOOMAN201618,8627473} for more details on how the game-theoretic framework presented here can be implemented in an imprecise Markov chain setting.

In order to describe uncertain processes on a more global level, we will use the notion of a \emph{path} $\omega$; an infinite sequence of state values.
The set of all paths is called the \emph{sample space} $\samplespace \coloneqq\statespace^\nats$.
For any path $\omega\in\samplespace$, the initial sequence that consists of its first $n$ state values is a situation in $\statespaceseq{1}{n}$ that is denoted by $\omega^n$.
The $n$-th state value is denoted by $\omega_n\in\statespace$.
A collection of paths $A \subseteq \samplespace$ is called an \emph{event}.
With any situation $\sit$, we associate the \emph{cylinder event} $\Gamma(\sit)\coloneqq\{\omega\in\samplespace\colon \omega^n=\sit\}$: the set of all paths $\omega\in\samplespace$ that `go through' the situation $\sit$.
Sometimes, when it is clear from the context, we will also use the notation `$\sit$' to denote the \emph{set} $\Gamma(\sit)$.
For example, we will use $\indica{\sit}$ as a shorthand notation for $\indica{\Gamma(\sit)}$.
Moreover, for any two extended real variables $g,h\in\setofgenextvariables(\samplespace{})$ and any situation~$s\in\situations$, we use $g \leq_s f$ to denote that $g(\omega) \leq f(\omega)$ for all $\omega\in\Gamma(s)$, and similarly for $\geq_s$, $>_s$ and $<_s$.

We will distinguish between local variables and global variables.
\emph{Local} variables are simply maps on the state space $\statespace{}$ and are typically interpreted as depending on the value of a single uncertain state $X_n$.
They were already used before, for instance, when we introduced local upper expectations.
In accordance with our earlier conventions, we use $\setofgenextvariables{}(\statespace{})$ to denote all local extended real variables, and similarly for $\setofgenextvariablesb{}(\statespace{})$ and $\setofgengambles{}(\statespace{})$.
\emph{Global} variables, on the other hand, are maps on the sample space $\samplespace{}$, and are therefore suitable for representing inferences that depend on the values of a large---possibly infinite---number of uncertain states.
For example, the hitting time of a subset $A\subseteq\statespace{}$ is described by the global (extended real) variable $\tau_A \in \setofgenextvariables{}(\samplespace{})$ that takes the value $\tau_A(\omega)\coloneqq \min\{n\in\nats \colon \omega_n\in A\}$ for any $\omega\in\samplespace{}$, and clearly depends on the values of an infinite number of subsequent states.
We denote the set of all \emph{global} extended real variables by $\setofextvariables \coloneqq\setofgenextvariables(\samplespace)$, and similarly for $\setofextvariablesb \coloneqq\setofgenextvariablesb(\samplespace)$ and $\setofgambles \coloneqq\setofgengambles(\samplespace)$.
For any natural $k\leq\ell$, a special type of global variable---that is not necessarily extended real-valued---is the projection map $X_{k:\ell}$; for any path $\omega\in\samplespace$, this variable assumes the value $X_{k:\ell}(\omega)\coloneqq(\omega_k,...,\omega_\ell)$.
As such, for any $k\in\nats{}$, $X_k=X_{k:k}$ can also be regarded as a type of global variable.
For any $m, n\in\nats$ and any map $f \colon \statespace^n \to \extreals$, this allows us to write $f(X_{m:m+n-1})$ to denote the extended real global variable defined by $f(X_{m:m+n-1}) \coloneqq f \circ X_{m:m+n-1}$.
In this way, we can elegantly associate a global variable $f(X_n)$ with any local variable $f \colon \statespace \to \extreals$ and any discrete time point $n\in\nats$.

Our aim now is to combine the local upper expectations $\lupprev{s}$---which only tell us something about state transitions---and construct a global uncertainty model in the form of a single upper expectation on the global variables $\setofextvariables$ (and conditional on the situations $\situations{}$).
A crucial tool to do so, is the notion of a supermartingale; a special type of process.

Any map $\process$ on $\situations$ is called a \emph{process}.
An extended real(-valued) process $\process$ is called bounded below if there is some $M\in\reals{}$ such that $\process(s) \geq M$ for all $s\in\situations$.
Furthermore, with any situation $s\in\situations$ and any extended real process $\process$, we can associate the local variable $\process(s\andstate)\in\setofgenextvariables(\statespace)$ defined by
$\process(s\andstate)(x) \coloneqq\process(sx)$ for all $x\in\statespace$.
The extended real variables $\liminf\process\in\setofextvariables$ and $\limsup\process\in\setofextvariables$, will be defined by
\begin{align*}
\liminf\process(\omega)\coloneqq\liminf_{n\to+\infty}\process(\omega^n)
\text{~~and~~}
\limsup\process(\omega)\coloneqq\limsup_{n\to+\infty}\process(\omega^n)
\end{align*}
for all $\omega\in\Omega$.
If $\liminf\process=\limsup\process$, we denote their common value by $\lim\process$.

For a given imprecise probabilities tree, a \emph{supermartingale} $\martingale$ is an extended real process such that
$\lupprev{s}(\martingale(s\andstate))\leq\martingale(s)$ for all $s\in\situations$.
So a supermartingale is an extended real process that, according to the local models $\lupprev{s}$, is expected to decrease.
When adopting a behavioural interpretation, supermartingales can be seen to represent betting strategies that are allowed by our subject.
Roughly speaking, the condition that $\lupprev{s}(\martingale(s\andstate))\leq\martingale(s)$ with $s=\sit$ then means that our subject---
for the sake of simplicity, we ignore the subtlety about the extended real-valuedness---is willing to receive the price $\martingale(\sit)$ for giving away the uncertain variable $\martingale(\sit X_{n+1})$ that will be evaluated in the next time instant.
Hence, if we take him up on his commitments, we can pay him $\martingale(\sit)$ to receive $\martingale(\sit X_{n+1})$.
The next time instant, if the state of the process turns out to be $x_{n+1} \in \statespace{}$, we obtain the---possibly negative---payoff $\martingale(x_{1:n+1})$.
By repeating this procedure, we find that the supermartingale $\martingale{}$ represents a possible evolution of our capital when we would gamble against the subject.
We will denote the set of all \emph{bounded below} supermartingales for a given imprecise probabilities tree by $\setofextsupmartb{}$.

In the framework of Shafer and Vovk, the role of our subject above is taken up by a player called `Forecaster', whereas supermartingales represent possible betting strategies for a second player called `Skeptic'.
Given this game-theoretic setting, they consider the following question: How can Skeptic use Forecaster's assessments to determine selling and buying prices for a gamble $f$ whose uncertain payoff depends on the process state at multiple or even an infinite number of time instances?
Shafer and Vovk argue that Skeptic should certainly agree on selling $f$ for a price $\alpha$ such that, if Skeptic starts with an initial capital $\alpha$ and gambles in an appropriate way against Forecaster, he will end up with a higher capital than the payoff $f(\omega)$ corresponding to $f$ \emph{irrespectively of the path $\omega \in \samplespace{}$ taken by the process}.
Indeed, selling $f$ for a price $\alpha$ means that Skeptic receives $\alpha - f$.
If Skeptic is then able to turn the initial capital $\alpha$ into a final capital $K$ such that $K(\omega) \geq f(\omega)$ for all paths $\omega \in \samplespace{}$, his net payoff $K - f$ is non-negative for all $\omega$.
Hence, Skeptic should accept the transaction of selling $f$ for $\alpha$.
The infimum of these prices $\alpha$ is what Shafer and Vovk then call the (global) game-theoretic upper expectation of $f$.

More formally, given an imprecise probabilities tree consisting of local upper expectations $\lupprev{s}$ for all $s\in\situations$, we use its compatible set of bounded below supermartingales $\setofextsupmartb{}$ to define the corresponding \emph{(global) game-theoretic upper expectation} $\upprevvovkk$ as follows.
\begin{definition}\label{def:upperexpectation2}
For any imprecise probabilities tree, the corresponding (global) game-theoretic upper expectation $\upprevvovkk(\cdot \vert \cdot) \colon \setofextvariables \times \situations \to \extreals$ is defined by
\begin{align}\label{upprev3}
\upprevvovkk (f \vert s)\coloneqq\inf \big\{ \martingale(s) \colon \martingale\in\setofextsupmartb \text{ and } \liminf\martingale \geq_s f \big\} \ \text{ for all } f\in\setofextvariables \text{ and all } s\in\situations.
\end{align}
\end{definition}

The game-theoretic lower expectation $\lowprevvovkk(\cdot \vert \cdot) \colon \setofextvariables \times \situations \to \extreals$ is defined by the conjugacy relation $\lowprevvovkk(f \vert s) \coloneqq -\upprevvovkk(-f \vert s)$ for all $f\in\setofextvariables$ and all $s\in\situations$.
We will show later in Corollary~\ref{corollary: Vovk is an upper expectation} that, for any $s\in\situations$, the map ${\upprevvovkk(\cdot \, \vert s)\colon\setofextvariables\to\extreals}$ satisfies~\ref{coherence: const is const}--\ref{coherence: continuity wrt non-negative functions}, which justifies calling $\upprevvovkk$ an upper expectation.
Mimicking the link between traditional expectations and probabilities, we call $\overline{\mathrm{P}}_{\mathrm{V}}(A\vert s)\coloneqq\upprevvovkk{}(\indica{A} \vert s)$, for any $A\subseteq\Omega$ and any $s\in\situations{}$, the \emph{game-theoretic upper probability} of the event $A$ conditional on the situation $s$. Similarly, we call $\underline{\mathrm{P}}_{\mathrm{V}}(A\vert s)\coloneqq\lowprevvovkk{}(\indica{A} \vert s)$ the \emph{game-theoretic lower probability} of $A$ conditional on $s$.
We will also let $\upprevvovkk{}(f) \coloneqq \upprevvovkk{}(f \vert \Box)$ and $\lowprevvovkk{}(f) \coloneqq \lowprevvovkk{}(f \vert \Box)$ for all $f\in\setofextvariables$.

Note that $\upprevvovkk{}$ does not depend on the values of the local models $\lupprev{s}$ on $\setofgenextvariables{}(\statespace{}) \setminus \setofgenextvariablesb{}(\statespace{})$, because the infimum in Definition~\ref{def:upperexpectation2} is taken over supermartingales that are bounded below.
This confirms our earlier claim in Section~\ref{Sect: Upper Expectations}, where we said that, as far as the global upper expectation $\upprevvovkk{}$ is concerned, we can assume without loss of generality that the local models $\lupprev{s}$ additionally satisfy Shafer and Vovk's axioms~\ref{coherence: sublinearity extended}--\ref{coherence: monotonicity extended}.
Our reason for adopting this particular definition, where only bounded below supermartingales are considered, will be discussed in Section~\ref{sect: Discussion}.
Intuitively, however, one could interpret this assumption as a concretisation of the fact that a subject (e.g. Skeptic) cannot borrow an infinite or even unbounded amount of money.

\section{Basic Properties of Game-Theoretic Upper Expectations}\label{Sect: Basic properties of game}

We start by establishing some basic, yet essential properties of game-theoretic upper expectations.
The main ones are an extended version of coherence, partial compatibility with the local upper expectations and a law of iterated upper expectations.
Most of these results are not entirely new and have already been proved in a slightly different setting; our contribution then consists in adapting their proofs to our setting.
We start with the following two, rather abstract lemmas about supermartingales.

\begin{lemma}\label{lemma: infima of supermartingales}
	Consider any $\martingale\in\setofextsupmartb{}$ and any situation~$s\in\situations$. Then
	\begin{equation*}
		\martingale(s) \geq\inf_{\omega\in\Gamma(s)} \limsup\martingale(\omega) \geq\inf_{\omega\in\Gamma(s)} \liminf\martingale(\omega).
	\end{equation*}
\end{lemma}
\begin{proof}
The proof is similar to that of \cite[Lemma 1]{DECOOMAN201618}, where instead real supermartingales were used.
Since $\martingale$ is a bounded below supermartingale, we have that $\lupprev{s}(\martingale(s\andstate))\leq\martingale(s)$, which by \ref{coherence: bounds} implies that $\inf_{x\in\statespace} \martingale(sx)\leq\martingale(s)$.
Hence, since $\statespace$ is finite, there is at least one $x\in\statespace$ such that $\martingale(sx)\leq\martingale(s)$.
Repeating this argument over and over again, leads us to the conclusion that there is some $\omega\in\Gamma(s)$ such that $\limsup_{n\to+\infty} \martingale(\omega^n)\leq\martingale(s)$ and therefore also $\inf_{\omega\in\Gamma(s)} \limsup\martingale(\omega)\leq\martingale(s)$.
The rest of the proof is now trivial.
\end{proof}

\begin{lemma}\label{lemma:positive:countable:linear:combination}
Consider any countable collection $\{\martingale_n\}_{n\in\natz}$ of supermartingales that have a common lower bound, and any countable collection of non-negative real numbers $\{\lambda_n\}_{n\in\natz}$ such that $\sum_{n\in\natz}\lambda_n$ is a real number $\lambda$.
Then $\martingale\coloneqq\sum_{n\in\natz}\lambda_n\martingale_n$ is again a bounded below supermartingale.
If, moreover, all $\martingale_n$ are non-negative, then so is~$\martingale$.
\end{lemma}

\begin{proof}
We only prove the first statement, as the second is then trivially true.
Since all $\martingale_n$ have a common lower bound, say $B \in \reals{}$, the processes $\martingale{}_n - B$ will be non-negative and therefore, because all reals $\lambda_n$ are also non-negative, the sum $\sum_{n\in\natz}\lambda_n[\martingale_n(s)-B]$ exists and is non-negative for all $s \in \situations{}$.
Then, in order to see that $\martingale{}$ is well-defined, note that
\begin{align}\label{Eq: lemma:positive:countable:linear:combination}
\sum_{n\in\natz} \lambda_n[\martingale_n(s)-B] + \lambda B
&= \lim_{n\to+\infty} \Bigl( \sum_{i=0}^{n} \lambda_i[\martingale_i(s)-B] + \sum_{i=0}^{n}\lambda_i B \Bigr) \nonumber \\
&= \lim_{n\to+\infty} \sum_{i=0}^{n} \Bigl(\lambda_i[\martingale_i(s)-B] + \lambda_i B \Bigr)
= \sum_{n\in\natz}\lambda_n \martingale_n(s)
\eqqcolon \martingale{}(s),
\end{align}
for all $s\in\situations{}$, where the first step takes into account that $\lambda B$ is real and the third step takes into account that all $\lambda_i$ and $B$ are real.
The equality above, together with the non-negativity of $\sum_{n\in\natz}\lambda_n[\martingale_n(s)-B]$ immediately shows that $\martingale{}$ is bounded below by the real $\lambda B$.
It also shows that $\martingale{}$ is a supermartingale.
Indeed, for any $s\in\situations{}$, we find that
\begin{multline*}
\lupprev{s}(\martingale{}(s\cdot))
\overset{\eqref{Eq: lemma:positive:countable:linear:combination}}{=} \lupprev{s}\Big(\sum_{n\in\natz} \lambda_n[\martingale_n(s \cdot)-B] + \lambda B\Big)
\overset{\text{\ref{coherence: const add}}}{=} \lupprev{s}\Big(\sum_{n\in\natz} \lambda_n[\martingale_n(s \cdot)-B]\Big) + \lambda B \\
\leq \sum_{n\in\natz} \lupprev{s}\Bigl( \lambda_n [\martingale_n(s\andstate) - B] \Bigr) + \lambda B
\overset{\text{\ref{coherence: homog for non-negative lambda}}}{=} \sum_{n\in\natz} \lambda_n \lupprev{s}(\martingale_n(s\andstate) - B) + \lambda B \\
\overset{\text{\ref{coherence: const add}}}{=} \sum_{n\in\natz} \lambda_n \Big(\lupprev{s}(\martingale_n(s\andstate)) - B\Big) + \lambda B
{=} \sum_{n\in\natz} \lambda_n \lupprev{s}(\martingale_n(s\andstate))
\leq \sum_{n\in\natz} \lambda_n \martingale_n(s) = \martingale{}(s),
\end{multline*}
where we were allowed to apply \ref{coherence: const add} and \ref{coherence: homog for non-negative lambda} because $\sum_{n\in\natz} \lambda_n[\martingale_n(s \cdot)-B]$, all $\martingale_n(s \cdot) - B$ and all $\martingale_n(s \cdot)$ are bounded below, where the first inequality followed from Proposition~\ref{proposition: countable subadditivity} and where the last inequality followed from the non-negativity of all $\lambda_n$ and the fact that all $\martingale{}_n$ are supermartingales.
\end{proof}

The following result states that $\upprevvovkk{}$ satisfies a version of the coherence axioms for global extended real variables.
A first version of the result was stated in \cite[Chapter 8]{Shafer:2005wx}, yet, our proof is very similar to that of \cite[Prop. 14]{DECOOMAN201618}: we adapt it here to the fact that our bounded below supermartingales take values in $\extreals$ rather than $\reals{}$.

\begin{proposition}\label{Prop: Coherence of Global Game-theoretic upper expectation}
For all extended real variables $f,g\in\setofextvariables$, all $\lambda\in\nnegreals{}$, all $\mu\in\reals{}$ and all situations $s\in\situations$, $\upprevvovkk$ satisfies
\begin{enumerate}[leftmargin=*,ref={\upshape V}\arabic*,label={\upshape V\arabic*.}, series=global coherence]
\item \label{vovk coherence 1}
$\upprevvovkk(f \vert s)\leq\sup_{\omega \in \Gamma(s)} f(\omega)$;
\item \label{vovk coherence 2}
$\upprevvovkk(f+g \vert s)\leq\upprevvovkk(f \vert s) + \upprevvovkk(g \vert s)$;
\item \label{vovk coherence 3}
$\upprevvovkk(\lambda f \vert s)=\lambda \upprevvovkk(f \vert s)$.
\item \label{vovk coherence 4}
$f \leq_s g \Rightarrow \upprevvovkk(f \vert s)\leq\upprevvovkk(g \vert s)$;
\item \label{vovk coherence 5}
$\inf_{\omega \in \Gamma(s)} f(\omega)\leq\lowprevvovkk(f \vert s)\leq\upprevvovkk(f \vert s)\leq\sup_{\omega \in \Gamma(s)} f(\omega)$;
\item \label{vovk coherence 6}
$\upprevvovkk(f + \mu \vert s)=\upprevvovkk(f \vert s) + \mu$.
\end{enumerate}
\end{proposition}
\begin{proof}
\ref{vovk coherence 1}.
If $\sup_{\omega\in\Gamma(s)} f(\omega)=+\infty$, the inequality is trivially satisfied.
If this is not the case, consider any real $M \geq \sup_{\omega\in\Gamma(s)} f(\omega)$ and the real process $\martingale$ that assumes the constant value $M$.
Then clearly $\martingale$ is a bounded below supermartingale and moreover $\liminf\martingale(\omega)=M \geq f(\omega)$ for all $\omega\in\Gamma(s)$.
Hence, Definition \ref{def:upperexpectation2} implies that $\smash{\upprevvovkk(f \vert s)\leq\martingale(s)=M}$.
Since this is true for every real $M \geq \sup_{\omega\in\Gamma(s)} f(\omega)$, \ref{vovk coherence 1} follows.

\ref{vovk coherence 2}.
If either $\upprevvovkk(f \vert s)$ or $\upprevvovkk(g \vert s)$ equals $+\infty$, then the inequality is trivially true.
So suppose that $\upprevvovkk(f \vert s) < +\infty$ and $\upprevvovkk(g \vert s) < +\infty$ and consider any real $c_1 > \upprevvovkk(f \vert s)$ and any real $c_2 > \upprevvovkk(g \vert s)$.
Then there are two bounded below supermartingales $\martingale_1$ and $\martingale_2$ such that $\martingale_1(s) \leq c_1$ and $\martingale_2(s) \leq c_2$ and moreover $\liminf\martingale_1 \geq_s f$ and $\liminf\martingale_2 \geq_s g$.
Now consider the extended real process $\martingale \coloneqq\martingale_1 + \martingale_2$.
Then $\martingale$ is a bounded below supermartingale because of Lemma~\ref{lemma:positive:countable:linear:combination}, which we can apply because $\martingale_1$ and $\martingale_2$ are both bounded below and hence have a common lower bound [note that the countable sum in Lemma~\ref{lemma:positive:countable:linear:combination} can be turned into a finite sum by setting all remaining supermartingales equal to zero].
Moreover, we will show that $\liminf (\martingale_1 + \martingale_2) \geq \liminf\martingale_1 + \liminf\martingale_2$ and therefore that $\liminf\martingale \geq_s f+g$, which, by Definition \ref{def:upperexpectation2}, implies that $\upprevvovkk(f+g \vert s)\leq\martingale(s)=\martingale_1(s)+\martingale_2(s)\leq c_1 + c_2$.
Since this then holds for any real $c_1 > \upprevvovkk(f \vert s)$ and any real $c_2 > \upprevvovkk(g \vert s)$, it follows that $\upprevvovkk(f+g \vert s)\leq\upprevvovkk(f \vert s) + \upprevvovkk(g \vert s)$.

So consider any $\omega\in\samplespace$ and any real $\alpha_1$ and $\alpha_2$ such that $\liminf\martingale_1(\omega) > \alpha_1$ and $\liminf\martingale_2(\omega) > \alpha_2$.
This is always possible because $\martingale_1$ and $\martingale_2$ are bounded below.
Then there are two natural numbers $N_1$ and $N_2$ such that $\martingale_1(\omega^{n_1}) \geq \alpha_1$ and $\martingale_2(\omega^{n_2}) \geq \alpha_2$ for all $n_1 \geq N_1$ and all $n_2 \geq N_2$.
Hence, we have that $\martingale_1(\omega^{n}) + \martingale_2(\omega^{n})  \geq \alpha_1 + \alpha_2$ for all $n \geq \max\{N_1,N_2\}$, implying that $\liminf (\martingale_1 + \martingale_2)(\omega) \geq \alpha_1 + \alpha_2$.
Since this holds for any real $\alpha_1$ and $\alpha_2$ such that $\liminf\martingale_1(\omega) > \alpha_1$ and $\liminf\martingale_2(\omega) > \alpha_2$, we indeed find that $\liminf (\martingale_1 + \martingale_2)(\omega) \geq \liminf\martingale_1(\omega) + \liminf\martingale_2(\omega)$.

\ref{vovk coherence 3}.
For $\lambda \in \posreals{}$, it suffices to note that $\martingale$ is a bounded below supermartingale such that $\liminf\martingale \geq_s f$ if and only if $\lambda \martingale$ is a bounded below supermartingale such that $\liminf \lambda \martingale \geq_s \lambda f$.
If $\lambda=0$, then $\lambda \upprevvovkk(f \vert s)=0$ because $(+\infty) \cdot 0=(-\infty) \cdot 0=0$.
To see that also $\upprevvovkk(\lambda f \vert s)=0$, start by noting that $\lambda f=0$ and hence, because of \ref{vovk coherence 1}, $\upprevvovkk(\lambda f \vert s) \leq 0$.
That $\upprevvovkk(\lambda f \vert s) < 0$ is impossible, follows from Lemma~\ref{lemma: infima of supermartingales} and Definition \ref{def:upperexpectation2}.
Hence, we indeed have that $\upprevvovkk(\lambda f \vert s)=0$.

\ref{vovk coherence 4}.
Consider any two $f,g\in\setofextvariables$ such that $f \leq_s g$.
Then for any $\martingale\in\setofextsupmartb{}$ such that $\liminf\martingale \geq_s g$, we also have that $\liminf\martingale \geq_s f$, and hence, by Definition \ref{def:upperexpectation2}, $\upprevvovkk(f \vert s)\leq\upprevvovkk(g \vert s)$.

\ref{vovk coherence 5}.
The first and third inequality follow trivially from \ref{vovk coherence 1} and the definition of the conjugate lower expectation $\lowprevvovkk$.
To prove the second inequality, assume \emph{ex absurdo} that $\lowprevvovkk(f \vert s) > \upprevvovkk(f \vert s)$.
Then $0 > \upprevvovkk(f \vert s)-\lowprevvovkk(f \vert s)$ which, by \ref{vovk coherence 2} and the definition of the conjugate lower expectation $\lowprevvovkk$, implies that $0 > \upprevvovkk(f + (- f) \vert s)$.
Since, according to our convention, the extended real variable $f + (-f)$ only assumes values in $\{ 0 ,+\infty \}$, we have that $f + (-f) \geq 0$ and therefore, by \ref{vovk coherence 4} and \ref{vovk coherence 3}, that $\upprevvovkk(f + (- f) \vert s) \geq \upprevvovkk(0 \vert s)=0$.
This is a contradiction.

\ref{vovk coherence 6}.
For any $\martingale\in\setofextsupmartb{}$ such that $\liminf\martingale \geq_s f + \mu$, we have that $\martingale-\mu\in\setofextsupmartb{}$ because of \ref{coherence: const add} and moreover $\liminf (\martingale-\mu) \geq_s f$.
Hence, $\upprevvovkk(f \vert s)\leq\martingale(s)-\mu$ and therefore also $\upprevvovkk(f \vert s) + \mu\leq\martingale(s)-\mu + \mu=\martingale(s)$.
Since this holds for any $\martingale\in\setofextsupmartb{}$ such that $\liminf\martingale \geq_s f + \mu$, we have that $\upprevvovkk(f \vert s) + \mu\leq\upprevvovkk(f + \mu \vert s)$.
By applying this inequality to $f'=f+ \mu$ and $\mu'=-\mu$, we also find that $\upprevvovkk(f + \mu \vert s)-\mu\leq\upprevvovkk(f \vert s)$.
\end{proof}

In order to formulate our next result, we require the concept of an $n$-measurable variable.
For a given $n\in\natz$, we call a global variable $f$ \emph{$n$-measurable} if it is constant on the cylinder events $\Gamma(\sit)$ for all $\sit\in\statespaceseq{1}{n}$, that is, if $f=\tilde{f}(X_{1:n})$ for some map $\tilde{f}$ on $\statespace^n$.
We will then also use the notation $f(\sit)$ for its constant value $f(\omega)$ on all paths $\omega\in\Gamma(\sit)$.
Similarly, for a global variable $f$ that only depends on the $n$-th state $X_n$, we will use $f(x_n)$ to denote its constant value on the event $\{\omega\in\Omega \colon \omega_n=x_n\}$.
We call a global variable $f\in\setofextvariables$ \emph{finitary} if it is $n$-measurable for some $n\in\natz$.
With any situation $\sit\in\situations$ and any $(n+1)$-measurable extended real variable $f$, we now associate a local variable $f(\sit \cdot)$ defined by
$f(\sit \cdot)(\situa{n+1}) \coloneqq f(\situation{1}{n+1}) \text{ for all } \situa{n+1}\in\statespace$.
On the other hand, for any extended real process $\process$, we will define the global variable $\process(X_{1:n}) \coloneqq\process \circ X_{1:n}$ that only depends of the first $n$ states, and is therefore finitary.

Our proof of Proposition~\ref{prop: global compatible with local} also requires the following additional notation and terminology.
For any two situations $s,t \in \situations$, we write that $s \sqsubseteq t$, or equivalently that $t \sqsupseteq s$, when every path that goes through $t$ also goes through $s$.
In that case we say that $s$ \emph{precedes} $t$ or that $t$ \emph{follows} $s$.
When $s \sqsubseteq t$ and $ s \not= t$, we write $s \sqsubset t$ and similarly for the relation $\sqsupset$.
When neither $s \sqsubseteq t$ nor $t \sqsubseteq s$, we say that $s$ and $t$ are \emph{incomparable}.
\begin{proposition}[Partial compatibility with local models]\label{prop: global compatible with local}
Consider any situation $\sit\in\situations$ and any $(n+1)$-measurable extended real variable $f$ that is bounded below.
Then,
\begin{equation*}
		\upprevvovkk (f \vert \sit)=\lupprev{\sit}(f(\sit \cdot)).
\end{equation*}
\end{proposition}
\begin{proof}
Our proof is similar to that of \cite[Corollary~3]{DECOOMAN201618}.
Consider any $\martingale\in\setofextsupmartb{}$ such that $\liminf\martingale \geq_{\sit} f$.
Then it follows from Lemma~\ref{lemma: infima of supermartingales} that, for all $x_{n+1}\in\statespace$,
\begin{align*}
\martingale(x_{1:n+1}) \geq\inf_{\omega\in\Gamma(x_{1:n+1})} \liminf\martingale(\omega) \geq\inf_{\omega\in\Gamma(x_{1:n+1})} f(\omega)=f(x_{1:n+1}).
\end{align*}
Hence, we have that $\martingale(\sit \cdot) \geq f(\sit \cdot)$, which implies by \ref{coherence: monotonicity} and the supermartingale character of $\martingale$ that
\begin{equation*}
\martingale(\sit) \geq \lupprev{\sit}(\martingale(\sit \cdot)) \geq \lupprev{\sit}(f(\sit \cdot)).
\end{equation*}
Since this holds for any $\martingale\in\setofextsupmartb{}$ such that $\liminf\martingale \geq_{x_{1:n}} f$, it follows from Definition \ref{def:upperexpectation2} that $\upprevvovkk (f \vert \sit) \geq \lupprev{\sit}(f(\sit \cdot))$.
To see that the inequality is an equality, consider the extended real process $\martingale$ defined by $\martingale(s) \coloneqq\lupprev{\sit}(f(\sit \cdot))$ for all $s \not\sqsupset \sit$, and by $\martingale(s) \coloneqq f(x_{1:n+1})$ for any $s\in\situations$ such that $s \sqsupseteq x_{1:n+1}$ for some $x_{n+1}\in\statespace$.
Then $\martingale$ is bounded below because $f$ is bounded below and $\lupprev{\sit}$ satisfies~\ref{coherence: bounds}.
It is also a supermartingale because $\lupprev{\sit}(\martingale(\sit \cdot))=\lupprev{\sit}(f(\sit \cdot))=\martingale(\sit)$ and, for all $s \not= \sit$, $\lupprev{s}(\martingale(s\andstate))=\martingale(s)$ because of \ref{coherence: bounds} and the fact that $\martingale{}(s\cdot)$ is constant and equal to $\martingale{}(s)$.
It is moreover easy to see that $\liminf\martingale \geq_{\sit} f$ is guaranteed because $f$ is $(n+1)$-measurable.
\end{proof}
We will show later on that this compatibility can be extended to the entire domain of the local models $\lupprev{s}$ provided that they additionally satisfy \ref{ext coherence lower cuts continuity}.

The following important result is an imprecise generalisation of the well-known `law of iterated expectations'.
The idea of the proof goes back to \cite[Proposition 8.7]{Shafer:2005wx}, yet, our proof is more similar to that of \cite[Theorem 16]{DECOOMAN201618}.


\begin{theorem}[Law of iterated upper expectations]\label{theorem: vovk iterated}
For any $f\in\setofextvariables$ and any $x_{1:n}\in\situations$, we have that
\begin{equation*}
\upprevvovkk(f \vert x_{1:n})=\upprevvovkk\Big(\upprevvovkk\left(f \vert x_{1:n} X_{n+1}\right) \Big\vert x_{1:n}\Big).
\end{equation*}
\end{theorem}
\begin{proof}
Fix any $f\in\setofextvariables$ and any $x_{1:n}\in\situations$.
We first show that $\upprevvovkk(\upprevvovkk(f \vert x_{1:n} X_{n+1}) \vert x_{1:n})\leq\upprevvovkk(f \vert x_{1:n})$.
If $\upprevvovkk(f \vert x_{1:n})=+\infty$, this is trivially satisfied.
If not, then for any fixed real $\alpha > \upprevvovkk(f \vert x_{1:n})$ there is a bounded below supermartingale $\martingale$ such that $\martingale(x_{1:n})\leq\alpha$ and $\liminf\martingale \geq_{x_{1:n}} f$.
Then it is clear that, for all $x_{n+1}\in\statespace$, $\liminf\martingale \geq_{x_{1:n+1}} f$, and hence $\upprevvovkk(f \vert x_{1:n+1})\leq\martingale(x_{1:n+1})$ by Definition \ref{def:upperexpectation2}.
Let $\martingale'$ be the process that is equal to $\martingale$ for all situations that precede $x_{1:n}$ or are incomparable with $x_{1:n}$, and that is equal to the constant $\martingale(x_{1:n+1})$ for all situations that follow $x_{1:n+1}$ for some $x_{n+1}\in\statespace$.
Clearly, $\martingale'$ is again a bounded below supermartingale and, because of the reasoning above, $\upprevvovkk(f \vert x_{1:n} X_{n+1})\leq\martingale(x_{1:n} X_{n+1}) =_{x_{1:n}}  \liminf\martingale'$.
Hence, it follows from Definition \ref{def:upperexpectation2} that
$\upprevvovkk(\upprevvovkk(f \vert x_{1:n} X_{n+1}) \vert x_{1:n})\leq\martingale'(x_{1:n})=\martingale(x_{1:n})\leq\alpha.$
Since this holds for any real $\alpha > \upprevvovkk(f \vert x_{1:n})$, we indeed have that $\upprevvovkk(\upprevvovkk(f \vert x_{1:n} X_{n+1}) \vert x_{1:n})\leq\upprevvovkk(f \vert x_{1:n})$.

We now prove the other inequality.
Again, if $\upprevvovkk(\upprevvovkk(f \vert x_{1:n} X_{n+1}) \vert x_{1:n})=+\infty$ it trivially holds, so we can assume it to be real or equal to $-\infty$.
Fix any real $\alpha > \upprevvovkk(\upprevvovkk(f \vert x_{1:n} X_{n+1}) \vert x_{1:n})$ and any $\epsilon\in\posreals$.
Then there must be a bounded below supermartingale $\martingale$ such that $\martingale(x_{1:n})\leq\alpha$ and $\liminf\martingale \geq_{x_{1:n}} \upprevvovkk(f \vert x_{1:n} X_{n+1})$.
Consider any such bounded below supermartingale.
Then for any $x_{n+1}\in\statespace$, we have that $\liminf\martingale \geq_{x_{1:n+1}} \upprevvovkk(f \vert x_{1:n+1})$, which by Lemma~\ref{lemma: infima of supermartingales} implies that $\martingale(x_{1:n+1}) \geq \upprevvovkk(f \vert x_{1:n+1})$.
Fix any $x_{n+1}\in\statespace$.
Then $\martingale(x_{1:n+1})$ is either real or equal to $+\infty$ because $\martingale$ is bounded below.
If $\martingale(x_{1:n+1})$ is real, then since $\martingale(x_{1:n+1}) \geq \upprevvovkk(f \vert x_{1:n+1})$, it follows from Definition \ref{def:upperexpectation2} that there is a bounded below supermartingale $\martingale_{x_{1:n+1}}$ such that $\martingale_{x_{1:n+1}}(x_{1:n+1})\leq\martingale(x_{1:n+1}) + \epsilon$ and $\liminf\martingale_{x_{1:n+1}} \geq_{x_{1:n+1}} f$.
If $\martingale(x_{1:n+1})$ is $+\infty$, let $\martingale_{x_{1:n+1}}$ be the constant supermartingale that is equal to $+\infty$ everywhere.
So, for all $x_{n+1}\in\statespace$, we have found a bounded below supermartingale $\martingale_{x_{1:n+1}}$ such that $\martingale_{x_{1:n+1}}(x_{1:n+1})\leq\martingale(x_{1:n+1}) + \epsilon$ and $\liminf\martingale_{x_{1:n+1}} \geq_{x_{1:n+1}} f$.
Let $\martingale^\ast$ be the process that is equal to $\martingale + \epsilon$ for all situations that precede or are incomparable with $x_{1:n}$, and that is equal to $\martingale_{x_{1:n+1}}$ for all situations that follow $x_{1:n+1}$ for some $x_{n+1}\in\statespace$.
Note that $\liminf\martingale^\ast \geq_{x_{1:n}} f$ because, for each $x_{n+1}\in\statespace$, we have that $\liminf\martingale^\ast =_{x_{1:n+1}} \liminf\martingale_{x_{1:n+1}} \geq_{x_{1:n+1}} f$.
We moreover show that $\martingale^\ast$ is a bounded below supermartingale.

The process $\martingale^\ast$ is clearly bounded below because $\martingale{}$ and all $\martingale{}_{x_{1:n+1}}$ are bounded below and $\statespace{}$ is finite.
Furthermore, for any $x_{n+1}\in\statespace$, we have that $\martingale^\ast(x_{1:n+1})=\martingale_{x_{1:n+1}}(x_{1:n+1})\leq\martingale(x_{1:n+1}) + \epsilon$, implying that $\martingale^\ast(x_{1:n} \cdot)\leq\martingale(x_{1:n} \cdot) + \epsilon$ and therefore, by \ref{coherence: monotonicity} and \ref{coherence: const add}, that
\begin{equation*}
\lupprev{x_{1:n}}(\martingale^\ast(x_{1:n} \cdot))\leq\lupprev{x_{1:n}}(\martingale(x_{1:n} \cdot) + \epsilon) = \lupprev{x_{1:n}}(\martingale(x_{1:n} \cdot))+ \epsilon \leq  \martingale(x_{1:n}) + \epsilon=\martingale^\ast(x_{1:n}).
\end{equation*}
Moreover, for all situations $s \not\sqsupseteq x_{1:n}$, we have by \ref{coherence: const add} that $\lupprev{s}(\martingale^\ast(s\andstate))=\lupprev{s}(\martingale(s\andstate) + \epsilon)=\lupprev{s}(\martingale(s\andstate))+ \epsilon\leq\martingale(s) + \epsilon=\martingale^\ast(s)$, and for all $s\in\situations$ such that  $s \sqsupseteq x_{1:n+1}$ for some $x_{n+1}\in\statespace$, we have that $\lupprev{s}(\martingale^\ast(s\andstate))=\lupprev{s}(\martingale_{x_{1:n+1}}(s\andstate))\leq\martingale_{x_{1:n+1}}(s)=\martingale^\ast(s)$.
All together, we have that $\lupprev{s}(\martingale^\ast(s\andstate))\leq\martingale^\ast(s)$ for all $s \in \situations{}$, implying that $\martingale^\ast$ is a supermartingale that, as shown before, is bounded below.

Since $\liminf\martingale^\ast \geq_{x_{1:n}} f$ and $\martingale^\ast(x_{1:n})=\martingale(x_{1:n}) + \epsilon\leq\alpha + \epsilon$, Definition~\ref{def:upperexpectation2} now implies that $\upprevvovkk(f \vert x_{1:n})\leq\alpha + \epsilon$.
This holds for any $\epsilon\in\posreals$ and any real $\alpha > \upprevvovkk(\upprevvovkk(f \vert x_{1:n} X_{n+1}) \vert x_{1:n})$, so we indeed conclude that $\upprevvovkk(f \vert x_{1:n})\leq\upprevvovkk(\upprevvovkk(f \vert x_{1:n} X_{n+1}) \vert x_{1:n})$.
\end{proof}

\begin{corollary}\label{corollary: upprevvovk is supermartingale}
For any $f\in\setofextvariablesb$, the process $\process$, defined by $\process(s) \coloneqq\upprevvovkk(f \vert s)$ for all $s\in\situations$, is a bounded below supermartingale.
\end{corollary}
\begin{proof}
Consider any $f\in\setofextvariablesb$.
Then $\process$ is bounded below because $f$ is bounded below and $\upprevvovkk$ satisfies~\ref{vovk coherence 5}.
Moreover, if for any $s \in \situations{}$ we let $\upprevvovkk(f \vert s\andstate)$ be the (bounded below) local variable that assumes the value $\upprevvovkk(f \vert sx)$ for all $x\in\statespace$, then it follows from Proposition~\ref{prop: global compatible with local} and Theorem \ref{theorem: vovk iterated} that
\vspace{0.2cm}
\begin{align*}
\lupprev{x_{1:n}}(\process(x_{1:n} \cdot))
= \lupprev{x_{1:n}}\big(\upprevvovkk(f \vert x_{1:n} \cdot)\big)
&= \upprevvovkk\big(\upprevvovkk(f \vert x_{1:n} X_{n+1}) \vert x_{1:n}\big) \\
&= \upprevvovkk(f \vert x_{1:n})=\process(x_{1:n}) \text{ for all } x_{1:n}\in\situations.
\end{align*}
Hence, $\process$ is indeed a supermartingale, and therefore a bounded below supermartingale.
\end{proof}

\section{Equivalent definitions for $\upprevvovkk{}$}\label{Sect: equivalent characterisations}

We start the current section by presenting two technical results that are essential for our further analysis of Definition~\ref{def:upperexpectation2}: Doob's Convergence Theorem and L\'evy's Zero-one Law.
Both of them also hold in a precise measure-theoretic context, but our results do not require any measurability conditions, nor do they require the local models to be precise.
The game-theoretic versions we present here are due to Shafer, Vovk and Takemura \cite{shafer2011levy,ShaferVovk2014ITIP,Vovk2019finance}.
However, since our framework slightly differs from theirs, we have adapted their proofs to our setting.
Some of the involved arguments are rather lengthy and technical, though, so we have chosen to relegate these proofs
to an appendix at the end of the paper.

To state the results, we require the following terminology.
For any $s\in\situations$, we say that a supermartingale $\martingale\in\setofextsupmartb{}$ is an \emph{$s$-test supermartingale} if it is non-negative and $\martingale(s)=1$.
If $s=\Box$, we simply say it is a test supermartingale.
For any $s\in\situations$, we say that an event $A \subseteq\Omega$ is \emph{strictly almost sure (s.a.s.)} within $\Gamma(s)$ if there is an $s$-test supermartingale that converges to $+\infty$ on $\Gamma(s) \setminus A$.
In that case, we call the event $A^c$ \emph{strictly null} within $\Gamma(s)$.
If $s=\Box$, we drop the `within' and simply speak of `strictly almost sure'.
For any two $f,g\in\setofextvariables$, we will use the notation $f \geq_s g$ s.a.s.---and similarly for $\leq_s$, $>_s$ and $<_s$---to indicate that the event $\{\omega\in\Omega \colon f(\omega) \geq g(\omega)\}$ is strictly almost sure within $\Gamma(s)$.


It can easily be shown that an event $A\subseteq\Omega$ is strictly almost sure within $\Gamma(s)$ if and only $\overline{\mathrm{P}}_{\mathrm{V}}(A^c\vert s) = 0$ or, equivalently,\footnote{This follows from the fact that $\overline{\mathrm{P}}_{\mathrm{V}}(A^c\vert s)=\upprevvovkk(\indica{A^c}\vert s)=\upprevvovkk(1-\indica{A}\vert s)=1+\upprevvovkk(-\indica{A}\vert s)=1-\lowprevvovkk(\indica{A}\vert s)=1-\underline{\mathrm{P}}_{\mathrm{V}}(A\vert s)$, using~\ref{vovk coherence 6} for the third equality.} if and only if $\underline{\mathrm{P}}_{\mathrm{V}}(A\vert s) = 1$; we refer to \cite[Proposition~8.4]{Vovk2019finance} for an illustration in the case where $s=\Box$.
This is similar to the traditional measure-theoretic definition of an almost sure event; that is, a measurable event with (measure-theoretic) probability one.
In contrast with the measure-theoretic definition however, the game-theoretic approach provides a clear behavioural interpretation for strictly null events $A\subseteq\Omega$: it says that there is a strategy for Skeptic that allows him to start with a finite capital (in the situation $s$) and become infinitely rich on all paths $\omega\in A$ (that moreover go through $s$) without ever borrowing money.

Theorem~\ref{Theorem: Doob} below establishes a version of Doob's convergence law.
It states that a bounded below supermartingale converges to a real number strictly almost surely.
This is somewhat intuitive (yet, not trivial at all): since a supermartingale is bounded below and expected to decrease, one would expect it to converge to a real number.
We precede Theorem~\ref{Theorem: Doob} with a technical result about the limit behaviour of the test supermartingales that are involved in Theorem~\ref{Theorem: Doob}.

\begin{proposition}\label{Prop: Doob}
Consider any supermartingale $\martingale\in\setofextsupmartb{}$.
If $\martingale(s)$ is real for some $s\in\situations$, then there is a $s$-test supermartingale $\martingale^\ast$ that converges to $+\infty$ on all paths $\omega\in\Gamma(s)$ where $\martingale$ does not converge to an extended real number, and that converges to an extended real number on all paths $\omega\in\Gamma(s)$ where $\martingale{}$ converges to a real number.
\end{proposition}

\begin{theorem}[Doob's Convergence Theorem]\label{Theorem: Doob}
Consider any supermartingale $\martingale\in\setofextsupmartb{}$.
If $\martingale(s)$ is real for some $s\in\situations$, then $\martingale$ converges to a real number strictly almost surely within $\Gamma(s)$.
\end{theorem}

The following result, a version of L\'evy's zero-one law, captures (and extends) yet another intuitive idea:
the upper probability of an event $A\subseteq\Omega$ conditional on a situation $\omega^n$ should (or, is expected to) converge to $1$ as $n\to+\infty$ if $\omega\in A$.

\begin{theorem}[L\'evy's zero-one law]\label{Theorem: Levy}
	For any $f\in\setofextvariablesb$ and any $s\in\situations$, the event
	\begin{align*}
		A \coloneqq\Big\{\omega\in\Omega \colon \liminf_{n\to+\infty} \upprevvovkk(f \vert \omega^n) \geq f(\omega) \Big\} \text{ is strictly almost sure within } \Gamma(s).
	\end{align*}
\end{theorem}



One of the major consequences of Doob's convergence theorem and L\'evy's zero-one law is that they allow us to draw some interesting conclusions about Definition~\ref{def:upperexpectation2}.
In particular, we can use them to show that the resulting game-theoretic upper expectation is not impacted much by changes that concern the limit behaviour of supermartingales and, more specifically, how this limit behaviour relates to the variable $f$ at hand;
see Proposition~\ref{prop: liminf can be replaced by lim in definition} and \ref{prop: MonotoneConvergenceSAS}
below.
As was the case for the previous results in this section, the ideas underlying the proofs of the following results are due to Shafer, Vovk and Takemura \cite{Vovk2019finance,shafer2011levy}.



Our first result shows that, in Definition~\ref{def:upperexpectation2}, we can restrict ourselves to the bounded below supermartingales that converge within $\Gamma(s)$.
That is, the limit inferior in Definition~\ref{def:upperexpectation2} can be replaced by a limit.

\begin{proposition}\label{prop: liminf can be replaced by lim in definition}
For any $f\in\setofextvariables$ and any $s\in\situations$, we have that
\begin{align*}
\upprevvovkk(f \vert s) =\inf \left\{ \martingale(s) \colon \martingale\in\setofextsupmartb{} \text{ and } \lim \martingale \geq_s f \right\},
\end{align*}
where the condition $\lim \martingale \geq_s f$ is taken to implicitly imply that $\lim \martingale$ exists within $\Gamma(s)$.
\end{proposition}
\begin{proof}
The inequality `$\leq$' is trivially satisfied since $\liminf\martingale =_s \lim \martingale$ for any bounded below supermartingale $\martingale$ such that the limit $\lim \martingale$ exists within $\Gamma(s)$.
It remains to prove the other inequality.
If $\upprevvovkk(f \vert s)=+\infty$, this is trivially satisfied.
Otherwise, fix any real $\alpha > \upprevvovkk(f \vert s)$.
Then, due to Definition \ref{def:upperexpectation2}, there is a supermartingale $\martingale\in\setofextsupmartb{}$ such that $\martingale(s)\leq\alpha$ and $\liminf\martingale \geq_s f$.
Because $\martingale$ is bounded below and $\alpha$ is real, $\martingale(s)$ is also real.
So, by Proposition~\ref{Prop: Doob}, there is an $s$-test supermartingale $\martingale^\ast$ that converges to $+\infty$ on all paths $\omega\in\Gamma(s)$ where $\martingale$ does not converge in $\extreals$ and converges in $\extreals$ on all paths $\omega\in\Gamma(s)$ where $\martingale$ converges in $\reals{}$.

Fix any $\epsilon\in\posreals$ and consider the process $\martingale'$ defined by $\martingale'(t) \coloneqq \martingale(t) + \epsilon \martingale^\ast(t)$ for all situations $t \sqsupseteq s$ and by $\martingale'(t) \coloneqq \martingale(s) + \epsilon \martingale^\ast(s)\leq\alpha + \epsilon$ for all situations $t \not\sqsupseteq s$.
Then $\martingale'$ is bounded below because $\martingale$ and $\martingale^\ast$ are bounded below.
Moreover, as we will now show, it is also a supermartingale.
On the one hand, for all situations $t \sqsupseteq s$, we have that
\begin{align*}
\lupprev{t}(\martingale'(t\cdot))
= \lupprev{t}(\martingale(t\cdot) + \epsilon \martingale^\ast(t\cdot))
\overset{\text{\ref{coherence: sublinearity}, \ref{coherence: homog for positive lambda}}}{\leq} \lupprev{t}(\martingale(t\cdot)) + \epsilon \lupprev{t}(\martingale^\ast(t\cdot))
\leq \martingale(t) + \epsilon \martingale^\ast(t)
= \martingale'(t),
\end{align*}
where the second inequality follows from the fact that $\martingale$ and $\martingale^\ast$ are bounded below supermartingales, and from the positivity of $\epsilon$.
On the other hand, for all situations $t\not\sqsupseteq s$, it can easily be seen that the local variable $\martingale'(t\cdot)$ is equal to the constant $\martingale(s) + \epsilon \martingale^\ast(s) = \martingale'(t)$.
This constant is furthermore real because $\martingale{}'(t)$ is bounded below and $\martingale(s)+\epsilon\martingale^\ast(s)\leq\alpha+\epsilon$.
Hence, due to \ref{coherence: const is const}: $\lupprev{t}(\martingale'(t\cdot)) = \martingale'(t)$.
So we can conclude that $\martingale'$ is indeed a supermartingale and more specifically, a bounded below supermartingale.
We also have that $\liminf\martingale' \geq_s f$ because $\epsilon \martingale^\ast$ is non-negative and $\liminf\martingale \geq_s f$.
We will now show that, moreover, for all $\omega\in\Gamma(s)$, this process $\martingale'$ converges in $\extreals$.

For any $\omega\in\Gamma(s)$, if $\martingale$ does not converge in $\extreals$, $\martingale^\ast$ converges to $+\infty$ and hence also $\martingale'$ because $\martingale$ is bounded below and $\epsilon$ is positive.
If $\martingale$ does converge in $\extreals$, it converges either to a real number or to $+\infty$ (convergence to $-\infty$ is impossible because it is bounded below).
If $\martingale$ converges to a real number, $\martingale^\ast$ converges in $\extreals$ and hence $\martingale'$ also converges in $\extreals$.
If $\martingale$ converges to $+\infty$, then so does $\martingale'$ because $\epsilon \martingale^\ast$ is non-negative.
Hence, for all $\omega\in\Gamma(s)$, $\martingale'$ converges in $\extreals$ and the limit $\lim \martingale'(\omega)$ therefore exists.

Now, recall that $\lim \martingale'=\liminf\martingale' \geq_s f$ and that $\martingale\in\setofextsupmartb{}$.
Hence, we have that
\begin{align*}
\inf \left\{ \martingale(s) \colon \martingale\in\setofextsupmartb{} \text{ and } \lim \martingale \geq_s f \right\}
\leq \martingale'(s)
= \martingale(s)+\epsilon\martingale^\ast
\leq \alpha + \epsilon.
\end{align*}
This holds for any $\epsilon\in\posreals$ and any $\alpha > \upprevvovkk(f \vert s)$, which implies that indeed
\begin{align*}
\inf \left\{ \martingale(s) \colon \martingale\in\setofextsupmartb{} \text{ and } \lim \martingale \geq_s f \right\}\leq\upprevvovkk(f \vert s). \tag*{\qedhere}
\end{align*}
\end{proof}

The following result shows that the condition $\liminf \martingale{} \geq_s f$ in Definition~\ref{def:upperexpectation2} should in fact merely hold strictly almost surely:
\begin{proposition}\label{prop: MonotoneConvergenceSAS}
Consider any $f\in\setofextvariables$ and any $s\in\situations$.
Then
\begin{align}\label{Eq: prop: MonotoneConvergenceSAS 1}
	\upprevvovkk(f \vert s) &=\inf \Big\{ \martingale(s) :  \martingale\in\setofextsupmartb \text{ and } \liminf\martingale \geq_s f \text{ s.a.s.} \Big\}.
\end{align}
\end{proposition}
\begin{proof}
Since every supermartingale $\martingale$ that satisfies $\liminf\martingale \geq_s f$ also satisfies $\liminf\martingale \geq_s f$ s.a.s., we clearly have that
\begin{equation*}
\upprevvovkk(f \vert s) \geq\inf \Big\{ \martingale(s) \colon  \martingale\in\setofextsupmartb \text{ and } \liminf\martingale \geq_s f \text{ s.a.s.} \Big\},
\end{equation*}
so it remains to prove the other inequality.
If the right hand side of Equation \eqref{Eq: prop: MonotoneConvergenceSAS 1} is equal to $+\infty$, then this inequality is trivially satisfied.
So consider the case where it is not.
Fix any $\alpha\in\reals{}$ such that $\alpha >\inf \big\{\martingale(s) \colon \martingale\in\setofextsupmartb \text{ and } \liminf\martingale \geq_s f \text{ s.a.s.} \big\}$ and any $\epsilon\in\posreals$.
Then there is some bounded below supermartingale $\martingale_\alpha$ such that $\liminf\martingale_\alpha \geq_s f \text{ s.a.s.}$ and
\begin{equation}\label{Eq: prop: MonotoneConvergenceSAS 2}
\martingale_\alpha(s)\leq\alpha.
\end{equation}
Since $\liminf\martingale_\alpha \geq_s f$ s.a.s., there is some $s$-test supermartingale $\martingale_\alpha^\ast$ that converges to $+\infty$ on $A \coloneqq\{ \omega\in\Gamma(s) \colon \liminf\martingale_\alpha(\omega) < f(\omega)\}$.
Consider the extended real process $\martingale_\alpha + \epsilon \martingale_\alpha^\ast$.
This process is again a bounded below supermartingale because of Lemma~\ref{lemma:positive:countable:linear:combination} [which we can apply because $\martingale_\alpha$ and $\martingale_\alpha^\ast$ are both bounded below and hence have a common lower bound].
Since $\martingale_\alpha^\ast$ converges to $+\infty$ on $A$ and because $\martingale_\alpha$ is bounded below, we have that $\liminf (\martingale_\alpha + \epsilon \martingale_\alpha^\ast)(\omega)=+\infty \geq f(\omega)$ for all $\omega\in A$.
Moreover, for all $\omega\in\Gamma(s) \setminus A$, we also have that $\liminf (\martingale_\alpha + \epsilon \martingale_\alpha^\ast)(\omega) \geq f(\omega)$, because $\liminf\martingale_\alpha(\omega) \geq f(\omega)$ and because $\epsilon \martingale_\alpha^\ast$ is non-negative.
Hence, $\liminf (\martingale_\alpha + \epsilon \martingale_\alpha^\ast) \geq_s f$ and consequently $\upprevvovkk(f \vert s) \leq (\martingale_\alpha + \epsilon \martingale_\alpha^\ast)(s)$.
It therefore follows from Equation \eqref{Eq: prop: MonotoneConvergenceSAS 2} that
\begin{align*}
\upprevvovkk(f \vert s)
&\leq (\martingale_\alpha + \epsilon \martingale_\alpha^\ast)(s)
= \martingale_\alpha(s) + \epsilon\leq\alpha + \epsilon.
\end{align*}
As this holds for any $\epsilon\in\posreals{}$, we have that $\upprevvovkk(f \vert s)\leq\alpha$, and since this is true for every $\alpha\in\reals{}$ such that $\alpha >\inf \big\{\martingale(s) \colon \martingale\in\setofextsupmartb \text{ and } \liminf\martingale \geq_s f \text{ s.a.s.} \big\}$, it follows that
\begin{align*}
\upprevvovkk(f \vert s)
\leq\inf \Big\{ \martingale(s) :  \martingale\in\setofextsupmartb \text{ and } \liminf\martingale \geq_s f \text{ s.a.s.} \Big\}. \tag*{\qedhere}
\end{align*}
\end{proof}

Clearly, the infimum in~\eqref{Eq: prop: MonotoneConvergenceSAS 1} is taken over a larger set compared to the infimum in Definition~\ref{def:upperexpectation2}.
Though the resulting game-theoretic upper expectation is not impacted by this difference, it does make sure that the infimum in~\eqref{Eq: prop: MonotoneConvergenceSAS 1} becomes attained:
\begin{proposition}\label{prop: inf becomes min in definition}
For any $f\in\setofextvariablesb$ and any $s \in \situations{}$, the infimum in Equation \eqref{Eq: prop: MonotoneConvergenceSAS 1} is attained.
\end{proposition}
\begin{proof}
Consider any $f\in\setofextvariablesb$ and any $s \in \situations{}$.
Let $\process$ be the extended real process defined by $\process(t) \coloneqq\upprevvovkk(f \vert t)$ for all $t\in\situations$.
Then $\process$ is a bounded below supermartingale because of Corollary~\ref{corollary: upprevvovk is supermartingale}.
Moreover, because of Theorem \ref{Theorem: Levy}, we have that $\liminf\process \geq_s f$ strictly almost surely.
Since $\process(s)=\upprevvovkk(f \vert s)$, this concludes the proof.
\end{proof}

\section{Continuity of $\upprevvovkk$ with respect to monotone sequences}\label{Sect: Continuity}

We now turn to the most important subject in this paper: continuity properties of $\upprevvovkk$.
Apart from their intrinsic theoretical value, these properties also have great practical relevance, in the sense that they provide possible ways to calculate (upper) expectations that would otherwise be difficult or even impossible to calculate numerically.
For example, suppose that we have some variable $f\in\setofextvariables{}$ and some situation $s\in\situations{}$ for which calculating $\upprevvovkk{}(f\vert s)$ directly is not feasible.
If we can find a sequence of simpler functions $f_n$ that converges to $f$ in such a way that $\upprevvovkk{}$ is continuous with respect to this convergence, then we can use $\upprevvovkk{}(f_n \vert s)$ to approximate $\upprevvovkk{}(f\vert s)$ provided that $n$ is large enough.
If the sequence of functions $f_n$ is moreover simple enough such that all individual $\upprevvovkk{}(f_n \vert s)$ can be calculated directly, then we obtain a practical method for calculating $\upprevvovkk{}(f \vert s)$.

We start by establishing results similar to those presented in an earlier conference paper \cite{TJoens2018Continuity}, where we used a version of $\upprevvovkk{}$ with real supermartingales instead of extended real ones.
These results mainly concern continuity with respect to monotone sequences that are bounded below.
The first one shows that, similar to the local models $\lupprev{s}$, the global upper expectation $\upprevvovkk{}$ also satisfies continuity with respect to non-decreasing sequences that are bounded below.
The idea behind this result goes back to \cite[Theorem 6.6]{ShaferVovk2014ITIP}, but an updated version can now also be found in \cite[Proposition~8.3]{Vovk2019finance}.
Once more, the setting for which \cite[Proposition~8.3]{Vovk2019finance} is stated slightly differs from ours; the authors do not necessarily consider a finite state space, and we do not impose the additional axioms \ref{coherence: sublinearity extended}--\ref{coherence: monotonicity extended} on the local models.
Moreover, they only give an explicit proof for the case that there is a single, fixed local model $\lupprev{}$ in all situations.
For these reasons, we provide an independent proof here.



\begin{theorem}\label{theorem: upward convergence game}
Consider any $s\in\situations$ and any non-decreasing sequence $\{f_n\}_{n\in\natz}$ in $\setofextvariablesb$ that converges pointwise to a variable $f\in\setofextvariablesb$.
Then we have that
$\upprevvovkk(f \vert s)=\lim_{n\to+\infty} \upprevvovkk(f_n \vert s)$.
\end{theorem}
\begin{proof}
As $f_0\in\setofextvariablesb$ is bounded below and the sequence $\{f_n\}_{n\in\natz}$ is non-decreasing, there is an $M\in\reals{}$ such that $f_n \geq M$ for all $n\in\natz$ and therefore, $f$ is also bounded below by $M$.
Hence, since $\upprevvovkk$ is constant additive [\ref{vovk coherence 6}], we can assume without loss of generality that $f$ and all $f_n$ are non-negative.

That $\lim_{n\to+\infty} \smash{\upprevvovkk(f_n \vert s)}$ exists, follows from the non-decreasing character of $\{f_n\}_{n\in\natz}$ and \ref{vovk coherence 4}.
Moreover, we have that $\smash{\upprevvovkk(f \vert s)} \geq \lim_{n\to+\infty} \smash{\upprevvovkk(f_n \vert s)}$ because $f \geq f_n$ [since $\{f_n\}_{n\in\natz}$ is non-decreasing] and because $\upprevvovkk$ satisfies~\ref{vovk coherence 4}.
It remains to prove the converse inequality.

For any $n\in\natz$, consider the extended real process $S_n$, defined by $S_n(t) \coloneqq\upprevvovkk(f_n \vert t)$ for all $t\in\situations$ and the extended real process $S$ defined by the limit $S(t) \coloneqq\lim_{n\to+\infty}S_n(t)$ for all $t\in\situations$.
This limit exists because $\{S_n(t)\}_{n\in\natz}$ is a non-decreasing sequence for all $t\in\situations$, due to the monotonicity [\ref{vovk coherence 4}] of $\smash{\upprevvovkk}$.
As $f_n$ is non-negative for all $n\in\natz$, $S_n$ is non-negative for all $n\in\natz$ because of \ref{vovk coherence 5} and therefore $S$ is also non-negative.
As a result, $S$ and all $S_n$ are non-negative extended real processes.

It now suffices to prove that $S$ is a bounded below supermartingale such that $\liminf S \geq_s f$ s.a.s. because it will then follow from Proposition~\ref{prop: MonotoneConvergenceSAS} that
\begin{align*}
	\upprevvovkk(f \vert s)
	&=\inf \Big\{ \martingale(s) :  \martingale\in\setofextsupmartb \text{ and } \liminf\martingale \geq_s f \text{ s.a.s.} \Big\} \\
	&\leq S(s)=\lim_{n\to+\infty}\upprevvovkk(f_n \vert s).
\end{align*}
This is what we now set out to do.

We first show that $S$ is a supermartingale; that it is bounded below follows trivially from its non-negativity.
For all situations $t\in\situations$, we already know that
$\{S_n(t\andstate)\}_{n\in\natz}$ is a non-decreasing sequence that converges to $S(t\andstate)$.
Since $S_n$ and $S$ are non-negative, we also have that $S_n(t\andstate), S(t\andstate)\in\smash{\setofgenextvariablesb(\statespace)}$.
Then, due to \ref{coherence: continuity wrt bounded below functions}, we have that
\begin{align}\label{Eq: upward monotone game}
\lupprev{t}(S(t\andstate))=\lim_{n\to+\infty} \lupprev{t}(S_n(t\andstate)) \text{ for all } t\in\situations.
\end{align}
$S_n$ is a supermartingale for all $n\in\natz$ because of Corollary~\ref{corollary: upprevvovk is supermartingale}, so it follows that $\lupprev{t}(S_n(t\andstate)) \leq S_n(t)$ for all $n\in\natz$ and all $t\in\situations$.
This implies, together with Equation \eqref{Eq: upward monotone game}, that
\begin{align*}
 \lupprev{t}(S(t\andstate))\leq\lim_{n\to+\infty} S_n(t)=S(t) \text{ for all } t\in\situations.
\end{align*}
Hence, $S$ is a supermartingale.

To prove that $\liminf S \geq_s f$ s.a.s., we will use L\'evy's zero-one law.
It follows from Theorem~\ref{Theorem: Levy} that, for all $n\in\natz$, there is an $s$-test supermartingale $\martingale_n$ that converges to $+\infty$ on the event
\begin{align*} \label{eq: proof upward convergence eq1}
	A_n \coloneqq\Big\{ \omega\in\Gamma(s) \colon \liminf_{m \to +\infty} \upprevvovkk(f_n \vert \omega^m) &< f_n(\omega) \Big\}.
\end{align*}
Now, consider the extended real process $\martingale$, defined by
\begin{equation*}
	\martingale(t) \coloneqq\sum_{n\in\natz} \lambda_n \martingale_n(t) \text{ for all } t\in\situations,
\end{equation*}
where the coefficients $\lambda_n > 0$ sum to $1$.
Then it follows from Lemma~\ref{lemma:positive:countable:linear:combination} that $\martingale$ is again a non-negative supermartingale.
Moreover, it is clear that $\martingale(s)=1$ and hence, $\martingale$ is an $s$-test supermartingale.


We show that $\martingale$ converges to $+\infty$ on all paths $\omega\in\Gamma(s)$ such that $\liminf_{m \to +\infty} S(\omega^m) < f (\omega)$.
Clearly, $\martingale$ converges to $+\infty$ on $\cup_{n\in\natz} A_n \eqqcolon A$.
Consider now any path $\omega\in\Gamma(s)$ for which $\liminf_{m \to +\infty} S(\omega^m) < f (\omega)$.
As explained before, $S_n(t)$ is non-decreasing in $n$ for all $t\in\situations$, so we have that $\sup_{n\in\natz} S_n(\omega^m)=\lim_{n\to+\infty} S_n(\omega^m)=S(\omega^m)$ for all $m\in\natz$. Since $\liminf_{m \to +\infty} S(\omega^m) < f (\omega)$, this implies that
\begin{align*}
 \liminf_{m \to +\infty} \sup_{n\in\natz} S_n(\omega^m)
 < \lim_{n\to+\infty} f_n(\omega).
\end{align*}
Since $\sup_{n\in\natz} \liminf_{m \to +\infty} S_n(\omega^m)\leq\liminf_{m \to +\infty} \sup_{n\in\natz} S_n(\omega^m)$ [because we obviously have that $S_n(\omega^m)\leq \sup_{n\in\natz} S_n(\omega^m)$ for all $n,m \in\natz$], this implies that
\begin{align}\label{eq: proof upward convergence eq2}
	\sup_{n\in\natz} \liminf_{m \to +\infty} \upprevvovkk(f_n \vert \omega^m)
	= \sup_{n\in\natz} \liminf_{m \to +\infty} S_n(\omega^m)
	\leq \liminf_{m \to +\infty} \sup_{n\in\natz} S_n(\omega^m)
	&< \lim_{n\to+\infty} f_n(\omega).
\end{align}
Hence, there is some $n_\omega \in \natz$ such that
\begin{equation*}
\sup_{n\in\natz} \liminf_{m \to +\infty} \upprevvovkk(f_n \vert \omega^m) <  f_{n_\omega}(\omega),
\end{equation*}
and therefore, we see that also
\begin{equation*}
\liminf_{m \to +\infty} \upprevvovkk(f_{n_\omega} \vert \omega^m) <  f_{n_\omega}(\omega).
\end{equation*}
So $\omega\in A_{n_\omega} \subseteq A$ and, as a consequence, $\martingale$ converges to $+\infty$ on $\omega$.
Hence, the $s$-test supermartingale $\martingale$ converges to $+\infty$ on all paths $\omega\in\Gamma(s)$ such that $\liminf_{m \to +\infty} S(\omega^m) < f (\omega)$, and therefore $\liminf S \geq_s f$ strictly almost surely.
\end{proof}

The fact that $\upprevvovkk{}$ satisfies continuity with respect to non-decreasing sequences, together with the properties in Proposition~\ref{Prop: Coherence of Global Game-theoretic upper expectation}, immediately implies that $\upprevvovkk{}$ is an upper expectation.

\begin{corollary}\label{corollary: Vovk is an upper expectation}
For any $s\in\situations$, the map $\upprevvovkk(\cdot \vert s) \colon \setofextvariables \to \extreals$ is an upper expectation on $\setofextvariables$.
\end{corollary}
\begin{proof}
 \ref{coherence: const is const} follows from \ref{vovk coherence 5}.
Properties~\ref{coherence: sublinearity}, \ref{coherence: homog for positive lambda} and \ref{coherence: monotonicity} respectively follow from \ref{vovk coherence 2}, \ref{vovk coherence 3} and \ref{vovk coherence 4}.
Finally, \ref{coherence: continuity wrt non-negative functions} follows from Theorem~\ref{theorem: upward convergence game}.
\end{proof}

Another immediate consequence of Theorem~\ref{theorem: upward convergence game} is that $\upprevvovkk{}$ satisfies a version of Fatou's lemma:

\begin{corollary}[Fatou's Lemma]\label{corollary: Fatou general}
For any situation $s\in\situations$ and any sequence $\{f_n\}_{n\in\natz}$ in $\setofextvariablesb$ that is uniformly bounded below, we have that
$\upprevvovkk(f \vert s)\leq\liminf_{n\to+\infty} \upprevvovkk(f_n \vert s)$ where $f \coloneqq\liminf_{n\to+\infty} f_n$.
\end{corollary}
\begin{proof}
Consider any $s\in\situations$ and any sequence $\{f_n\}_{n\in\natz}$ in $\setofextvariablesb$ that is uniformly bounded below.
For all $k\in\natz$, let $g_k$ be the global variable defined by $g_k(\omega) \coloneqq\inf_{n \geq k}{f_n}(\omega)$ for all $\omega\in\Omega$.
Then $\lim_{k \to +\infty}{g_k} = \liminf_{n\to+\infty} f_n \eqqcolon f$.
Furthermore, $\{g_k\}_{k\in\natz}$ is clearly non-decreasing and it is a sequence in $\setofextvariablesb$ because $\{f_n\}_{n\in\natz}$ is uniformly bounded below.
Hence, we can use Theorem \ref{theorem: upward convergence game} to find that
\begin{equation*}
\upprevvovkk(f \vert s)=\lim_{k \to +\infty} \upprevvovkk(g_k \vert s)=\liminf_{k \to +\infty} \upprevvovkk(g_k \vert s) \leq  \liminf_{k \to +\infty} \upprevvovkk(f_k \vert s),
\end{equation*}
where the inequality holds because, for all $k\in\natz$, $g_k \leq f_k$ and therefore, because of \ref{vovk coherence 4}, also $\upprevvovkk(g_k \vert s)\leq\upprevvovkk(f_k \vert s)$.
\end{proof}

The following result is entirely new and states that $\upprevvovkk{}$ is continuous with respect to non-increasing sequences of lower cuts.
So $\upprevvovkk{}(\cdot \vert s)$ satisfies \ref{ext coherence lower cuts continuity} on $\setofextvariables{}$ for any $s \in \situations{}$.

\begin{proposition}\label{Proposition: Continuity w.r.t. lower cuts}
For any $f\in\setofextvariables$ and any $s\in\situations$, we have that \/
$\lim_{\alpha \to \, -\infty} \upprevvovkk(f^{\vee \alpha} \vert s)=\upprevvovkk(f \vert s)$.
\end{proposition}
\begin{proof}
$\upprevvovkk(f^{\vee \alpha} \vert s)$ is non-decreasing in $\alpha$ because $f^{\vee \alpha}$ is non-decreasing in $\alpha$ and because $\upprevvovkk$ is monotone [\ref{vovk coherence 4}], and therefore $\lim_{\alpha \to \, -\infty} \upprevvovkk(f^{\vee \alpha} \vert s)$ exists.
Moreover, $f^{\vee \alpha} \geq f$ for all $\alpha\in\reals{}$, implying, by the monotonicity [\ref{vovk coherence 4}] of $\upprevvovkk$, that $\lim_{\alpha \to \, -\infty} \upprevvovkk(f^{\vee \alpha} \vert s) \geq \upprevvovkk(f \vert s)$.
It therefore only remains to prove the converse inequality.

If $\upprevvovkk(f \vert s)=+\infty$, then $\lim_{\alpha \to \, -\infty} \upprevvovkk(f^{\vee \alpha} \vert s)\leq\upprevvovkk(f \vert s)$ holds trivially.
If $\upprevvovkk(f \vert s) < +\infty$, fix any real $c > \upprevvovkk(f \vert s)$.
Then it follows from the definition of $\upprevvovkk(f \vert s)$ that there is some supermartingale $\martingale\in\setofextsupmartb$ such that $\martingale(s) \leq c$ and $\liminf\martingale \geq_s f$.
Since $\martingale$ is bounded below, it immediately follows that there is some $B \in\reals{}$ such that $\liminf\martingale \geq \alpha$ for all $\alpha\leq B$.
For such an $\alpha\leq B$, we have that $\liminf\martingale \geq_s f^{\vee \alpha}$, which by Definition~\ref{def:upperexpectation2} implies that $\upprevvovkk(f^{\vee \alpha} \vert s)\leq\martingale(s) \leq c$.
This holds for all $\alpha\leq B$, so we infer that $\lim_{\alpha \to \, -\infty} \upprevvovkk(f^{\vee \alpha} \vert s) \leq c$, and since this holds for any $c > \upprevvovkk(f \vert s)$, we conclude that indeed $\lim_{\alpha \to \, - \infty} \upprevvovkk(f^{\vee \alpha} \vert s)\leq\upprevvovkk(f \vert s)$.
\end{proof}

Proposition~\ref{Proposition: Continuity w.r.t. lower cuts} shows that $\upprevvovkk{}$ on $\setofextvariables{} \times \situations{}$ is uniquely determined by its values on $\setofextvariablesb{} \times \situations{}$.
Moreover, by Theorem~\ref{theorem: upward convergence game}, its values on $\setofextvariablesb{} \times \situations{}$ are on their turn uniquely determined by its values on $\setofgambles{} \times \situations{}$ (for any $f\in \setofextvariablesb{}$, there obviously is at least one sequence in $\setofgambles{}$ that converges non-decreasingly to $f$).
Together, these observations imply that $\upprevvovkk{}$ is uniquely determined by its values on the gambles in $\setofextvariables{}$.

Proposition~\ref{Proposition: Continuity w.r.t. lower cuts} also immediately establishes our claim from Section~\ref{Sect: Upper Expectations}, that the (global) game-theoretic upper expectation is only fully compatible with the local models if the local models additionally satisfy \ref{ext coherence lower cuts continuity}:

\begin{corollary}[Compatibility with local models]\label{corollary: complete compatibility}
For any imprecise probabilities tree \/ $\lupprev{}$, we have that  \/
		$\upprevvovkk (f \vert \sit)=\lupprev{\sit}(f(\sit \cdot))$
for all $\sit\in\situations$ and all \/ $(n+1)$-measurable variables $f\in\setofextvariables{}$, if and only if the local upper expectations \/ $\lupprev{s}$ all satisfy \ref{ext coherence lower cuts continuity} on $\setofgenextvariables{}(\statespace{})$.
\end{corollary}
\begin{proof}
Consider any $\sit\in\situations$ and any $(n+1)$-measurable extended real variable $f \in \setofextvariables{}$.
Clearly, $f^{\vee c}$ is bounded below and remains to be $(n+1)$-measurable for any $c \in \reals{}$.
Due to Proposition~\ref{prop: global compatible with local}, we have that $\upprevvovkk(f^{\vee c} \vert \sit)=\lupprev{\sit}(f^{\vee c}(\sit \cdot))$ for any $c \in \reals{}$.
Then, because $\upprevvovkk(\cdot\vert\sit)$ satisfies \ref{ext coherence lower cuts continuity} due to Proposition~\ref{Proposition: Continuity w.r.t. lower cuts}, we clearly also have that $\upprevvovkk (f \vert \sit)=\lupprev{\sit}(f(\sit \cdot))$ if $\lupprev{\sit}$ satisfies \ref{ext coherence lower cuts continuity}.

On the other hand, consider the case that there is some $\sit\in\situations{}$ such that $\lupprev{\sit}$ does not satisfy \ref{ext coherence lower cuts continuity}, meaning that there is at least a single $f\in\setofgenextvariables{}(\statespace{})$ such that $\lupprev{\sit}(f) \not= \lim_{c\to \, -\infty} \lupprev{\sit}(f^{\vee c})$ [the limit on the right hand side still exists because of monotonicity (\ref{coherence: monotonicity})].
Let $g\in\setofextvariables{}$ be an $(n+1)$-measurable variable such that $g(\sit \cdot) = f$.
Then, for all $c\in\reals{}$, $g^{\vee c}$ is bounded below and $(n+1)$-measurable and we clearly also have that $g^{\vee c}(\sit \cdot) = f^{\vee c}$.
Hence,
\begin{align*}
\lupprev{\sit}(g(\sit \cdot))
= \lupprev{\sit}(f)
\not= \lim_{c\to \, -\infty} \lupprev{\sit}(f^{\vee c})
= \lim_{c\to \, -\infty} \lupprev{\sit}(g^{\vee c}(\sit \cdot))
= \lim_{c\to \, -\infty} \upprevvovkk{}(g^{\vee c} \vert \sit)
= \upprevvovkk{}(g\vert \sit),
\end{align*}
where the second to last step follows from Proposition~\ref{prop: global compatible with local} and the last step follows from Proposition~\ref{Proposition: Continuity w.r.t. lower cuts}.
\end{proof}

As we have illustrated in Example~\ref{example: Vovk's axioms do not necessarily satisfy lower cuts continuity}, an upper expectation---and therefore also any local upper expectation---does not necessarily satisfy \ref{ext coherence lower cuts continuity}.
Hence, if we find compatibility (on the entire domain) between local and global models desirable, then we will need to impose \ref{ext coherence lower cuts continuity} as an additional axiom on the local models $\lupprev{s}$.
This is also the case even when we adopt the stronger axioms \ref{coherence: sublinearity extended}--\ref{coherence: monotonicity extended} (again, this follows from Example~\ref{example: Vovk's axioms do not necessarily satisfy lower cuts continuity}).
We therefore do not see what advantage could be gained from additionally imposing \ref{coherence: sublinearity extended}--\ref{coherence: monotonicity extended} on the local models:
it imposes constraints on how the values of the local models can be chosen, yet does not imply any additional properties for the global game-theoretic upper expectation
($\upprevvovkk{}$ only depends on the values of the local models on $\setofgenextvariablesb{}(\statespace{})$ and, on this restricted domain, axioms~\ref{coherence: sublinearity extended}--\ref{coherence: monotonicity extended} are as general as \ref{coherence: sublinearity}--\ref{coherence: monotonicity} according to Proposition~\ref{prop: extension with E6 is an extension}).
If one moreover finds it desirable to have complete compatibility with the local models, then one should additionally impose \ref{ext coherence lower cuts continuity} on the local models, irrespectively of whether he is already imposing \ref{coherence: sublinearity extended}--\ref{coherence: monotonicity extended}, at which point \ref{coherence: sublinearity extended}--\ref{coherence: monotonicity extended} become redundant due to Proposition~\ref{prop: local models with lower cuts satisfy vovk extended axioms}.

\section{Behaviour of $\upprevvovkk$ with respect to sequences of finitary variables}\label{Sect: Continuity wrt n-measurables}

Even though the upper expectation $\upprevvovkk{}$ is continuous with respect to non-decreasing sequences and with respect to non-increasing sequences of lower cuts, it is not necessarily continuous with respect to general pointwise convergence; see \cite[Example 1]{TJoens2018Continuity} for an illustration.\footnote{The version of the game-theoretic upper expectation used in \cite[Example 1]{TJoens2018Continuity} is with real-valued supermartingales instead of extended real-valued ones.
However, this does not make any difference because \cite[Example 1]{TJoens2018Continuity} only involves upper expectations of gambles and, on that domain, both versions of the game-theoretic upper expectation coincide; see Proposition~\ref{prop: two versions of vovk upprev are equal} further on.}
However, in many cases, the sequence of interest will be composed out of variables that are finitary; the individual variables then only depend on a finite number of states.
Compared to general sequences, sequences of finitary variables tend to be more well-behaved and therefore allow us to establish stronger continuity results.
Moreover, upper expectations of finitary variables or, more specifically, finitary gambles can often be computed rather efficiently; see \cite{8629186}.
If these computational methods were to be combined with the appropriate continuity properties---which tend to be stronger for sequences of finitary variables--we would also be able to compute upper expectations of a great deal of non-finitary variables.

Because of their importance in this section, we will denote the set of all bounded below variables $f \in \setofextvariablesb{}$ such that $f=\lim_{n\to+\infty}f_n$ for some sequence $\{f_n\}_{n\in\natz}$ of finitary variables by $\setoflimitsoffinmeasb{}$.
Our first result indicates that the variables in $\setoflimitsoffinmeasb$ play a crucial role in the characterisation of game-theoretic upper expectations.
It states that the upper expectation $\upprevvovkk{}(f \vert s)$ of any variable $f \in \setofextvariables{}$, conditional on any $s \in \situations{}$, is the lower envelope of the upper expectations $\upprevvovkk{}(g \vert s)$ of variables $g\in\setoflimitsoffinmeasb$ that dominate $f$ on the domain $\Gamma(s)$.
Once more, this result allows us to conclude that $\upprevvovkk{}$ is uniquely characterised by its values on a constrained domain; that of the (bounded below) limits of finitary variables.
We refer to \cite{Tjoens2019NaturalExtensionISIPTA,TJOENS202130} for a more detailed discussion.

\begin{proposition}\label{Prop: limits to general}
Consider any $f\in\setofextvariables$ and any $s\in\situations$.
Then
\begin{align}
	\upprevvovkk (f \vert s) &=\inf{\Big\{ \upprevvovkk(g \vert s) \colon g\in\setoflimitsoffinmeasb{} \text{ and }  g \geq_s f  \Big\}} \label{Eq: Prop: limits to general 1}
	= \inf{\Big\{ \upprevvovkk(g \vert s) \colon g\in\setoflimitsoffinmeasb{} \text{ and }  g \geq f  \Big\}}.
\end{align}
\end{proposition}
\begin{proof}
Because $\upprevvovkk$ is monotone [\ref{vovk coherence 4}], we have that $\upprevvovkk(f \vert s)\leq\upprevvovkk(g \vert s)$ for any $g\in\setoflimitsoffinmeasb{}$ such that $f \leq_s g$.
It therefore follows immediately that
\begin{align*}
\upprevvovkk (f \vert s) \leq\inf{\Big\{ \upprevvovkk(g \vert s) \colon g\in\setoflimitsoffinmeasb{} \text{ and }  g \geq_s f  \Big\}}
\leq\inf{\Big\{ \upprevvovkk(g \vert s) \colon g\in\setoflimitsoffinmeasb{} \text{ and }  g \geq f  \Big\}},
\end{align*}
where the last inequality follows from the fact that $g\geq f$ implies $g\geq_s f$ for any $g\in\setofextvariables{}$.
It remains to prove that $\inf{\Big\{ \upprevvovkk(g \vert s) \colon g\in\setoflimitsoffinmeasb{} \text{ and }  g \geq f  \Big\}}\leq\upprevvovkk(f\vert s)$.

Consider any $\martingale'\in\setofextsupmartb{}$ such that $\lim \martingale'$ exists within $\Gamma(s)$ and such that $\lim \martingale' \geq_s f$.
Let $\martingale{}$ be the extended real process defined by $\martingale{}(t)\coloneqq\martingale{}'(t)$ for all $t\sqsupseteq s$, and by $\martingale(t)\coloneqq+\infty$ for all $t \not\sqsupseteq s$.
We show that $\martingale{}$ is a bounded below supermartingale such that $\lim \martingale{} \geq f$.
The process $\martingale{}$ is bounded below because $\martingale{}'$ is bounded below.
Moreover, we have, for all $t\sqsupseteq s$, that $\lupprev{t}(\martingale{}(t\cdot))=\lupprev{t}(\martingale{}'(t\cdot))\leq\martingale{}'(t)=\martingale{}(t)$ because $\martingale{}'$ is a supermartingale, and, for all $t \not\sqsupseteq s$, we also have that $\lupprev{t}(\martingale{}(t\cdot))\leq\martingale{}(t)$ because then $\martingale{}(t)=+\infty$.
Hence, $\martingale{}$ is indeed a bounded below supermartingale.
Furthermore, note that $\lim\martingale=_s\lim\martingale'\geq_s f$ and, for any path $\omega$ not going through $s$, that $\lim\martingale{}(\omega)=+\infty\geq f(\omega)$, which all together implies that $\lim\martingale{}\geq f$.

Now, let $\{g_n\}_{n \in \natz{}}$ be the sequence defined by $g_n(\omega) \coloneqq \martingale{}(\omega^n)$ for all $n \in \natz{}$ and all $\omega \in \samplespace{}$.
Then it is clear that $\{g_n\}_{n \in \natz{}}$ is a sequence of $n$-measurable, and therefore finitary, extended real variables that is uniformly bounded below.
Moreover, since $\lim \martingale{}$ exists everywhere, we have that $g(\omega) \coloneqq \lim_{n \to +\infty} g_n(\omega) = \lim_{n \to +\infty} \martingale{}(\omega^n)$ exists for all $\omega \in \samplespace{}$.
Hence, $g\in\setoflimitsoffinmeasb{}$ and because $\lim \martingale \geq f$ also $g \geq f$.
It furthermore follows from Definition \ref{def:upperexpectation2} that $\upprevvovkk(g \vert s)\leq\martingale(s)$ because $\lim \martingale \geq_s g$ (since, in fact, $\lim \martingale = g$).
This implies that
\begin{align*}
\inf{\Big\{ \upprevvovkk(g \vert s) \colon g\in\setoflimitsoffinmeasb{} \text{ and }  g \geq f  \Big\}}\leq\martingale(s)=\martingale'(s).
\end{align*}
Since this holds for any $\martingale'\in\setofextsupmartb{}$ such that $\lim \martingale'$ exists within $\Gamma(s)$ and $\lim \martingale' \geq_s f$, it follows from Proposition~\ref{prop: liminf can be replaced by lim in definition} that
\begin{align*}
\inf{\Big\{ \upprevvovkk(g \vert s) \colon g\in\setoflimitsoffinmeasb{} \text{ and }  g \geq f  \Big\}}\leq\upprevvovkk (f \vert s).
\tag*{\qedhere}
\end{align*}
\end{proof}

The result above is expressed in terms of (bounded below) limits of finitary variables, but we could just as well have replaced them by (bounded below) limits of $n$-measurable gambles because, as we intend to prove next, each $f\in\setoflimitsoffinmeasb{}$ is a pointwise limit of a sequence of $n$-measurable gambles.
As a first step to establish this result, we aim to transition from any sequence $\{g_n\}_{n\in\natz}$ of finitary variables to a sequence $\{g^\xi_n\}_{n\in\natz}$ of $n$-measurable variables, without changing the essential characteristics of the sequence.
To this end, let $\{g^\xi_n\}_{n\in\natz}$ and $\xi\colon\nats\to\natz$ be defined by the following recursive expressions, where $\smash{g^\xi_0\coloneqq c\in\extreals}$ is a freely chosen extended real number and $\xi(1)\coloneqq0$:
\begin{align}\label{Eq: lemma: limits of n-measurables are equal to limits of finitary}
g^\xi_n \coloneqq
\begin{cases}
g_{\xi(n)} &\text{ if $g_{\xi(n)}$ is $n$-measurable;} \\
g^\xi_{n-1} &\text{ otherwise,}
\end{cases}
\qquad \text{ and } \qquad
\xi(n+1) \coloneqq
\begin{cases}
\xi(n)+1 &\text{ if $g_{\xi(n)}$ is $n$-measurable;}\\
\xi(n) &\text{ otherwise,}
\end{cases}
\end{align}
for all $n\in\nats$.
Note that $\{g_n\}_{n\in\natz}$ is a subsequence of $\{g^\xi_n\}_{n\in\natz}$ and that the additional terms in $\{g^\xi_n\}_{n\in\natz}$ do not impact the limit behaviour, nor, for a suitable choice of $c$, the (possibly) monotone character of the original sequence.

\begin{lemma}\label{lemma: limits of n-measurables are equal to limits of finitary}
For any sequence $\{g_n\}_{n\in\natz}$ of finitary variables, the sequence $\{g^\xi_n\}_{n\in\natz}$ is a sequence of $n$-measurable variables.
\end{lemma}
\begin{proof}
We prove this by induction.
$g^\xi_0=c$ is clearly $0$-measurable.
To prove the induction step, suppose that $g^\xi_{k-1}$ is $(k-1)$-measurable for some $k\in\nats$.
Then either we have that $g_{\xi(k)}$ is $k$-measurable, which directly implies that $g^\xi_{k}=g_{\xi(k)}$ is $k$-measurable.
Otherwise, $g^\xi_k$ is equal to $g^\xi_{k-1}$ implying that $g^\xi_k$ is $(k-1)$-measurable and therefore automatically $k$-measurable.
This concludes the induction step and hence, $\{g^\xi_n\}_{n\in\natz}$ is a sequence of $n$-measurable variables.
\end{proof}

We now establish our earlier claim that the variables in $\setoflimitsoffinmeasb{}$ are essentially limits of $n$-measurable gambles and therefore, that Proposition~\ref{Prop: limits to general} turns out to be stronger than it first appears.

\begin{proposition}\label{proposition: limits of n-measurables are equal to limits of finitary}
For any $f\in\setoflimitsoffinmeasb{}$, $f$ is the pointwise limit of a sequence $\{f_n\}_{n\in\natz}$ of $n$-measurable gambles.
Furthermore, we can guarantee that $B \leq f_n \leq \sup f$ for all $n\in\natz$, where $B$ is any real if \/ $\inf f = +\infty$ and $B=\inf f$ otherwise.
\end{proposition}
\begin{proof}
Fix any $f\in\setoflimitsoffinmeasb$.
Then, according to the definition of $\setoflimitsoffinmeasb$, $f$ is the pointwise limit of a sequence $\{g_n\}_{n\in\natz}$ of finitary variables.
Let $\{g^\xi_n\}_{n\in\natz}$ be the sequence defined by the recursive expression~\eqref{Eq: lemma: limits of n-measurables are equal to limits of finitary}, with $c=0$, which by Lemma~\ref{lemma: limits of n-measurables are equal to limits of finitary} is a sequence of $n$-measurable variables.
As explained in the text that preceeds Lemma~\ref{lemma: limits of n-measurables are equal to limits of finitary}, the sequences $\{g_n\}_{n\in\natz}$ and $\{g^\xi_n\}_{n\in\natz}$ have the same limit behaviour, so $\{g^\xi_n\}_{n\in\natz}$ converges pointwise to $f$.
Let $B$ be any real if $\inf f = +\infty$ and let $B \coloneqq \inf f$ if $\inf f \in \reals$ [the case where $\inf f=-\infty$ is impossible because $f$ is bounded below].
Let $\{f_n\}_{n\in\natz}$ be the sequence defined by bounding each $g^\xi_n$ above by $\min\{n,\sup f\}$ and below by $B$; so $f_n(\omega) \coloneqq \max\{ \min\{g^\xi_n(\omega), n, \sup f \}, B \}$ for all $\omega \in \samplespace{}$ and all $n\in\natz$.
Then it is clear that $\{f_n\}_{n\in\natz}$ is a sequence of $n$-measurable gambles because $\{g^\xi_n\}_{n \in \natz{}}$ is a sequence of $n$-measurable (possibly extended real) variables.
It also converges pointwise to $f$ because
\begin{align*}
f(\omega)
= \max\Big\{ \min\{f(\omega),\sup f\}, B \Big\}
&=  \max\Big\{ \min\big\{ \lim_{n \to +\infty} g^\xi_n(\omega), \lim_{n \to +\infty} n, \sup f \big\}, B \Big\} \\
&= \lim_{n \to +\infty} \max\Big\{ \min\{g^\xi_n(\omega), n, \sup f\} , B \Big\}
= \lim_{n \to +\infty} f_n(\omega),
\end{align*}
for all $\omega \in \samplespace{}$, where the first equality follows from the fact that $B \leq \inf f \leq f$.
Moreover, for all $n\in\natz$, we clearly have that $B\leq f_n$, and also $f_n\leq\sup f$ because $\min\{g^\xi_n(\omega), n, \sup f \} \leq \sup f$ for all $\omega\in\Omega$ and $B\leq \inf f \leq \sup f$.
Hence, $\{f_n\}_{n\in\natz}$ satisfies all of the conditions in the proposition.
\end{proof}


We now present two main results.
The first one states that $\upprevvovkk{}$ is continuous with respect to non-increasing sequences of finitary bounded above variables.
The second one says that, for any $f \in \setoflimitsoffinmeasb{}$, there is always a sequence of $n$-measurable gambles---and therefore also a sequence of finitary gambles---that converges pointwise to $f$ and for which $\upprevvovkk{}$ is continuous.

\begin{theorem}\label{Theorem: cont. wrt non-increasing n-measurables}
For any $s\in\situations$ and any non-increasing sequence $\smash{\{f_n\}_{n\in\natz}}$ of finitary, bounded above variables that converges pointwise to a variable $f\in\setofextvariables$, we have that
$\lim_{n\to+\infty} \smash{\upprevvovkk(f_n \vert s)=\upprevvovkk(f \vert s)}$.
\end{theorem}

\begin{theorem}\label{Theorem: limit of n-meas}
For any $s\in\situations$ and any $f\in\setoflimitsoffinmeasb{}$, there is a sequence $\smash{\{f_n\}_{n\in\natz}}$ of $n$-measurable gambles that is uniformly bounded below and that converges pointwise to $f$ such that \/ $\lim_{n\to+\infty} \smash{\upprevvovkk(f_n \vert s)=\upprevvovkk(f \vert s)}$.
\end{theorem}

Both of the results above have already led to valuable theoretical insights in the literature: Theorem~\ref{Theorem: cont. wrt non-increasing n-measurables} was crucial to obtain an equivalence result about hitting times and probabilities in imprecise Markov chains \cite{8627473}.
Theorem~\ref{Theorem: limit of n-meas} on the other hand, further establishes the importance of finitary variables and their limits when it comes to characterising $\upprevvovkk{}$ \cite{Tjoens2019NaturalExtensionISIPTA,TJOENS202130}.
In fact, Theorem~\ref{Theorem: limit of n-meas} was a key result in obtaining our alternative characterisation presented in \cite{Tjoens2019NaturalExtensionISIPTA} (or, more recently, the characterisation presented in \cite{TJOENS202130}).
For this reason, a version of Theorem~\ref{Theorem: limit of n-meas} was also already included in an online report \cite{Tjoens2019ContinuityARXIV} that serves as a technical reference for \cite{Tjoens2019NaturalExtensionISIPTA}.



The remainder of this section is devoted to the proofs of Theorem~\ref{Theorem: cont. wrt non-increasing n-measurables} and~\ref{Theorem: limit of n-meas}.
We start with two technical lemmas that are key in proving them both.

\begin{lemma}\label{Lemma: convergence of n-measurables}
Consider any global variable $h\in\setofextvariables$ taking values in $\natz{}$.
If\/ $h(\omega)=h(\tilde{\omega})$ for every $\omega\in\samplespace$ and $\tilde{\omega}\in\Gamma(\omega^{h(\omega)})$, then $h$ is bounded below and above---and therefore a gamble---and $h$ is\/ $(\sup h)$-measurable, with $\sup h \coloneqq \sup_{\omega\in\Omega} h(\omega)\in\natz{}$.
\end{lemma}
\begin{proof}
$h$ is clearly bounded below because it takes value in $\natz$.
Assume \emph{ex absurdo} that $h$ is not bounded above; so $\sup h =+\infty$.
Then we have that
\begin{equation*}
 \sup_{\omega\in\samplespace} h(\omega)=\sup_{x_{1}\in\statespace} \sup_{\omega\in\Gamma(x_{1})} h(\omega)=+\infty.
\end{equation*}
Since $\statespace$ is finite, there is clearly some $x_{1}^\ast\in\statespace$ for which $\sup_{\omega\in\Gamma(x_{1}^\ast)} h(\omega)=+\infty$.
Similarly, we also find that
\begin{equation*}
 \sup_{\omega\in\Gamma(x_{1}^\ast)} h(\omega)=\sup_{x_{2}\in\statespace} \sup_{\omega\in\Gamma(x_{1}^\ast x_{2})} h(\omega)=+\infty.
\end{equation*}
Since $\statespace$ is finite, there is again some $x_{2}^\ast\in\statespace$ for which $\sup_{\omega\in\Gamma(x_{1}^\ast x_{2}^\ast)} h(\omega)=+\infty$.
We can continue in this way and construct a path $\omega=x_{1}^\ast x_{2}^\ast ... x_{n}^\ast ...$ for which
\begin{equation}\label{Eq: Lemma: convergence of n-measurables}
	\sup_{\omega'\in\Gamma(\omega^n)} h(\omega')=+\infty \text{ for all } n\in\natz.
\end{equation}
However, $h$ takes values in $\natz$, so $h(\omega)\in\natz$.
This implies, together with the assumption that $h(\omega')=h(\omega)$ for every $\omega'\in\Gamma(\omega^{h(\omega)})$, that
\begin{equation*}
	\sup \set*{h(\omega') \colon \omega'\in\Gamma\big(\omega^{h(\omega)}\big)}
	= \sup \set*{h(\omega) \colon \omega'\in\Gamma\big(\omega^{h(\omega)}\big)}=h(\omega) \in\natz.
\end{equation*}
This is in contradiction with Equation~\eqref{Eq: Lemma: convergence of n-measurables} [for $n=h(\omega)$], so $h$ is bounded above, which together with the fact that $h$ is bounded below, implies that $h$ is a gamble.
The fact that $h$ is bounded above and that it takes values in $\natz{}$, also clearly implies that $\sup h \in\natz{}$.
To see that $h$ is $(\sup h)$-measurable, consider any $\omega\in\Omega$ and any $\tilde{\omega}\in\Gamma(\omega^{\sup h})$.
Then $\tilde{\omega}\in\Gamma(\omega^{h(\omega)})$ because $h(\omega)\leq\sup h$ and therefore, by assumption, we have that $h(\omega) =  h(\tilde{\omega})$.
\end{proof}

For the following technical lemma, we will associate with any sequence $\{f_n\}_{n\in\natz}$ of $n$-measurable gambles and any global variable $h\in\setofextvariables$ taking values in $\natz$, the global variable $f_{h}\in\setofextvariables$ defined by $f_h(\omega) \coloneqq f_{h(\omega)}(\omega)$.
We will also need the notion of a \emph{cut} $U\subset\situations{}$: a collection of pairwise incomparable situations.
We call a cut $U$ \emph{complete} if for all $\omega\in\Omega$ there is some $u\in U$ such that $\omega\in\Gamma(u)$.
Otherwise, we call $U$ \emph{partial}.
We will also use the simpler notation $s$ to denote the cut $\{s\}$ that consists of the single situation $s\in\situations$.
For any situation $s$ and any cut $U$, we write $s \sqsubset U$ if $s \sqsubset u$ for all $u\in U$.
So, if $s\not=\Box$ and $s\sqsubset U$, then $U$ must be partial.
Conversely, we write $U \sqsubset s$ if there is a $u\in U$ such that $u \sqsubset s$.
In a similar way, we extend the relations $\sqsubseteq$, $\sqsupset$ and $\sqsupseteq$.
Analogously to what we did before for situations, we say that a path $\omega\in\samplespace$ goes through a cut $U$ when there is some $n\in\natz$ such that $\omega^n\in U$.

\begin{lemma}\label{Lemma: upprev of n-measurables < upprev of limit}
Consider any $s\in\situations{}$ and any sequence $\{f_n\}_{n\in\natz}$ of $n$-measurable gambles that converges pointwise to a variable $f\in\setofextvariables$ that is bounded above.
Then, for any $m\in\natz$ and any $\alpha\in\reals$ such that $\upprevvovkk(f \vert s) < \alpha$, there is a gamble $h\colon\Omega\to\natz$ that is $(\sup h)$-measurable such that $m\leq h$ and \/ $\upprevvovkk{}(f_h \vert s) \leq \alpha$.
\end{lemma}
\begin{proof}
Fix any $m\in\natz$, any $\alpha\in\reals$ such that $\upprevvovkk(f \vert s) < \alpha$ and any $\epsilon\in\posreals$.
According to the definition of $\upprevvovkk(f \vert s)$, there is a supermartingale $\martingale\in\setofextsupmartb{}$ such that $\martingale(s)\leq\alpha$ and $\liminf\martingale \geq_s f$.
We start by showing that, for any $\omega\in\Gamma(s)$ and any $n^\ast\in\natz$, there is a natural number $n\geq n^\ast$ such that $\martingale(\omega^n) + \epsilon \geq f_n(\omega)$.

So consider any $\omega\in\Gamma(s)$.
First note that $\liminf \martingale(\omega) + \epsilon > f(\omega)$ because $\liminf\martingale\geq_s f$, $\liminf\martingale(\omega)> -\infty$ [$\martingale$ is bounded below] and $f(\omega)<+\infty$ [$f$ is bounded above].
This implies that there is a real $\beta$ such that $\liminf \martingale(\omega) + \epsilon > \beta > f(\omega)$.
Then, since $\{f_n(\omega)\}_{n\in\natz}$ converges to $f(\omega)$ and $\beta$ is a real such that $\beta > f(\omega)$, there is an index $N(\omega)\in\natz$ such that $\beta > f_n(\omega)$ for all $n\geq N(\omega)$.
Furthermore, by the definition of the limit inferior and the fact that $\beta$ is a real such that $\liminf \martingale(\omega) + \epsilon > \beta$, there is a second index $M(\omega)\in\natz$ such that $\martingale(\omega^n) + \epsilon > \beta$ for all $n \geq M(\omega)$.
Combined with the previous, we obtain that $\martingale(\omega^n) + \epsilon > \beta > f_n(\omega)$ for all $n\geq \max\{N(\omega),M(\omega)\}$.
Since this holds for all $n \geq \max\{N(\omega),M(\omega)\}$, there is for any $n^\ast\in\natz$ also an $n\geq n^\ast$ such that $\martingale(\omega^n)+\epsilon\geq f_n(\omega)$.

Let $\ell$ be the length of the string $s$ and consider the variable $h\in\setofextvariables$ defined by
\begin{equation*}
h(\omega) \coloneqq
\begin{cases}
\inf \big\{n \geq \max\{\ell,m\} \colon \martingale(\omega^n)+\epsilon \geq f_n(\omega) \big\} &\text{ if } \omega\in\Gamma(s); \\
\max\{\ell,m\} &\text{ otherwise, }
\end{cases}
\text{ for all } \omega\in\Omega.
\end{equation*}
It clearly follows from the argument above that $h$ takes values in $\natz$.
We will now moreover show that $h(\omega)=h(\tilde{\omega})$ for any $\omega\in\Omega$ and any $\tilde{\omega}\in\Gamma(\omega^{h(\omega)})$, implying that $h$ satisfies the conditions in Lemma~\ref{Lemma: convergence of n-measurables}.

Consider any $\omega\in\Omega$ and any $\tilde{\omega}\in\Gamma(\omega^{h(\omega)})$.
We distinguish two cases: $\omega\in\Gamma(s)$ and $\omega\not\in\Gamma(s)$.
If $\omega\in\Gamma(s)$, then it follows from the definition of $h$ that $\martingale(\omega^{h(\omega)})+\epsilon \geq f_{h(\omega)}(\omega)$.
Since $\omega^{h(\omega)}=\tilde{\omega}^{h(\omega)}$ [because $\tilde{\omega}\in\Gamma(\omega^{h(\omega)})$] and since $f_{h(\omega)}$ is $h(\omega)$-measurable, this implies that $\martingale(\tilde{\omega}^{h(\omega)})+\epsilon \geq f_{h(\omega)}(\tilde{\omega})$.
Then, according to the definition of $h$ and since $\tilde{\omega}\in\Gamma(s)$ [because $h(\omega)\geq\ell$ and $\omega\in\Gamma(s)$, and therefore $\tilde{\omega}\in\Gamma(\omega^{h(\omega)})\subseteq\Gamma(s)$], we have that $h(\tilde{\omega})\leq h(\omega)$.
On the other hand, since $\tilde{\omega}\in\Gamma(s)$ and $\omega\in\Gamma(\tilde{\omega}^{h(\tilde{\omega})})$ [because $\omega^{h(\omega)}=\tilde{\omega}^{h(\omega)}$ and $h(\tilde{\omega}) \leq h(\omega)$], we can infer, in exactly the same way as before, that $h(\omega) \leq h(\tilde{\omega})$.
So we conclude that $h(\omega) = h(\tilde{\omega})$ in case that $\omega\in\Gamma(s)$.
If $\omega\not\in\Gamma(s)$, then $\tilde{\omega}\not\in\Gamma(s)$ because $h(\omega)\geq\ell$ and therefore $\Gamma(\omega^{h(\omega)})\cap\Gamma(s)=\emptyset$.
Then it follows immediately from the definition of~$h$ that $h(\omega)=h(\tilde{\omega})$.
Hence, $h$ satisfies the conditions in Lemma~\ref{Lemma: convergence of n-measurables}, so we have that $\sup h \in\natz$ and that $h$ is a $(\sup h)$-measurable gamble.
Furthermore, we trivially have that $h \geq \ell$ and $h\geq m$.

Let $U \coloneqq\{t\in\situations \colon (\exists \omega\in\samplespace)  t=\omega^{h(\omega)} \}$.
Note that any two (different) situations $t$ and $u$ in $U$ are incomparable.
Indeed, assume \emph{ex absurdo} that this is not the case.
Then there are at least two different situations $t,u \in U$ such that $t \sqsubseteq u$.
Let $\omega$ and $\tilde{\omega}$ be two paths such that $t=\omega^{h(\omega)}$ and $u=\tilde{\omega}^{h(\tilde{\omega})}$.
Since $t \sqsubseteq u$, we have that $\tilde{\omega}\in\Gamma(t)=\Gamma(\omega^{h(\omega)})$, which due to our previous considerations implies that $h(\tilde{\omega})=h(\omega)$.
Hence, taking into account that $\tilde{\omega}\in\Gamma(\omega^{h(\omega)})$, we infer that $\omega^{h(\omega)}=\tilde{\omega}^{h(\tilde{\omega})}$ and therefore that $t=u$.
This is in contradiction with our assumption that $t$ and $u$ are different, so we conclude that all situations in $U$ are pairwise incomparable and therefore, that $U$ is a cut.
$U$ is also a complete cut. To see why, observe that since $h$ takes values in $\natz$, it follows that, for any $\omega\in\samplespace{}$, $\omega^{h(\omega)}$ is a situation, which by definition is an element of $U$; it is moreover clear that for this situation $\omega^{h(\omega)}$, we have that $\omega\in\Gamma(\omega^{h(\omega)})$. Hence, $U$ is a complete cut, because for all $\omega\in\samplespace$, there is a situation $u\in U$ such that $\omega\in\Gamma(u)$.

For any situation $t \sqsupseteq U$, let us write $u(t)$ to denote the unique situation in $U$ such that $u(t) \sqsubseteq t$.
That $u(t)$ is unique follows from the fact that the situations in $U$ are incomparable.
Indeed, if there would be a second $u'(t)\in U$ such that $u'(t) \sqsubseteq t$, this would imply that either $u(t) \sqsubseteq u'(t)$ or $u'(t) \sqsubseteq u(t)$, which is impossible since $u(t)$ and $u'(t)$ are incomparable.
Now let $\martingale_U$ be the extended real process defined by
\begin{align*}
\martingale_U(t) \coloneqq
\begin{cases}
\martingale(u(t)) &\text{ if } U \sqsubset t; \\
\martingale(t) &\text{ if } U \not\sqsubset t,
\end{cases}
\text{ for all } t\in\situations.
\end{align*}
The process $\martingale_U$ is bounded below because $\martingale$ is bounded below.
To see that $\martingale_U$ is a supermartingale, note that
\begin{align*}
\martingale_U(t\andstate) \coloneqq
\begin{cases}
\martingale(u(t)) &\text{ if } U \sqsubseteq t; \\
\martingale(t\andstate) &\text{ if } U \not\sqsubseteq t,
\end{cases}
\text{ for all } t\in\situations.
\end{align*}
So for any situation $t\sqsupseteq U$, we have that $\martingale_U(t)=\martingale(u(t))$ and that $\martingale_U(t\cdot)=\martingale(u(t))$, which implies that $\lupprev{t}(\martingale_{U}(t\cdot))=\martingale_U(t)$ because of \ref{coherence: bounds}.
On the other hand, for any situation $t\not\sqsupseteq U$, we also have that $\lupprev{t}(\martingale_{U}(t\cdot))\leq\martingale_U(t)$ because $\lupprev{t}(\martingale(t\cdot))\leq\martingale(t)$ [$\martingale$ is a supermartingale].
As a consequence, $\martingale_U$ is a bounded below supermartingale.

For any $\omega\in\samplespace$, we now let $u(\omega)$ be the unique situation in $U$ such that $\omega\in\Gamma(u(\omega))$.
Then clearly $u(\omega)=\omega^{h(\omega)}$.
Moreover, for an $m\in\natz$ large enough such that $U\sqsubset\omega^m$, we also have that $\martingale_U(\omega^m) = \martingale(u(\omega^m)) = \martingale(u(\omega))$.
Hence,
\begin{align*}
\lim_{m \to +\infty} \martingale_U(\omega^m)=\lim_{m \to +\infty} \martingale(u(\omega))=\martingale(u(\omega))=\martingale(\omega^{h(\omega)}) \text{ for all } \omega\in\samplespace.
\end{align*}
Therefore, by definition of $h$, we have that
\begin{align*}
	\lim_{m \to +\infty} \martingale_U(\omega^m)+\epsilon
	=\martingale(\omega^{h(\omega)})+\epsilon
	\geq f_{h(\omega)}(\omega)
	\eqqcolon f_{h}(\omega) \text{ for all } \omega\in\Gamma(s).
\end{align*}
Then by Definition~\ref{def:upperexpectation2} and taking into account that $\martingale_U\in\setofextsupmartb{}$ and therefore $[\martingale_U + \epsilon]\in\setofextsupmartb{}$ [because the local models $\lupprev{s}$ satisfy \ref{coherence: const add}], it follows that $\upprevvovkk(f_h\vert s)\leq\martingale_U(s)+\epsilon$.
Moreover, $\martingale_U(s)=\martingale(s)$ because $h \geq \ell$ and therefore $U \not\sqsubset s$, so we also have that $\upprevvovkk(f_h\vert s)\leq\martingale(s)+\epsilon\leq\alpha+\epsilon$.
Since this inequality holds for any $\epsilon\in\posreals$, we infer that $\upprevvovkk(f_h\vert s)\leq\alpha$, which together with the fact that $h$ is a $(\sup h)$-measurable gamble such that $m\leq h$, establishes the lemma.
\end{proof}

The idea underlying the proof of Theorem~\ref{Theorem: cont. wrt non-increasing n-measurables} originates in a result by some of us \cite[Theorem~3]{DeBock2014Continuity}, however, once more, real supermartingales were adopted there.
Moreover, our result here considers sequences of (extended real) finitary variables that are bounded above, instead of sequences of $n$-measurable gambles.

\medskip
\begin{proofof}{Theorem~\ref{Theorem: cont. wrt non-increasing n-measurables}}
Fix any $s\in\situations$ and any non-increasing sequence $\{f_n\}_{n\in\natz}$ of finitary, bounded above variables that converges pointwise to a variable $f\in\setofextvariables$.
Note that $f$ is then also bounded above.
Because $f_n \geq f_{n+1} \geq f$ for all $n\in\natz$ and $\upprevvovkk$ is monotone [\ref{vovk coherence 4}], the limit $\lim_{n\to+\infty} \upprevvovkk(f_n \vert s)$ exists and we have that $\lim_{n\to+\infty} \upprevvovkk(f_n \vert s) \geq \upprevvovkk(f \vert s)$.
So we are left to show that $\lim_{n\to+\infty} \upprevvovkk(f_n \vert s) \leq \upprevvovkk(f \vert s)$.

Consider the sequence $\{f_n^{\vee -n}\}_{n\in\natz}$ and note that it suffices to show that $\lim_{n\to+\infty} \upprevvovkk(f_n^{\vee -n} \vert s) \leq \upprevvovkk(f \vert s)$, where the limit $\lim_{n\to+\infty} \upprevvovkk(f_n^{\vee -n} \vert s)$ exists because $\{f_n^{\vee -n}\}_{n\in\natz}$ is clearly non-increasing [since $\{f_n\}_{n\in\natz}$ is non-increasing] and $\upprevvovkk$ is monotone [\ref{vovk coherence 4}].
Indeed, it will then follow that $\lim_{n\to+\infty} \upprevvovkk(f_n \vert s) \leq \upprevvovkk(f \vert s)$ because $f_n \leq f_n^{\vee -n}$ for all $n\in\natz$ and therefore, by \ref{vovk coherence 4}, that $\lim_{n\to+\infty} \upprevvovkk(f_n \vert s) \leq \lim_{n\to+\infty} \upprevvovkk(f_n^{\vee -n} \vert s)\leq \upprevvovkk(f \vert s)$.

Since $\{f_n\}_{n\in\natz}$ is a sequence of finitary variables that converges non-increasingly to $f$, the same holds for the sequence $\{f_n^{\vee -n}\}_{n\in\natz}$.
In fact, $\{f_n^{\vee -n}\}_{n\in\natz}$ is even a sequence of finitary gambles because each $f_n$ is bounded above.
Now let $g_n \coloneqq f_n^{\vee -n}$ for all $n\in\natz$ and consider the sequence $\{g_n^{\xi}\}_{n\in\natz}$ defined by the recursive expression~\eqref{Eq: lemma: limits of n-measurables are equal to limits of finitary}, with $c\in\reals$ such that $c\geq g_0$.
Due to Lemma~\ref{lemma: limits of n-measurables are equal to limits of finitary}, $\{g_n^{\xi}\}_{n\in\natz}$ is a sequence of $n$-measurable variables.
Since $\{g_n\}_{n\in\natz} = \{f_n^{\vee -n}\}_{n\in\natz}$ is a sequence of finitary gambles that converges non-increasingly to $f$, it follows from Equation~\eqref{Eq: lemma: limits of n-measurables are equal to limits of finitary} and our choice of $c$ that $\{g_n^{\xi}\}_{n\in\natz}$ is a sequence of $n$-measurable gambles that converges non-increasingly to $f$.
Indeed, it can easily be checked that the transition from $\{g_n\}_{n\in\natz}$ to $\{g_n^{\xi}\}_{n\in\natz}$ preserves limit behaviour, the non-increasing character and the fact that the individual elements of the sequence are gambles.
In the same way, we can also see that
\begin{equation}\label{Eq: Theorem: cont. wrt non-increasing n-measurables}
\lim_{n\to+\infty}\upprevvovkk(f_n^{\vee -n} \vert s)=\lim_{n\to+\infty}\upprevvovkk(g_n\vert s)=\lim_{n\to+\infty}\upprevvovkk(g_n^{\xi}\vert s).
\end{equation}
Consider any real $\alpha > \upprevvovkk(f\vert s)$, which is guaranteed to exist because $f$ is bounded above and therefore, by \ref{vovk coherence 1}, $\upprevvovkk(f \vert s) < +\infty$.
Then since $\{g_n^\xi\}_{n\in\natz}$ is a sequence of $n$-measurable gambles that converges non-increasingly to~$f$, Lemma~\ref{Lemma: upprev of n-measurables < upprev of limit} implies that there is a gamble $h\colon\Omega\to\natz$ that is $(\sup h)$-measurable and is such that $\upprevvovkk(g_h^\xi \vert s) \leq \alpha$ [we simply let $m=0$ in the lemma].
Since $\{g_n^\xi\}_{n\in\natz}$ is non-increasing and $\upprevvovkk$ is monotone [\ref{vovk coherence 4}], we have that $\upprevvovkk(g_h^\xi \vert s) \geq \upprevvovkk(g_{\sup h}^\xi \vert s) \geq \lim_{n\to+\infty}\upprevvovkk(g_n^\xi \vert s)$, so we infer that $\lim_{n\to+\infty}\upprevvovkk(g_n^\xi \vert s) \leq \alpha$.
Recalling Equation~\eqref{Eq: Theorem: cont. wrt non-increasing n-measurables}, it follows that $\lim_{n\to+\infty}\upprevvovkk(f_n^{\vee -n} \vert s) \leq \alpha$.
Since this holds for any real $\alpha > \upprevvovkk(f\vert s)$, we conclude that $\lim_{n\to+\infty}\upprevvovkk(f_n^{\vee -n} \vert s) \leq \upprevvovkk(f\vert s)$ as desired.
\end{proofof}
\medskip

For any countable net $\{c_{\, (m,n)}\}_{m,n\in\natz}$ in $\extreals$, we say that $c \coloneqq \lim_{(m,n) \to +\infty} c_{\, (m,n)}$ is the Moore-Smith limit of $\{c_{\, (m,n)}\}_{m,n\in\natz}$ if, for each neighbourhood $A$ of $c$, there is a couple $(m^\ast,n^\ast)\in\natz^2$ such that $c_{\, (m,n)}\in A$ for all $m\geq m^\ast$ and all $n\geq n^\ast$.
Then, for any countable net $\{f_{\, (m,n)}\}_{m,n\in\natz}$ in $\setofextvariables$ such that $\lim_{(m,n) \to +\infty} f_{\, (m,n)}(\omega)$ exists for all $\omega\in\Omega$, we write $\lim_{(m,n) \to +\infty} f_{\, (m,n)}$ to denote the variable in $\setofextvariables$ defined by $\lim_{(m,n) \to +\infty} f_{\, (m,n)}(\omega)$ for all $\omega\in\Omega$.

\begin{lemma}\label{lemma: Moore-smith}
Consider any sequence $\{f_n\}_{n\in\natz}$ in $\setofgambles{}$ that converges pointwise to some variable $f \in \setofextvariablesb{}$.
Then we have that $\lim_{(m,n) \to +\infty} f_n^{\wedge m} = f$.
\end{lemma}
\begin{proof}
Consider any $\omega \in \samplespace{}$.
First consider the case that $f(\omega) \in \reals{}$ and fix any $\epsilon\in\posreals$.
Then there is an $n^\ast \in \natz$ such that $\vert f_{n}(\omega) - f(\omega) \vert \leq \epsilon$ for all $n \geq n^\ast$.
Consider any $m^\ast \geq f(\omega) + \epsilon$.
Then for all $n \geq n^\ast$ and all $m \geq m^\ast$, we have that $f_{n}(\omega) \leq f(\omega) + \epsilon \leq m$, so $f_{n}^{\wedge m}(\omega) = f_{n}(\omega)$ and therefore $\vert f_n^{\wedge m}(\omega) - f(\omega) \vert = \vert f_n(\omega) - f(\omega) \vert \leq \epsilon$.
So we have that $\lim_{(m,n) \to +\infty} f_n^{\wedge m}(\omega) = f(\omega)$.
If $f(\omega) = + \infty$, fix any $\alpha > 0$.
Then there is an $n^\ast \in \natz$ such that $f_{n}(\omega) \geq \alpha$ for all $n \geq n^\ast$.
If we now take $m^\ast \geq \alpha$, then clearly also $f_{n}^{\wedge m}(\omega) \geq \alpha$ for all $n \geq n^\ast$ and all $m \geq m^\ast$.
Hence, we have that $\lim_{(m,n) \to +\infty} f_n^{\wedge m}(\omega) = f(\omega)$ which, together with our earlier considerations, allows us to conclude that $\lim_{(m,n) \to +\infty} f_n^{\wedge m} = f$.
\end{proof}

\begin{proofof}{Theorem~\ref{Theorem: limit of n-meas}}
Fix any $s\in\situations$ and any $f\in\setoflimitsoffinmeasb{}$.
According to Proposition~\ref{proposition: limits of n-measurables are equal to limits of finitary}, there is a sequence $\{f_n\}_{n\in\natz}$ of $n$-measurable gambles that converges pointwise to $f$ and such that $B \leq f_n \leq \sup f$ for all $n\in\natz$, where $B$ is any real if $\inf f = +\infty$ and $B\coloneqq\inf f$ if $\inf f \in\reals$ [$\inf f = -\infty$ is impossible because $f$ is bounded below].
Fix any $\ell\in\nats$ and note that the sequence $\{f_n^{\wedge \ell}\}_{n\in\natz}$ is a sequence of $n$-measurable gambles that converges pointwise to $f^{\wedge \ell}$ because $\{f_n\}_{n\in\natz}$ is a sequence of $n$-measurable gambles that converges pointwise to $f$.
Moreover, $f^{\wedge \ell}$ is bounded above by $\ell$, so Lemma~\ref{Lemma: upprev of n-measurables < upprev of limit} guarantees that, for any $m\in\natz$ and any $\alpha\in\reals$ such that $\upprevvovkk(f^{\wedge \ell} \vert s) < \alpha$, there is a gamble $h\colon\Omega\to\natz$ that is $(\sup h)$-measurable such that $m\leq h$ and $\upprevvovkk(f_h^{\wedge \ell} \vert s) \leq \alpha$.
Since $f^{\wedge \ell}$ is both bounded below and above and $\upprevvovkk$ satisfies \ref{vovk coherence 5}, we have that $\upprevvovkk(f^{\wedge \ell} \vert s)\in\reals$ and therefore, that $\upprevvovkk(f^{\wedge \ell} \vert s) < \upprevvovkk(f^{\wedge \ell} \vert s) + 1/\ell$.
So in particular, Lemma~\ref{Lemma: upprev of n-measurables < upprev of limit} guarantees that, for any $m\in\natz$, there is a gamble $h\colon\Omega\to\natz$ that is $(\sup h)$-measurable such that $m\leq h$ and $\upprevvovkk(f_h^{\wedge \ell} \vert s) \leq \upprevvovkk(f^{\wedge \ell} \vert s) + 1/\ell$.
This holds for any $\ell\in\nats$, so there is a sequence $\{h_\ell\}_{\ell\in\natz}$ of $(\sup h_\ell)$-measurable gambles $h_\ell\colon\Omega\to\natz$ such that $h_0=0$ and, for all $\ell\in\nats$, satisfies $h_{\ell} \geq \sup h_{\ell-1}+1$ and $\upprevvovkk(f_{h_\ell}^{\wedge \ell} \vert s) \leq \upprevvovkk(f^{\wedge \ell} \vert s) + 1/\ell$.
We now show that $\{f_{h_\ell}^{\wedge \ell}\}_{\ell\in\natz}$ is a sequence of finitary gambles that is uniformly bounded below and that converges pointwise to $f$ such that $\lim_{\ell\to+\infty}\upprevvovkk(f_{h_\ell}^{\wedge \ell} \vert s) = \upprevvovkk(f \vert s)$.

Each $f_{h_\ell}^{\wedge \ell}$ is a gamble because it is bounded above by $\ell$ and because, since each $f_n$ is bounded below by $B$, $f_{h_\ell}^{\wedge \ell}$ is bounded below by $\min\{B,\ell\}$.
It then also immediately follows that $\{f_{h_\ell}^{\wedge \ell}\}_{\ell\in\natz}$ is uniformly bounded below by $\min\{B,0\}$.
Moreover, each $f_{h_\ell}^{\wedge \ell}$ is finitary because $h_\ell$ is $(\sup h_\ell)$-measurable.
Indeed, the latter implies that $h_\ell(\omega)=h_\ell(\tilde{\omega})$ for any $\omega\in\Omega$ and any $\tilde{\omega}\in\Gamma(\omega^{(\sup h_\ell)})$.
Then we also have that
\begin{equation*}
f_{h_\ell}^{\wedge \ell}(\omega)
\coloneqq f_{h_\ell(\omega)}^{\wedge \ell}(\omega)
=  f_{h_\ell(\tilde{\omega})}^{\wedge \ell}(\omega)
=  f_{h_\ell(\tilde{\omega})}^{\wedge \ell}(\tilde{\omega})
\eqqcolon f_{h_\ell}^{\wedge \ell}(\tilde{\omega}),
\end{equation*}
where the third step follows from the fact that $f_{h_\ell(\tilde{\omega})}^{\wedge \ell}$ is $h_\ell(\tilde{\omega})$-measurable and that $\tilde{\omega}\in\Gamma(\omega^{(\sup h_\ell)})\subseteq\Gamma(\omega^{h_\ell(\tilde{\omega})})$ because $h_\ell(\tilde{\omega}) \leq \sup h_\ell$.
As a consequence, $f_{h_\ell}^{\wedge \ell}$ is $(\sup h_\ell)$-measurable too and therefore, finitary.

To see that $\{f_{h_\ell}^{\wedge \ell}\}_{\ell\in\natz}$ converges pointwise to $f$, recall that $\{f_n\}_{n\in\natz}$ is a sequence of gambles that converges pointwise to $f\in\setoflimitsoffinmeasb$.
So Lemma~\ref{lemma: Moore-smith} implies that $\lim_{(\ell,n)\to+\infty}f_n^{\wedge\ell}=f$, meaning that, for any $\omega\in\Omega$ and any neighbourhood $A$ of $f(\omega)$, there is a couple $(\ell^\ast,n^\ast)\in\natz^2$ such that $f_n^{\wedge\ell}(\omega)\in A$ for all $\ell\geq\ell^\ast$ and all $n\geq n^\ast$.
Then, since $\{h_\ell\}_{\ell\in\natz}$ is increasing in $\ell$ [because $h_{\ell}\geq\sup h_{\ell-1}+1$ for all $\ell\in\nats$], there is an $\ell'\in\natz$ such that $\ell\geq\ell^\ast$ and $h_\ell(\omega)\geq n^\ast$ for all $\ell\geq\ell'$.
Together with the previous, this implies that $f_{h_\ell}^{\wedge\ell}(\omega)\coloneqq f_{h_\ell(\omega)}^{\wedge\ell}(\omega) \in A$ for all $\ell\geq\ell'$.
Since there is such an $\ell'\in\natz$ for any $\omega\in\Omega$ and any neighbourhood $A$ of $f(\omega)$, we have that $\lim_{\ell\to+\infty}f_{h_\ell}^{\wedge\ell}=f$.

Finally, to see that $\lim_{\ell\to+\infty}\upprevvovkk(f_{h_\ell}^{\wedge \ell} \vert s) = \upprevvovkk(f \vert s)$, recall that $\{h_\ell\}_{\ell\in\natz}$ is such that $\upprevvovkk(f_{h_\ell}^{\wedge \ell} \vert s) \leq \upprevvovkk(f^{\wedge \ell} \vert s) + 1/\ell$ for all $\ell\in\nats$.
So we have that
\begin{align*}
\limsup_{\ell\to+\infty} \upprevvovkk(f_{h_\ell}^{\wedge \ell} \vert s) \leq \limsup_{\ell\to+\infty} \Big(\upprevvovkk(f^{\wedge \ell} \vert s) + 1/\ell \Big) = \limsup_{\ell\to+\infty} \upprevvovkk(f^{\wedge \ell} \vert s) = \upprevvovkk(f\vert s),
\end{align*}
where the last equality follows from Theorem~\ref{theorem: upward convergence game} which we can apply because $\{f^{\wedge \ell}\}_{\ell\in\natz}$ is a non-decreasing sequence in $\setofextvariablesb$ [because $f$ is bounded below] that converges pointwise to $f\in\setofextvariablesb$.
On the other hand, we have that $\liminf_{\ell\to+\infty} \upprevvovkk(f_{h_\ell}^{\wedge \ell} \vert s) \geq \upprevvovkk(f\vert s)$ because of Corollary~\ref{corollary: Fatou general} and the fact that $\{f_{h_\ell}^{\wedge \ell}\}_{\ell\in\natz}$ is uniformly bounded below by $\min\{B,0\}$ and converges pointwise to $f$.
Hence, we conclude that $\lim_{\ell\to+\infty}\upprevvovkk(f_{h_\ell}^{\wedge \ell} \vert s) = \upprevvovkk(f \vert s)$.

Our last step of the proof consists in modifying $\{f_{h_\ell}^{\wedge \ell}\}_{\ell\in\natz}$ such that it becomes a sequence of $n$-measurable gambles that still is uniformly bounded below and converges pointwise to $f$ in such a way that it is continuous with respect to $\upprevvovkk$.
We consider the sequence $\{f'_\ell\}_{\ell\in\natz}\coloneqq\big\{(f_{h_\ell}^{\wedge \ell})^{\, \xi}\big\}_{\ell\in\natz}$ defined through the recursive expression~\eqref{Eq: lemma: limits of n-measurables are equal to limits of finitary}, with $c=0$.
Then Lemma~\ref{lemma: limits of n-measurables are equal to limits of finitary} guarantees that $\{f'_\ell\}_{\ell\in\natz}$ is a sequence of $n$-measurable variables.
Moreover, from \eqref{Eq: lemma: limits of n-measurables are equal to limits of finitary}, it should be clear that $\{f'_\ell\}_{\ell\in\natz}$ is also a sequence of gambles that is uniformly bounded below and that converges pointwise to $f$ such that $\lim_{\ell\to+\infty}\upprevvovkk(f'_\ell \vert s) = \upprevvovkk(f \vert s)$.
Indeed, in the same way as we have argued in the proof of Theorem~\ref{Theorem: cont. wrt non-increasing n-measurables}, this follows from the fact that $\{f_{h_\ell}^{\wedge \ell}\}_{\ell\in\natz}$ is a sequence of gambles that is uniformly bounded below and that converges pointwise to $f$ such that $\lim_{\ell\to+\infty}\upprevvovkk(f_{h_\ell}^{\wedge \ell} \vert s) = \upprevvovkk(f \vert s)$.
So we conclude that $\{f'_\ell\}_{\ell\in\natz}$ is a sequence of $n$-measurable gambles that is uniformly bounded below and that converges pointwise to $f$ such that $\lim_{\ell\to+\infty}\upprevvovkk(f'_\ell \vert s) = \upprevvovkk(f \vert s)$.
\end{proofof}

\section{Discussion and alternative versions of game-theoretic upper expectations}\label{sect: Discussion}

An important contribution of this article is that we provide an overview of the main properties that are satisfied by a particular game-theoretic upper expectation operator, for the case where state spaces are assumed to be finite.
Some of these properties were already shown to hold in a slightly different setting, and our contribution consists in adapting them to our setting.
Other properties, mainly situated in Section~\ref{Sect: Continuity} and~\ref{Sect: Continuity wrt n-measurables}, are entirely new and create novel insights.
Section~\ref{Sect: Continuity wrt n-measurables}, for example, shows that the game-theoretic upper expectation $\upprevvovkk{}$ behaves in a particularly interesting way with respect to (limits of) finitary variables:
Proposition~\ref{Prop: limits to general} implies that $\upprevvovkk{}$ is uniquely characterised by its values on the domain $\setoflimitsoffinmeasb{}$, and Theorems~\ref{Theorem: cont. wrt non-increasing n-measurables} and~\ref{Theorem: limit of n-meas} show that $\upprevvovkk{}$ has rather strong continuity properties when it comes to sequences of finitary variables.

These results have already proved useful elsewhere, in showing how game-theoretic upper expectations can be alternatively characterised without the use of supermartingales \cite{TJOENS202130}.
Moreover, in that same paper \cite{TJOENS202130}, we also use these properties to relate the game-theoretic upper expectation to a more traditional measure-theoretic model.
Specifically, we show there that game-theoretic and measure-theoretic upper expectations coincide when all local models $\lupprev{s}$ are precise---that is, when they correspond to linear expectations---and that, in a general imprecise setting, the game-theoretic upper expectation is always more conservative (higher) than the measure-theoretic one.
Recently, we discovered that an even stronger relation exists; game-theoretic and measure-theoretic upper expectations coincide on a domain that includes all monotone limits of finitary gambles and all bounded below Borel measurable variables \cite{TJoens2021Equivalence}.
We also refer to \cite[Section~9]{Vovk2019finance} for further details on the relation between game-theoretic and measure-theoretic models.

It also becomes apparent, as a consequence of Proposition~\ref{prop: extension with E6 is an extension}, that Shafer and Vovk's axiomatisation for a local upper expectation can be weakened---and therefore generalised---while leaving the resulting global game-theoretic upper expectation unaffected.
This weaker axiomatisation is moreover equivalent to a particular extension of coherence to extended real-valued variables; see Proposition~\ref{Prop: alt. characterisation upper exp.}.
Finally, our axiomatisation, and even Shafer and Vovk's stronger axiomatisation, do not suffice in order to guarantee compatibility of the local models with the global game-theoretic upper expectation.
Indeed, it follows from Example~\ref{example: Vovk's axioms do not necessarily satisfy lower cuts continuity} and Corollary~\ref{corollary: complete compatibility} that such compatibility is only guaranteed if we additionally impose \ref{ext coherence lower cuts continuity} on the local models.


Now, before we conclude the article, we want to clarify why we have chosen to work with the specific game-theoretic upper expectation in Definition~\ref{def:upperexpectation2}.
As already mentioned in the introduction, it seems appropriate to motivate this, because many slightly different versions of this definition have appeared in the literature \cite{Shafer:2005wx,shafer2011levy,ShaferVovk2014ITIP,DECOOMAN201618}, and it may perhaps not be entirely clear what these differences entail.

Most of the versions that appear in the literature only differ in how the supermartingales are allowed to behave.
In Section~\ref{Sect: equivalent characterisations}, it was shown that the definition of $\upprevvovkk{}$ is fairly robust with respect to changes that concern the limit behaviour of supermartingales and, more specifically, how this limit behaviour relates to the variable $f$ in consideration: see also \cite[Sections~7--8]{Vovk2019finance}.
However, there are two issues that do impact the resulting game-theoretic upper expectation, and which were often also considered in the past: whether to define supermartingales as real processes or as extended real processes, and whether they are required to be bounded below or not.
Let us first focus on the latter issue.

Considering that supermartingales are interpreted as capital processes, we think it is more natural to assume that they should be bounded below, simply because one cannot borrow an infinite or unbounded amount of money.
But there is more to it than this interpretational argument, because by restricting ourselves to bounded below supermartingales we avoid the undesirable situation where the upper expectation would be lower than the lower expectation \cite[Example~1]{DECOOMAN201618}.
This would occur irrespectively of whether supermartingales are assumed to be real-valued or extended real-valued.
Indeed, let $\setofubextsupmartb{}$ be the set of all (not necessarily bounded below) extended real-valued supermartingales, let $\smash{\setofubextsupmartbb{} \subseteq \setofubextsupmartb{}}$ be the subset of all real-valued supermartingales, and let $\ubupprevvovk{}$ and $\smash{\uboldupprevvovk{}}$ be the game-theoretic upper expectations that are obtained by replacing the set $\setofextsupmartb{}$ in Definition~\ref{def:upperexpectation2} with the respective sets $\setofubextsupmartb{}$ and $\setofubextsupmartbb{}$.
Since $\smash{\setofubextsupmartbb{} \subseteq \setofubextsupmartb{}}$, we clearly have that $\smash{\ubupprevvovk{}(f \vert s) \leq \uboldupprevvovk{}(f \vert s)}$ and $\smash{\uboldlowprevvovk{}(f \vert s) \leq \ublowprevvovk{}(f \vert s)}$ for all $f \in \setofextvariables{}$ and all $s \in \situations{}$.
Now, as was shown in \cite[Example~1]{DECOOMAN201618}, there is some $f \in \setofextvariables{}$ and some $s \in \situations{}$ (and some imprecise probabilities tree) such that $\smash{\uboldlowprevvovk{}(f \vert s) > \uboldupprevvovk{}(f \vert s)}$, which then also implies that $\ublowprevvovk{}(f \vert s) > \ubupprevvovk{}(f \vert s)$.
In order to prevent this from happening, we limit ourselves to bounded below (extended real-valued or real-valued) supermartingales.

Now we are left with the question of whether to allow supermartingales to become extended real-valued or not.
Let us write $\smash{\setofextsupmartbb{}}$ to denote the set of all real-valued bounded below supermartingales and let~$\smash{\oldupprevvovk{}}$ be the corresponding game-theoretic upper expectation:
\begin{equation*}
\oldupprevvovk (f \vert s) \coloneqq \inf\big\{ \martingale(s) \colon \martingale\in\setofextsupmartbb \text{ and } \liminf\martingale \geq_s f \big\} \ \text{ for all } f\in\setofextvariables \text{ and all } s\in\situations.
\end{equation*}
At first sight, we would be inclined to use the game-theoretic upper expectation $\oldupprevvovk$ with real-valued supermartingales, because it allows for a more direct interpretation.
Indeed, if we interpret a supermartingale as the possible evolution of a person's capital, it is not clear to us what it means if this person's capital were to become infinite in value.
Moreover, on the domain $\setofgambles{}\times\situations{}$, the version $\smash{\oldupprevvovk}$ with real supermartingales has rather desirable properties;
as we will show below with Proposition~\ref{prop: two versions of vovk upprev are equal}, it coincides with our version $\upprevvovkk{}$ on the domain $\setofgambles{} \times \situations{}$.
All things considered, it ought not to surprise that $\oldupprevvovk$ was frequently adopted in earlier work by both ourselves \cite{DECOOMAN201618,TJoens2018Continuity} and by Shafer and Vovk \cite{Shafer:2005wx,ShaferVovk2014ITIP}.
However, as we will soon point out, it has a rather undesirable property once it is applied to the entire domain $\setofextvariables{}\times\situations{}$.

Before we do so, we want to draw the attention to the fact that there is yet another, maybe even more intuitive, possible way of defining game-theoretic upper expectations: using bounded (above and below) real-valued supermartingales.
On the one hand, as we already explained, we consider it sensible to assume that supermartingales should be bounded below because one cannot borrow an infinite or unbounded amount of money.
On the other hand, for similar reasons, we could also impose that supermartingales should be bounded above; receiving an infinite or unbounded amount of money seems impossible in practice too.
We therefore think it is appropriate to also consider a definition of the game-theoretic upper expectation with bounded (and hence real-valued) supermartingales.
Let $\smash{\setofextsupmartbbb{}}$ be the set of all such bounded (real-valued) supermartingales and let~$\smash{\oldupprevvovkb{}}$ be the corresponding game-theoretic upper expectation:
\begin{equation*}
\oldupprevvovkb (f \vert s) \coloneqq \inf\big\{ \martingale(s) \colon \martingale\in\setofextsupmartbbb \text{ and } \liminf\martingale \geq_s f \big\} \ \text{ for all } f\in\setofextvariables \text{ and all } s\in\situations.
\end{equation*}
Despite that it allows for a rather direct interpretation, such a version $\oldupprevvovkb{}$ of the game-theoretic upper expectation has, to the best of our knowledge, never been used in the literature.
The reason, presumably, is that for extended real variables $f$ that are not bounded above, the value of $\oldupprevvovkb(f)$ does not provide any information because it will simply be equal to $+\infty$; that is, the infimum over the empty set.
However, as our next result shows, this version $\oldupprevvovkb{}$ coincides with our version $\upprevvovkk{}$---and therefore also with the version $\oldupprevvovk$ that uses real-valued bounded below supermartingales---on the domain $\setofgambles{}\times\situations{}$ of all gambles (and all situations).

\begin{proposition}\label{prop: two versions of vovk upprev are equal}
For any gamble $f\in\setofgambles{}$ and any situation $s\in\situations{}$, we have that $\upprevvovkk{}(f \vert s) = \oldupprevvovk{}(f \vert s) = \oldupprevvovkb{}(f \vert s)$.
\end{proposition}

This result relies on the following two lemmas.

\begin{lemma}\label{LemmaBoundedSupermartingale}
For any $\martingale\in\setofextsupmartb$ and any $B\in\reals$, the process $\martingale_B$, defined by
$\martingale_B(s)\coloneqq\min\{\martingale(s), B\}$ for all $s\in\situations$, is a bounded real-valued supermartingale.
\end{lemma}
\begin{proof}
It is clear that, since $\martingale$ is a bounded below extended real process, $\martingale_B$ is a bounded (above and below) real process.
Moreover, $\martingale_B(s)\leq\martingale(s)$ for all $s\in\situations$, so it follows that $\martingale_B(s\andstate)\leq\martingale(s\andstate)$ for all $s\in\situations$.
Fix any $s\in\situations$.
If $\martingale(s) \leq B$, then
\begin{align*}
\lupprev{s}(\martingale_B(s\andstate))\leq\lupprev{s}(\martingale(s\andstate))\leq\martingale(s)=\martingale_B(s),
\end{align*}
where the first equality follows from the monotonicity [\ref{coherence: monotonicity}] of $\lupprev{s}$ (and the fact that $\martingale(s\andstate)$ and $\martingale_B(s\andstate)$ are bounded below).
If $\martingale(s) > B$, it follows from $\martingale_B(s\andstate) \leq B$ and the monotonicity [\ref{coherence: monotonicity}] of $\lupprev{s}$ that
\begin{align*}
\lupprev{s}(\martingale_B(s\andstate))\leq\lupprev{s}( B) \overset{\text{\ref{coherence: const is const}}}{=} B=\martingale_B(s).
\end{align*}
So, we conclude that $\lupprev{s}(\martingale_B(s\andstate))\leq\martingale_B(s)$ for all situations $s\in\situations$.
Hence, $\martingale_B$ is a bounded real-valued supermartingale.
\end{proof}

\begin{lemma}\label{Lemma: the minimum-limit inferior switch}
For any extended real process $\process$ and any path $\omega\in\samplespace$, we have that
\begin{equation*}
\min \left\{B,\liminf_{n\to+\infty}\process(\omega^n)\right\}=\liminf_{n\to+\infty} \min\{B,\process(\omega^n)\} \text{ for all $B\in\reals$.}
\end{equation*}
\end{lemma}
\begin{proof}
Fix any $B\in\reals$.
It is easy to check that
$\min \left\{B,\liminf_{n\to+\infty}\process(\omega^n)\right\} \geq \liminf_{n\to+\infty} \min\{B,\process(\omega^n)\}$.
We prove the converse inequality by contradiction.
Suppose that
\begin{align*}
\min \left\{B,\liminf_{n\to+\infty}\process(\omega^n) \right\} >
\liminf_{n\to+\infty} \min\{B,\process(\omega^n)\},
\end{align*}
or, equivalently, that
\begin{equation*}
\min \left\{B,\liminf_{n\to+\infty}\process(\omega^n) \right\} > \sup_{m}\inf_{n \geq m} \min\{B,\process(\omega^n)\}.
\end{equation*}
Then there is some $\epsilon\in\posreals$ such that
\begin{equation*}
\min \set*{B,\liminf_{n\to+\infty}\process(\omega^n)}-\epsilon >\inf_{n \geq m} \min\{B,\process(\omega^n)\}=\min \set*{B,\inf_{n \geq m}\process(\omega^n)}
\end{equation*}
for all $ m\in\natz$.
Since $\min \{B,\liminf_{n\to+\infty}\process(\omega^n)\}-\epsilon < B$, it can only be that
\begin{align*}
\inf_{n \geq m}\process(\omega^n) < \min \set*{B,\liminf_{n\to+\infty}\process(\omega^n)}-\epsilon\leq\liminf_{n\to+\infty}\process(\omega^n)-\epsilon \text{ for all $m\in\natz$,}
\end{align*}
from which we infer that
$\inf_{n \geq m}\process(\omega^n) < \sup_{k}\inf_{n \geq k}\process(\omega^n)-\epsilon$ for all $m\in\natz$.
This contradicts the definition of the supremum operator.
\end{proof}

\begin{proofof}{Proposition~\ref{prop: two versions of vovk upprev are equal}}
Since $\setofextsupmartbbb{}\subseteq\setofextsupmartbb{}\subseteq\setofextsupmartb{}$, it is clear that $\upprevvovkk{}(f \vert s) \leq \oldupprevvovk{}(f \vert s) \leq \oldupprevvovkb{}(f \vert s)$.
So it suffices to prove that $\oldupprevvovkb{}(f \vert s) \leq \upprevvovkk{}(f \vert s)$.
Consider any $\martingale{}\in\setofextsupmartb{}$ such that $\liminf \martingale{} \geq_s f$.
Let $B \coloneqq \sup f$ and let $\martingale_{B}$ be defined by $\martingale_{B}(t)\coloneqq\min\{\martingale(t), B\}$ for all $t\in\situations$.
Since $f$ is a gamble, $\sup f$ is real and hence, due to Lemma~\ref{LemmaBoundedSupermartingale}, $\martingale_{B}$ is a bounded real-valued supermartingale.
Moreover, we have that $B \geq f$ and $\liminf \martingale{} \geq_s f$, and therefore that $\min\{B,\liminf \martingale{}\} \geq_s f$, which by Lemma~\ref{Lemma: the minimum-limit inferior switch} implies that $\liminf \martingale_{B} \geq_s f$.
Together with the fact that $\martingale{}_B \in \setofextsupmartbbb{}$, this implies by the definition of $\oldupprevvovkb{}$ that $\oldupprevvovkb{}(f \vert s) \leq \martingale{}_B(s) \leq \martingale{}(s)$.
Since this holds for any $\martingale{}\in\setofextsupmartb{}$ such that $\liminf \martingale{} \geq_s f$, we infer from the definition of $\upprevvovkk{}$ that $\oldupprevvovkb{}(f \vert s) \leq \upprevvovkk{}(f \vert s)$.
\end{proofof}

\vspace*{0.8em}

So the versions $\upprevvovkk{}$, $\oldupprevvovk{}$ and $\oldupprevvovkb{}$ are all mathematically equivalent on $\setofgambles{}\times\situations{}$. We therefore prefer $\oldupprevvovkb{}$ over both $\upprevvovkk{}$ and $\oldupprevvovk{}$ on $\setofgambles{}\times\situations{}$ because it can be given a (more) direct interpretation.
However, if we consider the entire domain $\setofextvariables{}\times\situations{}$, then the version $\oldupprevvovkb{}$ is unsuitable because, as explained above, it assigns $+\infty$ to every variable that is not bounded above.
Our remaining two options are then to use either $\upprevvovkk{}$ or $\oldupprevvovk{}$. From a purely interpretational point of view, we prefer $\oldupprevvovk{}$.
Unfortunately, the version $\oldupprevvovk{}$---and also the version $\oldupprevvovkb{}$---does not satisfy continuity with respect to non-decreasing sequences, which we consider to be an important mathematical drawback.

\begin{example}
Consider an uncertain process with state space $\statespace{} \coloneqq \{0,1\}$ and let $\lupprev{\Box}(f) \coloneqq f(0)$ for all $f \in \setofgenextvariables{}(\statespace{})$.
It can easily be checked that $\lupprev{\Box}$ satisfies properties \ref{coherence: const is const}--\ref{coherence: continuity wrt non-negative functions}.
For any $n \in \nats{}$, the variable $n \indica{1}(X_1)$ is a gamble, so Proposition~\ref{prop: two versions of vovk upprev are equal} implies that $\upprevvovkk{}(n \indica{1}(X_1)) = \oldupprevvovk{}(n \indica{1}(X_1)) = \oldupprevvovkb{}(n \indica{1}(X_1))$.
Then, taking into account the boundedness of $n \indica{1}(X_1)$, it follows from Proposition~\ref{prop: global compatible with local} that $\oldupprevvovkb{}(n \indica{1}(X_1)) = \oldupprevvovk{}(n \indica{1}(X_1)) = \upprevvovkk{}(n \indica{1}(X_1)) = \lupprev{\Box}(n \indica{1})$, which by definition of \/ $\lupprev{\Box}$ leads us to conclude that $\oldupprevvovkb{}(n \indica{1}(X_1)) = \oldupprevvovk{}(n \indica{1}(X_1)) = n \indica{1}(0) = 0$ for all $n \in \nats{}$.

On the other hand, consider the upper expectation $\oldupprevvovk{} \big(+\infty \indica{1}(X_1) \big)$.
It is defined as the infimum of $\martingale{}(\Box)$ over real-valued supermartingales $\martingale{}\in\setofextsupmartbb{}$ such that \/ $\liminf \martingale{} \geq +\infty \indica{1}(X_1)$.
Any such supermartingale $\martingale{}$ should converge to $+\infty$ on all paths $\omega \in \Gamma(1)$ such that $\omega^{1} = 1$.
Since $\martingale{} \in \setofextsupmartbb{} \subseteq \setofextsupmartb{}$, we can then apply Lemma~\ref{lemma: infima of supermartingales} to infer that $\martingale{}(1) = +\infty$.
This contradicts the real-valuedness of $\martingale{}$, which allows us to conclude that there is no $\martingale{}\in\setofextsupmartbb{}$ such that \/ $\liminf \martingale{} \geq +\infty \indica{1}(X_1)$.
Hence, $\oldupprevvovk{} \big(+\infty \indica{1}(X_1) \big) = +\infty$ and therefore, because $\oldupprevvovk{}\big(+\infty \indica{1}(X_1) \big)\leq\oldupprevvovkb{}\big(+\infty \indica{1}(X_1) \big)$ (recall that $\setofextsupmartbbb{}\subseteq\setofextsupmartbb{}$), also $\oldupprevvovkb{} \big(+\infty \indica{1}(X_1) \big) = +\infty$.

All together, we have that $\lim_{n \to +\infty} n \indica{1}(X_1) = +\infty \indica{1}(X_1)$ but
\begin{equation*}
\lim_{n \to +\infty} \oldupprevvovk{} \big(n \indica{1}(X_1) \big) \not= \oldupprevvovk{} \big(+\infty \indica{1}(X_1) \big) \ \text{ and } \ \lim_{n \to +\infty} \oldupprevvovkb{} \big(n \indica{1}(X_1) \big) \not= \oldupprevvovkb{} \big(+\infty \indica{1}(X_1) \big),
\end{equation*}
which implies that both $\oldupprevvovk{}$ and $\oldupprevvovkb{}$ fail to satisfy continuity with respect to non-decreasing sequences.
\hfill$\Diamond$
\end{example}

We have chosen to work with the version $\upprevvovkk{}$ in this paper, and therefore implicitly assumed that global (game-theoretic) upper expectations satisfy continuity with respect to non-decreasing sequences.
This is in line with how we introduced our local uncertainty models $\lupprev{s}$: the continuity property~\ref{coherence: continuity wrt non-negative functions} was directly adopted as part of their definition.
The fact that our global upper expectation $\upprevvovkk{}$ (and each local upper expectation) satisfies this non-decreasing continuity is mathematically convenient, but it does require a definition that is more abstract compared to $\oldupprevvovk{}$ or $\oldupprevvovkb{}$.
Nonetheless, we do not regard this as an issue because, on the domain $\setofgambles{}\times\situations{}$, our version $\upprevvovkk{}$ is equivalent to the version $\oldupprevvovkb{}$ (or $\oldupprevvovk{}$).
So one could just as well consider it to be defined as $\oldupprevvovkb{}$ (or $\oldupprevvovk{}$) on this restricted domain.
Then, as far as the values of $\upprevvovkk{}$ on $(\setofextvariables{}\setminus\setofgambles{})\times\situations{}$ are concerned, recall that these are uniquely determined by the values on $\setofgambles{}\times\situations{}$ through the continuity of $\upprevvovkk{}$; see the discussion below Proposition~\ref{Proposition: Continuity w.r.t. lower cuts}.
Hence, we can regard $\upprevvovkk{}$ on $(\setofextvariables{}\setminus\setofgambles{})\times\situations{}$ as an extension of $\upprevvovkk{}$ on $\setofgambles{}\times\situations{}$, where the latter is justified by its equivalence with $\oldupprevvovkb{}$, and where the extension is motivated by  assumptions of continuity.

\section*{Acknowledgements}

Natan T'Joens's research was supported and funded by the Special Research Fund (BOF) of Ghent University (reference number: 356).
The research by Jasper De Bock and Gert de Cooman was funded through project number 3G028919 of the Research Foundation – Flanders (FWO).
We are also indebted to an anonymous reviewer who provided helpful feedback.

\bibliographystyle{plain}
\bibliography{references}

\appendix

\section{Proofs of the results in Section~\ref{Sect: equivalent characterisations}}

In the following proofs, we will frequently use the terminology and notations concerning cuts, as introduced in the paragraph above Lemma~\ref{Lemma: upprev of n-measurables < upprev of limit}.
Moreover, for any two cuts $U$ and $V$, we will write $U \sqsubset V$ if $(\forall v\in V)(\exists u\in U)u \sqsubset v$, and similarly for the relations $\sqsubseteq$, $\sqsupset$ and $\sqsupseteq$.
We will also consider the following sets:
\begin{align*}
[U,V] \coloneqq\{s\in\situations : (\exists u\in U)(\exists v\in V)u \sqsubseteq s \sqsubseteq v \}, \\
[U,V) \coloneqq\{s\in\situations : (\exists u\in U)(\exists v\in V)u \sqsubseteq s \sqsubset v \}, \\
(U,V] \coloneqq\{s\in\situations : (\exists u\in U)(\exists v\in V)u \sqsubset s \sqsubseteq v \}, \\
(U,V) \coloneqq\{s\in\situations : (\exists u\in U)(\exists v\in V)u \sqsubset s \sqsubset v \}.
\end{align*}

\begin{proofof}{Proposition~\ref{Prop: Doob}}
Let $t\in\situations{}$ be any fixed situation where $\martingale{}(t)$ is real.
We can assume that $\martingale$ is non-negative and that $\martingale(t)=1$ without loss of generality.
Indeed, because the original supermartingale is bounded below and real in $t$, we can obtain such a process by translating and scaling---by adding a positive constant and then multiplying the supermartingale by a positive real---the originally considered supermartingale in an appropriate way.
This process will then again be a (bounded below) supermartingale because the local models $\lupprev{s}$ satisfy \ref{coherence: const add} and \ref{coherence: homog for non-negative lambda}.
Moreover, the new supermartingale will have the same convergence character as the original one.

To start, fix any couple of rational numbers $0<a<b$ and consider the following recursively constructed sequences of cuts $\{U_k^{a,b}\}_{k\in\nats}$ and $\{V_k^{a,b}\}_{k\in\nats}$.
	Let $V_1^{a,b} \coloneqq\{ s \sqsupseteq t \colon   \ \martingale(s) < a \text{ and } (\forall s'\in[t,s)) \ \martingale(s') \geq a \}$
	and, for $k\in\nats$,
	\begin{enumerate}
		\item let
			$U_k^{a,b} \coloneqq\{s\in\situations \colon V_{k}^{a,b} \sqsubset s : \martingale(s) > b \text{ and } (\forall s'\in (V_{k}^{a,b} , s)) \ \martingale(s') \leq b \}$;
		\item
		let
			$V_{k+1}^{a,b} \coloneqq\{ s\in\situations \colon  U_{k}^{a,b} \sqsubset s , \ \martingale(s) < a \text{ and } (\forall s'\in (U_{k}^{a,b} , s)) \ \martingale(s') \geq a \}$.
	\end{enumerate}
	Note that all $U_k^{a,b}$ and all $V_k^{a,b}$ are indeed (partial or complete) cuts.

	Next, consider the extended real process $\martingale^{a,b}$ defined by $\martingale^{a,b}(s) \coloneqq\martingale(t)$ for all $s \not\sqsupset t$ and by
	\begin{align}\label{Eq: proof Doob: def M^a,b}
	 	\martingale^{a,b} (sx) \coloneqq\begin{cases}
	 		\martingale^{a,b}(s) + \left[\martingale(sx) -\martingale(s) \right]   &\text{ if } s\in  [V_{k}^{a,b} , U_{k}^{a,b}) \text{ for some } k\in\nats;\\
	 		\martingale^{a,b}(s) &\text{ otherwise,}
	 		\end{cases}
	 \end{align}
	 for all $sx \sqsupset t$ with $x \in \statespace$.
	 We prove that $\martingale^{a,b}$ is a non-negative supermartingale that converges to $+\infty$ on all paths $\omega\in\Gamma(t)$ such that
	 \begin{equation}\label{Eq: doob conv to infty}
	 \liminf\martingale(\omega) < a < b < \limsup\martingale(\omega).
	 \end{equation}
	 For any situation~$s$ and for any $k\in\nats$, when $U_k^{a,b} \sqsubset s$, we denote by $u_k^s$ the (necessarily unique) situation in $U_k^{a,b}$ such that $u_k^s \sqsubset s$.
	 Similarly, for any $k\in\nats$, when $V_k^{a,b} \sqsubset s$, we denote by $v_k^s$ the (necessarily unique) situation in $V_k^{a,b}$such that $v_k^s \sqsubset s$.
	 Note that $ V_1^{a,b} \sqsubset U_1^{a,b} \sqsubset V_2^{a,b} \sqsubset \cdots \sqsubset V_n^{a,b} \sqsubset U_n^{a,b} \sqsubset \cdots $.
	 Hence, for any situation~$s$ we can distinguish the following three cases:
	 \begin{itemize}
	 \item
	 The first case is that $V_1^{a,b} \not\sqsubset s$. Then we have that
	 \begin{equation}\label{proof Doob: first case}
	 \martingale^{a,b} (s)=\martingale^{a,b} (t)=\martingale(t)=1.
	 \end{equation}
	  \item
	 The second case is that $V_k^{a,b} \sqsubset s$ and $U_{k}^{a,b} \not\sqsubset s$ for some $k\in\nats$.
	 Then by applying Equation \eqref{Eq: proof Doob: def M^a,b} for each subsequent step and cancelling out the intermediate terms, which is possible because $\martingale$ is real for any situation $s'\in\situations$ such that $V_{k'}^{a,b} \sqsubseteq s'$ and $U_{k'}^{a,b} \not\sqsubseteq s'$ for some $k'\in\nats$ [this follows readily from the definition of the cuts $V_{k'}^{a,b}$ and $U_{k'}^{a,b}$], we have that
	 \begin{equation}\label{proof Doob: second case}
	 \martingale^{a,b} (s)- \martingale^{a,b} (t)
	= \sum_{\ell=1}^{k-1}
	 \left[ \martingale(u_\ell^{s})-\martingale(v_{\ell}^{s}) \right] + \martingale(s)-\martingale(v_{k}^{s}).
	 \end{equation}

	 \item
	 The third case is that $U_k^{a,b} \sqsubset s$ and $V_{k+1}^{a,b} \not\sqsubset s$ for some $k\in\nats$. Then we have that
	 \begin{equation}\label{proof Doob: third case}
	 \martingale^{a,b} (s)- \martingale^{a,b} (t)
	=\sum_{\ell=1}^{k} [\martingale(u_\ell^{s})-\martingale(v_{\ell}^{s})],
	 \end{equation}
	 where, again, we used the fact that $\martingale$ is real for any situation $s'\in\situations$ such that $V_{k'}^{a,b} \sqsubseteq s'$ and $U_{k'}^{a,b} \not\sqsubseteq s'$ for some $k'\in\nats$.
	\end{itemize}

	That $\martingale^{a,b}(s)$ is non-negative, is trivially satisfied in the first case because $\martingale{}$ is non-negative.
	To see that this is also true for the third case, observe that $0 < b < \martingale(u_\ell^{s})$ and $0\leq\martingale(v_\ell^{s}) < a$ for all $\ell\in\{ 1,...,k\}$.
	This implies that $\martingale(u_\ell^{s})-\martingale(v_\ell^{s}) > b-a > 0$ for all $\ell\in\{ 1,...,k\}$ and therefore directly that $\martingale^{a,b}(s)$ is non-negative because of Equation \eqref{proof Doob: third case} and the fact that $\martingale^{a,b}(t) = \martingale(t)$ is non-negative.
	In the second case, it follows from Equations \eqref{proof Doob: first case}, \eqref{proof Doob: second case} and \eqref{proof Doob: third case} that
	\begin{align}\label{proof Doob: eq martingale differences}
	\martingale^{a,b} (s)
	=\martingale^{a,b} (v_{k}^{s}) + \martingale(s)-\martingale(v_{k}^{s}).
	\end{align}
	We prove by induction that $\martingale^{a,b} (v_{\ell}^{s}) \geq \martingale(v_{\ell}^{s})$ for all $\ell\in\{ 1,...,k\}$, and therefore, by Equation \eqref{proof Doob: eq martingale differences} and because $\martingale$ is non-negative, that $\martingale^{a,b} (s)$ is non-negative.

	If $\ell=1$, then either $v_1^s=t$ or $v_1^s \not= t$.
	If $v_1^s=t$, then $\martingale^{a,b}(v_1^s)=\martingale^{a,b}(t)=\martingale(t)=\martingale(v_1^s)$.
	If $v_1^s \not= t$, we have, by the definition of $V_1^{a,b}$, that $\martingale(v_1^s) < a$ and $a\leq\martingale(t)=\martingale^{a,b}(t)=\martingale^{a,b}(v_1^s)$.
	Hence, in both cases, we have that $\martingale^{a,b}(v_1^s) \geq \martingale(v_1^s)$.
	Now suppose that $\martingale^{a,b} (v_{\ell}^{s}) \geq \martingale(v_{\ell}^{s})$ for some $\ell\in\{ 1,...,k-1\}$.
	Then, $\martingale^{a,b} (v_{\ell+1}^{s})=\martingale^{a,b} (v_{\ell}^{s}) + [\martingale(u_{\ell}^{s})-\martingale(v_{\ell}^{s})] \geq \martingale(u_{\ell}^{s}) > b > a> \martingale(v_{\ell+1}^{s})$, which concludes our induction step.
	So indeed $\martingale^{a,b} (v_{\ell}^{s}) \geq \martingale(v_{\ell}^{s})$ for all $\ell\in\{ 1,...,k\}$.

	Next, we show that $\lupprev{s}(\martingale^{a,b}(s\andstate))\leq\martingale^{a,b}(s)$ for all $s\in\situations$, and hence, that $\martingale$ is a non-negative supermartingale.
	Consider any $s\in\situations$.
	If $s\in  [V_{k}^{a,b} , U_{k}^{a,b})$ for some $k\in\nats$, it follows from Equation~\eqref{Eq: proof Doob: def M^a,b} that
	 \begin{align*}
	 \lupprev{s}(\martingale^{a,b}(s\andstate))
	=\lupprev{s}(\martingale^{a,b}(s) + \martingale(s\andstate)-\martingale(s))
	 &\overset{\text{\ref{coherence: const add}}}{=} \lupprev{s}(\martingale(s\andstate)) + \martingale^{a,b}(s)-\martingale(s)
	 \leq \martingale^{a,b}(s),
	 \end{align*}
	 where we were allowed to use \ref{coherence: const add} because $\martingale(s) < +\infty$ as a consequence of $s\in  [V_{k}^{a,b} , U_{k}^{a,b})$ and the definitions of $V_{k}^{a,b}$ and $U_{k}^{a,b}$, and where the last step follows from $\martingale$ being a supermartingale and $0 \leq \martingale(s) < +\infty$.
	 Otherwise, if $s \not\in [V_{k}^{a,b} , U_{k}^{a,b})$ for all $k\in\nats$, we have that $\lupprev{s}(\martingale^{a,b}(s\andstate))=\lupprev{s}(\martingale^{a,b}(s))=\martingale^{a,b}(s)$, where we have used \ref{coherence: bounds} for the last inequality.
	 Hence, we have that $\lupprev{s}(\martingale^{a,b}(s\andstate))\leq\martingale^{a,b}(s)$ for all $s\in\situations$, and we can therefore infer that $\martingale^{a,b}$ is indeed a non-negative supermartingale.

	 Let us now show that $\martingale^{a,b}$ converges to $+\infty$ on all paths $\omega\in\Gamma(t)$ for which Equation \eqref{Eq: doob conv to infty} holds.
	 Consider such a path $\omega$.
	 First, it follows from $\liminf\martingale(\omega) < a$ that there exists some $n_1\in\natz$ such that $\omega^{n_1} \sqsupseteq t$ and $\martingale(\omega^{n_1}) < a$.
	 Take the first such $n_1$.
	 Then it follows from the definition of $V_1^{a,b}$ that $\omega^{n_1}\in V_1^{a,b}$.
	 Next, it follows from $\limsup\martingale (\omega) > b$ that there exists some $m_1\in\natz$ for which $m_1 > n_1$ and $\martingale(\omega^{m_1}) > b$.
	 Take the first such $m_1$.
	 Then it follows from the definition of $U_1^{a,b}$ that $\omega^{m_1}\in U_1^{a,b}$.
	 Repeating similar arguments over and over again allows us to conclude that $\omega$ goes through all the cuts $V_1^{a,b} \sqsubset U_1^{a,b} \sqsubset V_2^{a,b} \sqsubset \cdots \sqsubset V_k^{a,b} \sqsubset U_k^{a,b} \sqsubset \cdots $.
	 For all $n> n_1$, let $k_n \in \nats$ be the index such that $V_{k_n}^{a,b} \sqsubset \omega^n$ and $V_{k_n+1}^{a,b} \not\sqsubset \omega^n$.
	 Note that $k_n \to +\infty$ for $n \to +\infty$.
	 Now, if $V_{k_n}^{a,b} \sqsubset \omega^n$ and $U_{k_n}^{a,b} \not\sqsubset \omega^n$ for some $n > n_1$, then we use Equation~\eqref{proof Doob: second case} to see that $\martingale^{a,b}(\omega^n)- \martingale^{a,b}(t)$ is bounded below by $(k_{n}-1)(b-a)+\martingale(\omega^n) - a \geq (k_{n}-1)(b-a)-a$ [$\martingale$ is non-negative].
	 If on the other hand $U_{k_n}^{a,b} \sqsubset \omega^n$ and $V_{k_n+1}^{a,b} \not\sqsubset \omega^n$ for some $n > n_1$, then Equation~\eqref{proof Doob: third case} implies that $\martingale^{a,b}(\omega^n)- \martingale^{a,b}(t)$ is bounded below by $k_{n} (b-a) \geq (k_{n}-1)(b-a) -a$.
	 All together, $\martingale^{a,b}(\omega^n)- \martingale^{a,b}(t)$ is bounded below by $(k_{n}-1)(b-a)-a$ for all $n > n_1$, which implies that
	 \begin{align*}
	 \lim_{n \to +\infty} \big( \martingale^{a,b}(\omega^n)- \martingale^{a,b}(t) \big)
	 \geq \lim_{n \to +\infty} (k_{n}-1)(b-a)-a = +\infty,
	 \end{align*}
	 because $\lim_{n\to +\infty} k_n = +\infty$ and $(b-a) > 0$.
	 This also implies that $\lim_{n \to +\infty} \martingale^{a,b}(\omega^n) = +\infty$ because $\martingale^{a,b}(t)=\martingale(t)=1$.

	We now use the countable set of rational couples $K \coloneqq\{ (a,b)\in\mathbb{Q}^2 : 0<a<b \}$ to define the process $\martingale^\ast$:
	\begin{equation*}
		\martingale^\ast \coloneqq\sum_{(a,b)\in K} w^{a,b} \martingale^{a,b},
	\end{equation*}
	with coefficients $w^{a,b}>0$ that sum to $1$.
	Hence, $\martingale^\ast$ is a countable convex combination of the non-negative supermartingales $\martingale^{a,b}$.
	By Lemma~\ref{lemma:positive:countable:linear:combination}, $\martingale^\ast$ is then a non-negative supermartingale.
	It is moreover clear that $\martingale^\ast(t)=\martingale(t)=1$, implying, together with its non-negativity, that $\martingale^\ast$ is a $t$-test supermartingale.
	We now show that $\martingale^\ast$ converges in the desired way as described by the proposition.

	If $\martingale$ does not converge to an extended real number on some path $\omega\in\Gamma(t)$, then $\liminf\martingale(\omega) < \limsup\martingale(\omega)$.
	Since
	$\liminf\martingale(\omega) \geq\inf_{s\in\mathscr{X}\textsuperscript{\raisebox{1pt}{\tiny $\ast$}}}{} \martingale(s) \geq 0$, there is at least one couple $(a',b')\in K$ such that $\liminf\martingale(\omega) < a' < b' < \limsup\martingale(\omega)$, and as a consequence $\martingale^{a',b'}$ converges to  $+\infty$ on $\omega$.
	Then also $\lim w^{a',b'} \martingale^{a',b'}(\omega)=+\infty$ since $w^{a',b'} > 0$.
	For all other couples $(a,b)\in K \setminus \{(a',b')\}$, we have that $w^{a,b} \martingale^{a,b}$ is non-negative, so $\martingale^\ast$ indeed converges to $+\infty$ on $\omega$.

	Finally, we show that $\martingale^\ast$ converges in $\extreals$ on every path $\omega\in\Gamma(t)$ where $\martingale$ converges to a real number.
	Fix any such $\omega\in\Gamma(t)$.
	Then for any $\epsilon\in\posreals$, there is an $n^\ast\in\natz$ such that, for all $\ell\geq n\geq n^\ast$, $\vert \martingale(\omega^\ell)-\martingale(\omega^n) \vert \leq \epsilon$ and therefore $\martingale(\omega^\ell)-\martingale(\omega^n) \geq -\epsilon$.
	Now fix any couple of rational numbers $0<a<b$ and, for any $i\in\nats$, let $v_i^{\omega}$ and $u_i^{\omega}$ be the situations in respectively $V_i^{a,b}$ and $U_i^{a,b}$ where $\omega$ passes through [if it passes through these cuts].
	We prove that $\martingale^{a,b}(\omega^\ell)-\martingale^{a,b}(\omega^n) \geq -2\epsilon$ for any $\ell\geq n\geq n^\ast$.
	To do so, let us distinguish the following four cases:
	\begin{itemize}
	\item
	$V_1^{a,b} \not\sqsubset \omega^n$ or $U_k^{a,b} \sqsubset \omega^n \not\sqsupset V_{k+1}^{a,b}$ for some $k\in\nats$ and moreover, $V_1^{a,b} \not\sqsubset \omega^\ell$ or $U_{k'}^{a,b} \sqsubset \omega^\ell \not\sqsupset V_{k'+1}^{a,b}$ for some $k'\in\nats$.
	Using Equations~\eqref{proof Doob: first case} and~\eqref{proof Doob: third case} for both $\omega^n$ and $\omega^\ell$, we get that
	\begin{align*}
	\martingale^{a,b}(\omega^\ell)-\martingale^{a,b}(\omega^n)
	&= [\martingale^{a,b}(\omega^\ell)-\martingale^{a,b}(t)]
	-[\martingale^{a,b}(\omega^n)-\martingale^{a,b}(t)] \\
	&= \sum_{i=1}^{k'}[\martingale(u_i^{\omega}) - \martingale(v_i^{\omega})]
	- \sum_{i=1}^{k}[\martingale(u_i^{\omega}) - \martingale(v_i^{\omega})],
	\end{align*}
	where we assume $k'=0$ if $V_1^{a,b} \not\sqsubset \omega^\ell$ and $k=0$ if $V_1^{a,b} \not\sqsubset \omega^n$.
	Since $k'\geq k$ [because $n\leq\ell$ and therefore $\omega^n\sqsubseteq\omega^\ell$] and $\martingale(u_i^{\omega}) - \martingale(v_i^{\omega})> b-a > 0$ for all $i\in\nats$, we have that $\martingale^{a,b}(\omega^\ell)-\martingale^{a,b}(\omega^n) \geq 0 > -2\epsilon$ [where we also implicitly use the convention that $+\infty-\infty=+\infty$].

	\item
	$V_1^{a,b} \not\sqsubset \omega^n$ or $U_k^{a,b} \sqsubset \omega^n \not\sqsupset V_{k+1}^{a,b}$ for some $k\in\nats$ and moreover, $V_{k'}^{a,b} \sqsubset \omega^\ell \not\sqsupset U_{k'}^{a,b}$ for some $k'\in\nats$.
	Using Equations~\eqref{proof Doob: first case} and~\eqref{proof Doob: third case} for $\omega^n$ and Equation~\eqref{proof Doob: second case} for $\omega^\ell$, we find that
	\begin{align}\label{Eq: proof Doob eq 7}
	\martingale^{a,b}(\omega^\ell)-\martingale^{a,b}(\omega^n)
	&= [\martingale^{a,b}(\omega^\ell)-\martingale^{a,b}(t)]
	-[\martingale^{a,b}(\omega^n)-\martingale^{a,b}(t)] \nonumber \\
	&= \Big[\sum_{i=1}^{k'-1}[\martingale(u_i^{\omega}) - \martingale(v_i^{\omega})] + \martingale(\omega^\ell) - \martingale(v_{k'}^{\omega})\Big]
	- \sum_{i=1}^{k}[\martingale(u_i^{\omega}) - \martingale(v_i^{\omega})],
	\end{align}
	where we assume $k=0$ if $V_1^{a,b} \not\sqsubset \omega^n$.
	Note that $k' \geq k+1$, because $k' \geq k$ [since $\omega^n \sqsubseteq \omega^\ell$] and $k'=k$ is impossible.
	Indeed, if $k=0$, $k'=k$ is impossible because $k'\in\nats$.
	Otherwise, if $k>0$, $k'=k$ would imply that $U_{k'}^{a,b} = U_{k}^{a,b} \sqsubset \omega^n \sqsubseteq \omega^\ell$, contradicting the assumption that $\omega^\ell \not\sqsupset U_{k'}^{a,b}$.
	Hence, taking into account that $\martingale(u_i^{\omega}) - \martingale(v_i^{\omega})> b-a > 0$ for all $i\in\nats$, we infer from~\eqref{Eq: proof Doob eq 7} that $\martingale^{a,b}(\omega^\ell)-\martingale^{a,b}(\omega^n) \geq \martingale(\omega^\ell) - \martingale(v_{k'}^{\omega})$ [again, also using the convention that $+\infty-\infty=+\infty$].
	Finally, observe that $\omega^{n^\ast} \sqsubseteq \omega^n \sqsubseteq v_{k+1}^{\omega} \sqsubseteq v_{k'}^{\omega}\sqsubset \omega^\ell$---the situation $v_{k+1}^{\omega}$ exists because $V_{k+1}^{a,b} \sqsubseteq V_{k'}^{a,b} \sqsubset \omega^\ell$---and therefore, recalling how $n^\ast$ was chosen,
	\begin{align*}
	\martingale^{a,b}(\omega^\ell)-\martingale^{a,b}(\omega^n) \geq \martingale(\omega^\ell) - \martingale(v_{k'}^{\omega}) \geq -\epsilon \geq -2\epsilon.
	\end{align*}

	\item
	$V_k^{a,b} \sqsubset \omega^n \not\sqsupset U_k^{a,b}$ for some $k\in\nats$ and $U_{k'}^{a,b} \sqsubset \omega^\ell \not\sqsupset V_{k'+1}^{a,b}$ for some $k'\in\nats$ [we automatically have that $V_1^{a,b}\sqsubset\omega^\ell$ because $V_k^{a,b} \sqsubset \omega^n\sqsubseteq\omega^\ell$].
	Using Equation~\eqref{proof Doob: second case} for $\omega^n$ and Equation~\eqref{proof Doob: third case} for $\omega^\ell$, we get that
	\begin{align*}
	\martingale^{a,b}(\omega^\ell)-\martingale^{a,b}(\omega^n)
	&= \sum_{i=1}^{k'}[\martingale(u_i^{\omega}) - \martingale(v_i^{\omega})]
	- \Big[\sum_{i=1}^{k-1}[\martingale(u_i^{\omega}) - \martingale(v_i^{\omega})] + \martingale(\omega^n) - \martingale(v_{k}^{\omega})\Big] \\
	&= \sum_{i=1}^{k'}[\martingale(u_i^{\omega}) - \martingale(v_i^{\omega})]
	- \Big[\sum_{i=1}^{k-1}[\martingale(u_i^{\omega}) - \martingale(v_i^{\omega})] + \martingale(\omega^n)\Big] + \martingale(v_{k}^{\omega}) \\
	&= \sum_{i=1}^{k'}[\martingale(u_i^{\omega}) - \martingale(v_i^{\omega})]
	-
	\sum_{i=1}^{k-1}[\martingale(u_i^{\omega}) - \martingale(v_i^{\omega})] - \martingale(\omega^n) + \martingale(v_{k}^{\omega}),
	\end{align*}
	where the second step follows because $\martingale(v_{k}^{\omega})$ is real [as a consequence of the definition of $V_k^{a,b}$] and the third step follows because $\sum_{i=1}^{k-1}[\martingale(u_i^{\omega})- \martingale(v_i^{\omega})] \geq 0$ [since all $\martingale(u_i^{\omega}) - \martingale(v_i^{\omega})$ are positive] and $\martingale(\omega^n) \geq 0$ [because $\martingale{}$ is non-negative].
	Using the fact that $k'\geq k$ and that all $\martingale(u_i^{\omega}) - \martingale(v_i^{\omega})$ are positive, the equation above implies that
	\begin{align*}
	\martingale^{a,b}(\omega^\ell)-\martingale^{a,b}(\omega^n)
	&\geq \sum_{i=k}^{k'}[\martingale(u_i^{\omega}) - \martingale(v_i^{\omega})]
	 - \martingale(\omega^n) + \martingale(v_{k}^{\omega}) \nonumber \\
	 &= \sum_{i=k+1}^{k'}[\martingale(u_i^{\omega}) - \martingale(v_i^{\omega})] + \martingale(u_k^{\omega})
	 	 - \martingale(\omega^n)
	 \geq \martingale(u_k^{\omega})
	 - \martingale(\omega^n),
	\end{align*}
	where the equality follows from the fact that $\martingale(v_{k}^{\omega})$ is real-valued and the last inequality follows once more from the positivity of all $\martingale(u_i^{\omega}) - \martingale(v_i^{\omega})$.
	Then since $\omega^{n^\ast} \sqsubseteq \omega^n \sqsubseteq u_k^{\omega}$---the situation $u_k^{\omega}$ exists because $U_k^{a,b} \sqsubseteq U_{k'}^{a,b} \sqsubset \omega^\ell$---we infer from our assumptions about $n^\ast$ that
	\begin{align*}
	\martingale^{a,b}(\omega^\ell)-\martingale^{a,b}(\omega^n)
	 \geq \martingale(u_k^{\omega})
	 - \martingale(\omega^n)
	 \geq -\epsilon \geq -2\epsilon.
	\end{align*}

	\item
	$V_k^{a,b} \sqsubset \omega^n \not\sqsupset U_k^{a,b}$ for some $k\in\nats$ and $V_{k'}^{a,b} \sqsubset \omega^\ell \not\sqsupset U_{k'}^{a,b}$ for some $k' \in\nats$.
	Using Equation~\eqref{proof Doob: second case} for both $\omega^n$ and $\omega^\ell$, we find that
	\begin{multline*}
	\martingale^{a,b}(\omega^\ell)-\martingale^{a,b}(\omega^n) \\
	\begin{aligned}
	&= \Big[\sum_{i=1}^{k'-1}[\martingale(u_i^{\omega}) - \martingale(v_i^{\omega})]
	+ \martingale(\omega^\ell) - \martingale(v_{k'}^{\omega}) \Big]
	- \Big[\sum_{i=1}^{k-1}[\martingale(u_i^{\omega}) - \martingale(v_i^{\omega})]
	+ \martingale(\omega^n) - \martingale(v_{k}^{\omega})\Big] \\
	&= \sum_{i=1}^{k'-1}[\martingale(u_i^{\omega}) - \martingale(v_i^{\omega})]
	+ \martingale(\omega^\ell) - \martingale(v_{k'}^{\omega})
	- \sum_{i=1}^{k-1}[\martingale(u_i^{\omega}) - \martingale(v_i^{\omega})]
	- \martingale(\omega^n) + \martingale(v_{k}^{\omega})\\
	&\geq \sum_{i=k}^{k'-1}[\martingale(u_i^{\omega}) - \martingale(v_i^{\omega})]
	+ \martingale(\omega^\ell) - \martingale(v_{k'}^{\omega})
	- \martingale(\omega^n) + \martingale(v_{k}^{\omega})
	\end{aligned}
	\end{multline*}
	where the two last steps follow in a similar way as before;
	first using the real-valuedness of $\martingale{}(v_k^{\omega})$ and the non-negativity of both $\martingale{}(\omega^n)$ and $\sum_{i=1}^{k-1}[\martingale(u_i^{\omega}) - \martingale(v_i^{\omega})]$, and then using the fact that $k'\geq k$ and that all $\martingale(u_i^{\omega}) - \martingale(v_i^{\omega})$ are positive.
	If $k'=k$, and therefore $\martingale{}(v_{k'}^{\omega}) = \martingale{}(v_{k}^{\omega})\in\reals{}$, it follows from the expression above that $\martingale^{a,b}(\omega^\ell)-\martingale^{a,b}(\omega^n)\geq \martingale(\omega^\ell)-\martingale(\omega^n)\geq-\epsilon\geq-2\epsilon$.
	Otherwise, if $k'>k$, then we use the real-valuedness of $\martingale(v_k^\omega)$ to deduce from the expression above that
	\begin{align*}
	\martingale^{a,b}(\omega^\ell)-\martingale^{a,b}(\omega^n)
	&\geq \sum_{i=k+1}^{k'-1}[\martingale(u_i^{\omega}) - \martingale(v_i^{\omega})]
	+ \martingale(\omega^\ell) - \martingale(v_{k'}^{\omega})
	+ \martingale(u_k^\omega) - \martingale(\omega^n) \\
	&\geq \martingale(\omega^\ell) - \martingale(v_{k'}^{\omega})
	+ \martingale(u_k^\omega) - \martingale(\omega^n) \geq -2\epsilon,
	\end{align*}
	where the last inequality follows from our assumptions about $n^\ast$ and the fact that $\omega^{n^\ast}\sqsubseteq \omega^{n} \sqsubseteq u_k^\omega \sqsubset v_{k'}^\omega \sqsubset \omega^\ell$.
	\end{itemize}
	Hence, we conclude that for any $\epsilon\in\posreals$, there is an $n^\ast\in\natz$ such that $\martingale^{a,b}(\omega^\ell)-\martingale^{a,b}(\omega^n) \geq -2\epsilon$ for all $\ell\geq n\geq n^\ast$ and any couple of rational numbers $0<a<b$.

	To see that this implies that $\martingale^\ast$ converges to an extended real number on $\omega$, assume \emph{ex absurdo} that it does not.
	Then there is some $\epsilon\in\posreals$ such that $\liminf\martingale^\ast(\omega)<\limsup\martingale^\ast(\omega)- 2\epsilon$.
	As proved above, there is an~$n^\ast$ such that, for all $\ell\geq n\geq n^\ast$ and any couple of rational numbers $0<a<b$, $\martingale^{a,b}(\omega^\ell)-\martingale^{a,b}(\omega^n) \geq -2\epsilon$ and therefore also $\martingale^{a,b}(\omega^\ell) \geq\martingale^{a,b}(\omega^n)-2\epsilon$.
	Then it follows directly from the definition of $\martingale^\ast$ that also $\martingale^\ast(\omega^\ell) \geq\martingale^\ast(\omega^n)-2\epsilon$ for all $\ell\geq n\geq n^\ast$.
	However, this is in contradiction with $\liminf\martingale^\ast(\omega)<\limsup\martingale^\ast(\omega)-2\epsilon$, because the latter would require that there is some couple $\ell\geq n\geq n^\ast$ such that $\martingale^\ast(\omega^\ell)< \martingale^\ast(\omega^n)-2\epsilon$.
	Hence, $\martingale^\ast$ converges to an extended real number on $\omega$.
\end{proofof}
\vspace*{0.8em}
\begin{proofof}{Theorem~\ref{Theorem: Doob}}
Due to Proposition~\ref{Prop: Doob}, there is an $s$-test supermartingale $\martingale^\ast$ that converges to $+\infty$ on every path $\omega\in\Gamma(s)$ where $\martingale$ does not converge to an extended real number.
Let $B\in\reals{}$ be a lower bound for $\martingale$ and let $\martingale'\coloneqq \tfrac{1}{\martingale(s)-B+1}(\martingale-B+\martingale^\ast)$.
Since both $\martingale^\ast$ and $\martingale-B$ are supermartingales [because of \ref{coherence: const add}], Lemma~\ref{lemma:positive:countable:linear:combination} implies that $(\martingale-B+\martingale^\ast)$ is a supermartingale too and therefore, since $1\leq \martingale(s)-B+1 < +\infty$ [$\martingale(s)$ is real], \ref{coherence: homog for non-negative lambda} implies that $\martingale'$ is a supermartingale.
Moreover, $\martingale'$ is non-negative because both $\martingale-B$ and $\martingale^\ast$ are non-negative and $\martingale(s)-B+1 \geq 1$ which, together with $\martingale'(s)=1$, allows us to conclude that $\martingale'$ is an $s$-test supermartingale.
Furthermore, consider any path $\omega\in\Gamma(s)$ such that $\martingale(\omega^n)$ does not converge to a real number.
Then either it converges to $+\infty$ or it does not converge in $\extreals$ [because $\martingale$ is bounded below].
In the first case, it follows from the non-negativity of $\martingale^\ast$ and the positivity of $\tfrac{1}{\martingale(s)-B+1}$ that $\martingale'$ also converges to $+\infty$ on $\omega$.
If $\martingale(\omega^n)$ does not converge in $\extreals$, then $\martingale^\ast$ converges to $+\infty$ on $\omega$ and therefore, because $\martingale-B$ is non-negative and $\tfrac{1}{\martingale(s)-B+1}$ is positive, $\martingale'$ also converges to $+\infty$ on $\omega$.
All together, we have that $\martingale'$ is an $s$-test supermartingale that converges to $+\infty$ on every path $\omega\in\Gamma(s)$ where $\martingale$ does not converge to a real number.
\end{proofof}

\vspace*{0.8em}
\begin{proofof}{Theorem~\ref{Theorem: Levy}}
	 Since $\upprevvovkk(\cdot \vert t)$ satisfies~\ref{vovk coherence 6} for any $t\in\situations$, we have, for any $c\in\reals{}$ and any $\omega \in \samplespace{}$, that $\liminf_{n\to+\infty} \upprevvovkk(f \vert \omega^n) \geq f(\omega)$ if and only if $\liminf_{n\to+\infty} \upprevvovkk(f + c \vert \omega^n) \geq f(\omega) + c$.
	 Therefore, and because f is bounded below, we can assume without loss of generality that $f$ is a global extended real variable such that $\inf f > 0$.

	We now associate with any couple of rational numbers $0<a<b$ the following recursively constructed sequences of cuts $\{U_k^{a,b}\}_{k\in\natz}$ and $\{V_k^{a,b}\}_{k\in\nats}$.
	Let $U_0^{a,b} \coloneqq\{ s \}$ and, for $k\in\nats$,
	\begin{enumerate}
		\item
		let
			$V_k^{a,b} \coloneqq\{ v\in\situations \colon  U_{k-1}^{a,b} \sqsubset v , \ \upprevvovkk(f \vert v) < a \text{ and } (\forall t\in (U_{k-1}^{a,b} , v)) \ \upprevvovkk(f \vert t) \geq a \}$;
		\item if $V_k^{a,b}$ is non-empty, choose a positive supermartingale $\martingale_{k}^{a,b}\in\setofextsupmartb{}$ such that
		$\martingale_{k}^{a,b}(v) < a$ and $\liminf\martingale_k^{a,b}  \geq_{v} f$ for all $v\in V_k^{a,b}$, and let
			$U_k^{a,b} \coloneqq\{u\in\situations \colon V_{k}^{a,b} \sqsubset u : \martingale_{k}^{a,b}(u) > b \text{ and } (\forall t\in (V_{k}^{a,b} , u)) \ \martingale_{k}^{a,b}(t) \leq b \}$;
		\item
		if $V_k^{a,b}$ is empty, let $U_k^{a,b} \coloneqq\emptyset$.
	\end{enumerate}
	Note that all $U_k^{a,b}$ and all $V_k^{a,b}$ are indeed (partial or complete) cuts.
	We now first show that, if $V_k^{a,b}$ is non-empty, there always is a supermartingale $\martingale_{k}^{a,b}$ that satisfies the conditions above.
	We infer from the definition of the cut $V_k^{a,b}$ that
	\begin{equation*}
	\inf \bigg\{ \martingale(v) :  \martingale\in\setofextsupmartb \text{ and } \liminf\martingale \geq_v f \bigg\} = \upprevvovkk(f \vert v) < a \text{ for all } v\in V_k^{a,b}.
	 \end{equation*}
	So, for all $v\in V_k^{a,b}$, we can choose a supermartingale $\martingale_{k,v}^{a,b}$ such that $\martingale_{k,v}^{a,b}(v) < a$ and $\liminf\martingale_{k,v}^{a,b} \geq_v f$.
	Consider now the extended real process $\martingale_{k}^{a,b}$ defined, for all $ t\in\situations$, by
	\begin{align*}
		\martingale_{k}^{a,b}(t) \coloneqq
		\begin{cases}
		\martingale_{k,v}^{a,b}(t) \ &\text{ if } v \sqsubseteq t \text{ for some } v\in V_k^{a,b}; \\
		a \ &\text{ otherwise.}
		\end{cases}
	\end{align*}
	It is clear that $\martingale_{k}^{a,b}(v) < a \text{ and } \liminf\martingale_{k}^{a,b}  \geq_v f$ for all $v\in V_k^{a,b}$.
	We furthermore show that $\martingale_{k}^{a,b}$ is a positive supermartingale.

	It follows from Lemma~\ref{lemma: infima of supermartingales} that, for all $v\in V_k^{a,b}$,
	\begin{align}\label{eq: Levy martingale larger than inf f}
	\martingale_{k,v}^{a,b}(t) \geq\inf_{\omega\in\Gamma(t)} \liminf\martingale_{k,v}^{a,b}(\omega) \geq\inf_{\omega\in\Gamma(t)} f(\omega) \geq\inf f > 0 \text{ for all } t \sqsupseteq v.
	\end{align}
	Since also $a > 0$, it follows that $\martingale_{k}^{a,b}$ is positive.
	To show that $\lupprev{t}(\martingale_{k}^{a,b}(t\andstate))\leq\martingale_{k}^{a,b}(t)$ for all $t\in\situations$, fix any $t\in\situations$ and consider two cases.
	If $V_k^{a,b} \sqsubseteq t$, then $\martingale_{k}^{a,b}(t)=\martingale_{k,v}^{a,b}(t)$ and $\martingale_{k}^{a,b}(t\andstate)=\martingale_{k,v}^{a,b}(t\andstate)$ for some $v\in V_k^{a,b}$, and therefore
	$\lupprev{t}(\martingale_{k}^{a,b}(t\andstate))=\lupprev{t}(\martingale_{k,v}^{a,b}(t\andstate))\leq\martingale_{k,v}^{a,b}(t)=\martingale_{k}^{a,b}(t)$.
	If $V_k^{a,b} \not\sqsubseteq t$, then for any $x \in \statespace{}$ we either have that $tx\in V_k^{a,b}$, which implies that $\martingale_{k}^{a,b}(tx) < a$, or $V_k^{a,b} \not\sqsubseteq tx$ and therefore $\martingale_{k}^{a,b}(tx)=a$.
	Hence, we infer that $\martingale_{k}^{a,b}(t\andstate)\leq a$, and therefore, by \ref{coherence: monotonicity} and \ref{coherence: bounds}, that $\lupprev{t}(\martingale_{k}^{a,b}(t\andstate))\leq a = \martingale_{k}^{a,b}(t)$.
	So we can conclude that $\martingale_{k}^{a,b}$ is a positive supermartingale.

	Next, consider the extended real process $\mathscr{T}^{a,b}$ defined by $\mathscr{T}^{a,b}(t) \coloneqq 1$ for all $t \not\sqsupset s$, and
	\begin{align}\label{Eq: Levy: def martingale T^ab}
	 	\mathscr{T}^{a,b} (tx) \coloneqq
	 	\begin{cases}
	 		\martingale_{k}^{a,b}(tx) \mathscr{T}^{a,b}(t) / \martingale_{k}^{a,b}(t)   &\text{ if } t\in  [V_{k}^{a,b} , U_{k}^{a,b}) \text{ for some } k\in\nats;\\
	 		\mathscr{T}^{a,b}(t) &\text{ otherwise,}
	 	\end{cases}
	 \end{align}
	 for all $t \sqsupseteq s$ and all $x \in \statespace{}$.
	 We prove that this process is a positive $s$-test supermartingale that converges to $+\infty$ on all paths $\omega\in\Gamma(s)$ such that
	 \begin{equation}\label{Eq: levy conv to infty}
	 \liminf_{n\to+\infty} \upprevvovkk(f \vert \omega^n) < a < b < f(\omega).
	 \end{equation}
	 That $\mathscr{T}^{a,b}$ is well-defined follows from the fact that, for any $k\in\nats$ and any $t\in  [V_{k}^{a,b} , U_{k}^{a,b})$, $\martingale_{k}^{a,b}(t)$ is positive and moreover real because of the definition of $U_{k}^{a,b}$.
	 The process $\mathscr{T}^{a,b}$ is also positive because, for any $k\in\nats$ and any $t\in [V_{k}^{a,b}, U_{k}^{a,b})$, $\martingale_k^{a,b}(t)$ is real and positive and $\martingale_k^{a,b}(t\andstate)$ is positive, and therefore $\martingale_{k}^{a,b}(t\andstate)/ \martingale_{k}^{a,b}(t)$ is positive.
	 Furthermore, if $t\in  [V_{k}^{a,b} , U_{k}^{a,b})$ for some $k\in\nats$, we have that
	 \begin{align*}
	 \lupprev{t}(\mathscr{T}^{a,b}(t\andstate))=\lupprev{t}(\martingale_{k}^{a,b}(t\andstate) \mathscr{T}^{a,b}(t) / \martingale_{k}^{a,b}(t))
	 &\overset{\text{\ref{coherence: homog for non-negative lambda},\ref{coherence: homogeneity wrt infty}}}{=} \lupprev{t}(\martingale_{k}^{a,b}(t\andstate)) \mathscr{T}^{a,b}(t) / \martingale_{k}^{a,b}(t)
	 \leq \mathscr{T}^{a,b}(t),
	 \end{align*}
	 where the second step also uses the fact that $\martingale_k^{a,b}(t\andstate)$ and $\mathscr{T}^{a,b}(t) / \martingale_{k}^{a,b}(t)$ are positive [since, as we have shown above, $\martingale_{k}^{a,b}(t)$ is real and positive, and $\mathscr{T}^{a,b}(t)$ is positive], and where the last step uses the supermartingale character of $\martingale_{k}^{a,b}$ together with the fact that $\mathscr{T}^{a,b}(t) / \martingale_{k}^{a,b}(t)$ is positive.
	 Otherwise, if $t \not\in  [V_{k}^{a,b} , U_{k}^{a,b})$ for all $k\in\nats$, then $\lupprev{t}(\mathscr{T}^{a,b}(t\andstate))=\lupprev{t}(\mathscr{T}^{a,b}(t))=\mathscr{T}^{a,b}(t)$ because of \ref{coherence: bounds}.
	 Hence, we have that $\lupprev{t}(\mathscr{T}^{a,b}(t\andstate))\leq\mathscr{T}^{a,b}(t)$ for all $t\in\situations$, which together with the fact that $\mathscr{T}^{a,b}(s)=1$, allows us to conclude that $\mathscr{T}^{a,b}$ is indeed a positive $s$-test supermartingale.

	 Next, we show that $\mathscr{T}^{a,b}$ converges to $+\infty$ on all paths $\omega\in\Gamma(s)$ for which \eqref{Eq: levy conv to infty} holds.
	 Consider such a path $\omega$.
	 Then $\omega$ goes through all the cuts $U_0^{a,b} \sqsubset V_1^{a,b} \sqsubset U_1^{a,b} \sqsubset ... \sqsubset V_k^{a,b} \sqsubset U_k^{a,b} \sqsubset ... $.
	 Indeed, it is trivial that $\omega$ goes through $U_0^{a,b}=\{s\}$.
	 Furthermore, it follows from $\liminf_{n\to+\infty} \upprevvovkk(f \vert \omega^n) < a$ that there is an $n_1\in\nats$ such that $\omega^{n_1} \sqsupset s$ and $\upprevvovkk(f \vert \omega^{n_1}) < a$.
	 Take the first such $n_1\in\nats$.
	 Then it follows from the definition of $V_1^{a,b}$ that $\omega^{n_1}\in V_1^{a,b}$.
	 Next, it follows from $\liminf_{n\to+\infty} \martingale_1^{a,b} (\omega^n) \geq f(\omega) > b$ that there exists some $m_1\in\nats$ for which $m_1 > n_1$ and $\martingale_1^{a,b} (\omega^{m_1}) > b$.
	 Take the first such $m_1$.
	 Then it follows from the definition of $U_1^{a,b}$ that $\omega^{m_1}\in U_1^{a,b}$.
	 Repeating similar arguments over and over again allows us to conclude that $\omega$ indeed goes through all the cuts $U_0^{a,b} \sqsubset V_1^{a,b} \sqsubset U_1^{a,b} \sqsubset ... \sqsubset V_k^{a,b} \sqsubset U_k^{a,b} \sqsubset ... $.

	 In what follows, we use the following notation.
	 For any $k\in\natz$, let $u_k^\omega$ be the (necessarily unique) situation in $U_k^{a,b}$ where $\omega$ goes through.
	 Similarly, for any $k\in\nats$, let $v_k^\omega$ be the (necessarily unique) situation in $V_k^{a,b}$ where $\omega$ goes through.
	 For all $n\in\natz$, let $k_n \in \natz$ be defined by $k_n \coloneqq 0$ if $V_1^{a,b} \not\sqsubset \omega^n$ and otherwise, let $k_n$ be such that $V_{k_n}^{a,b} \sqsubset \omega^n$ and $V_{k_n+1}^{a,b} \not\sqsubset \omega^n$.
	 Note that $k_n \to +\infty$ for $n \to +\infty$ because $\omega$ goes through all the cuts $V_1^{a,b} \sqsubset V_2^{a,b} \sqsubset ... \sqsubset V_k^{a,b} \sqsubset ... $.
	 For any $n\in\natz$ such that $k_n\geq1$, we now have one of the following two cases:
	 \begin{enumerate}
	  \item
	 The first case is that $\omega^n\in (V_{k_n}^{a,b} , U_{k_n}^{a,b}]$.
	 Then by applying Equation~\eqref{Eq: Levy: def martingale T^ab} for each subsequent step and cancelling out the intermediate terms, which is possible because $\martingale_{\ell}^{a,b}(s')$ is real for any $s'\in [V_{\ell}^{a,b} , U_{\ell}^{a,b})$ and any $\ell \in \nats$ [this follows readily from the definition of the cuts $V_{\ell}^{a,b}$ and $U_{\ell}^{a,b}$], we find that
	 \begin{equation*}
	 \mathscr{T}^{a,b} (\omega^n)=\Bigg( \prod_{\ell=1}^{{k_n}-1} \frac{\martingale_{\ell}^{a,b}(u_\ell^{\omega})}{\martingale_{\ell}^{a,b}(v_{\ell}^{\omega})} \Bigg) \frac{\martingale_{k_n}^{a,b}(\omega^n)}{\martingale_{k_n}^{a,b}(v_{k_n}^{\omega})}.
	 \end{equation*}
	 Since $\martingale_{k_n}^{a,b}(\omega^n) \geq\inf f > 0$ [due to Equation \eqref{eq: Levy martingale larger than inf f}], $\martingale_{\ell}^{a,b}(u_\ell^{\omega}) > b > 0$ for all $\ell\in\{ 1,...,k_n-1\}$ and $0 < \martingale_{\ell}^{a,b}(v_\ell^{\omega}) < a$ for all $\ell\in\{ 1,...,k_n\}$, we get that
	 \begin{equation*}
	 \mathscr{T}^{a,b} (\omega^n) \geq \Big( \frac{b}{a} \Big)^{k_n-1}  \frac{\martingale_{k_n}^{a,b}(\omega^n)}{a} \geq \Big( \frac{b}{a} \Big)^{k_n-1} \Big( \frac{\inf f}{a} \Big).
	 \end{equation*}
	 \item
	 The second case is that $\omega^n \in (U_{k_n}^{a,b} , V_{k_n+1}^{a,b}]$.
	 Then, by repeatedly applying Equation~\eqref{Eq: Levy: def martingale T^ab}, we have that
	 \begin{equation*}
	 \mathscr{T}^{a,b} (\omega^n)=\prod_{\ell=1}^{k_n} \frac{\martingale_{\ell}^{a,b}(u_\ell^{\omega})}{\martingale_{\ell}^{a,b}(v_{\ell}^{\omega})}.
	 \end{equation*}
	 Since $\martingale_{\ell}^{a,b}(u_\ell^{\omega}) > b > 0$ and $0 < \martingale_{\ell}^{a,b}(v_\ell^{\omega}) < a$ for all $\ell\in\{1,...,k_n\}$, we find that
	 \begin{equation*}
	 \mathscr{T}^{a,b}(\omega^n) > \Big(\frac{b}{a}\Big)^{k_n}.
	 \end{equation*}
	\end{enumerate}
	Since $\inf f > 0$, $a>0$ and $\frac{b}{a} > 1$, and since $\lim_{n\to+\infty} k_n=+\infty$, it follows from the two expressions above that indeed $\lim_{n\to+\infty} \mathscr{T}^{a,b} (\omega^n)=+\infty$.

	To finish, we use the countable set of rational couples $K \coloneqq\{ (a,b)\in\mathbb{Q}^2 : 0<a<b \}$ to define the process $\mathscr{T}$:
	\begin{equation*}
		\mathscr{T} \coloneqq\sum_{(a,b)\in K} w^{a,b} \mathscr{T}^{a,b},
	\end{equation*}
	with coefficients $w^{a,b}>0$ that sum to $1$.
	Hence, $\mathscr{T}$ is a countable convex combination of the positive $s$-test supermartingales $\mathscr{T}^{a,b}$.
	By Lemma~\ref{lemma:positive:countable:linear:combination}, $\mathscr{T}$ is then also a supermartingale.
	It is also positive, because all $w^{a,b} \mathscr{T}^{a,b}$ are positive.
	Since it is moreover clear that $\mathscr{T}(s)=1$, the process $\mathscr{T}$ is a positive $s$-test supermartingale.
	Furthermore, $\mathscr{T}$ converges to $+\infty$ on the paths $\omega\in\Gamma(s)$ where $\liminf_{n\to+\infty} \upprevvovkk(f \vert \omega^n) < f(\omega)$.
	Indeed, consider such a path $\omega$.
	Then since $f(\omega) \geq\inf f > 0$, there is at least one couple $(a',b')\in K$ such that $\liminf_{n\to+\infty} \upprevvovkk(f \vert \omega^n) < a' < b' < f(\omega)$, and as a consequence $\lim_{n\to+\infty} \mathscr{T}^{a',b'}(\omega^n)=+\infty$.
	Then also $\lim_{n\to+\infty} w^{a',b'} \mathscr{T}^{a',b'}(\omega^n)=+\infty$, and since $w^{a,b} \mathscr{T}^{a,b}$ is positive for all other couples $(a,b)\in K \setminus (a',b')$, the positive $s$-test supermartingale $\mathscr{T}$ indeed converges to $+\infty$ on $\omega$.
	\end{proofof}
\end{document}